\documentclass[10pt]{amsart}

\usepackage[all]{xy}
\usepackage{amsmath, amssymb, hyperref}
\usepackage{tikz}

\newtheorem{thm}{Theorem}[section]
\newtheorem{prop}[thm]{Proposition}
\newtheorem{cor}[thm]{Corollary}
\newtheorem{lem}[thm]{Lemma}

\theoremstyle{remark}
\newtheorem{remark}[thm]{Remark}

\theoremstyle{definition}

\newcommand*\isom{\xrightarrow{\sim}}

\makeatletter
\renewcommand*\env@matrix[1][*\c@MaxMatrixCols c]{%
  \hskip -\arraycolsep
  \let\@ifnextchar\new@ifnextchar
  \array{#1}}
\newcommand{\pair}[1]{\langle#1\rangle}
\newcommand{\divisor}{\operatorname{div}}
\newcommand{\ord}{\operatorname{ord}}
\newcommand{\Pic}{\operatorname{Pic}}

\newcommand{\rmH}{H}
\newcommand{\rank}{\operatorname{rank}}

\newcommand{\R}{R}
\newcommand{\Ker}{\operatorname{Ker}}
\newcommand{\Lie}{\operatorname{Lie}}
\newcommand{\Spec}{\operatorname{Spec}}
\newcommand{\Sp}{\operatorname{Sp}}

\newcommand{\spann}{\operatorname{span}}
\newcommand{\V}{V}
\newcommand{\E}{E}
\newcommand{\Jac}{\operatorname{Jac}}
\newcommand{\Sing}{\operatorname{Sing}}
\newcommand{\ceeone}{c_1}

\renewcommand\Im{\operatorname{Im}}

\def\qq{\mathbb{Q}}
\def\PP{\mathbb{P}}
\def\rr{\mathbb{R}}
\def\zz{\mathbb{Z}}
\def\cc{\mathbb{C}}

\def\mm{\mathcal{M}}
\def\hh{\mathcal{H}}

\def\bb{\mathcal{B}}
\def\QQ{\boldsymbol{Q}}

\def\vv{\boldsymbol{V}}
\def\VV{\mathcal{V}}
\def\ww{{\boldsymbol{W}}}
\def\oo{\mathcal{O}}
\def\pp{\mathcal{P}}

\def\ee{\mathcal{E}}
\def\ff{\mathcal{F}}

\def\nn{\mathbb{N}}
\def\deldelbar{\partial \bar{\partial}}
\def\Ar{\mathrm{Ar}}

\def\omar{\omega_{\Ar}}
\def\opar{\oo(P)_{\Ar}}
\def\oqar{\oo(Q)_{\Ar}}

\def\omsq{\pair{\omar,\omar}}
\def\d{\mathrm{d}}
\def\Lbar{\bar{L}}
\def\Ybar{\bar{Y}}
\def\Xbar{\bar{X}}

\def\Mbar{\bar{M}}
\def\Zbar{\bar{Z}}

\def\sing{\mathrm{sing}}

\def\Sm{\mathrm{Sm}}
\def\can{\mathrm{can}}
\def\vareps{\varepsilon}
\def\eps{\epsilon}
\def\a{a}
\def\Gr{\mathrm{Gr}}

\numberwithin{equation}{section}

\begin{document}

\title[Faltings delta-invariant and semistable degeneration]{Faltings delta-invariant and semistable degeneration}

\author{Robin de Jong}

\begin{abstract} We determine the asymptotic behavior of the Arakelov metric, the Arakelov-Green's function, and the Faltings delta-invariant for arbitrary one-parameter families of complex curves with semistable degeneration. The leading terms in the asymptotics are given a combinatorial interpretation in terms of S.\ Zhang's theory of admissible Green's functions on polarized metrized graphs. 
\end{abstract}

\maketitle

\thispagestyle{empty}

\section{Introduction}

In 1984 G. Faltings \cite{fa} introduced the delta-invariant $\delta_F(C)$ of a compact and connected Riemann surface $C$. The delta-invariant appears as an archimedean contribution in the so-called Noether formula for arithmetic surfaces, which we recall briefly in a simplified form (cf.\ \cite[Section~6]{fa} and \cite{mb} for more general versions). A precise definition of the delta-invariant is given in Section~\ref{sec:metrization} below.

Let $\pi \colon \Xbar \to \Spec \zz$ be a semistable curve of positive genus over the ring of integers. Assume that the total space $\Xbar$ is regular. Let $\bar{\omega}$ denote the relative dualizing sheaf of $\pi$, equipped with the Arakelov canonical metric at the archimedean place. Let $\deg \det \R \pi_* \bar{\omega}$ denote the Faltings height of $\Xbar$, and let $(\bar{\omega},\bar{\omega})_\Ar$ denote the self-intersection, in the Arakelov sense, of $\bar{\omega}$. Then the Noether formula for $\pi$ is the identity
\begin{equation} \label{noether} 12 \deg \det \R \pi_* \bar{\omega} = (\bar{\omega},\bar{\omega})_\Ar + \sum_p \delta_p(\Xbar) \log p +  \delta_F(\Xbar(\cc)) 
\end{equation}
in $\rr$. Here the summation is over the prime numbers $p$, and the non-archimedean invariant $\delta_p(\Xbar)$ denotes the discriminant, i.e., the number of singular points in the geometric fiber, at the prime $p$ of $\zz$. 

The Noether formula suggests that the delta-invariant $\delta_F(\Xbar(\cc))$ should be interpreted as an archimedean analogue of the discriminant $\delta_p(\Xbar)$. Phrased a little differently, this means that one should think of the delta-invariant $\delta_F(C)$ of a compact connected Riemann surface $C$ as the minus logarithm of a suitable ``distance'' from the moduli point $[C]$ in the moduli space of curves $\mm_h(\cc)$ towards the boundary of $\mm_h(\cc)$ in Deligne-Mumford's compactification $\overline{\mm}_h(\cc)$, parametrizing stable curves. 

In view of this interpretation, Faltings raised in \cite{fa} the question to determine the asymptotic behavior of his delta-invariant towards the boundary in $\overline{\mm}_h(\cc)$. Around 1990 J. Jorgenson \cite{jo} and R. Wentworth \cite{we} independently were able to describe the asymptotic behavior of the delta-invariant in a one-parameter family of curves degenerating into a stable curve with precisely one node. They also answered the question as to the asymptotic behavior of the intimately related  Arakelov canonical metric and Arakelov-Green's function. Note that the results of Jorgenson and Wentworth correspond to considering holomorphic arcs in $\overline{\mm}_h(\cc)$ passing through a generic point of the boundary. This leaves the question  as to what one can say in cases where the special fiber develops more than one node, i.e., what happens near points in codimension $\geq 2$ strata of $\overline{\mm}_h(\cc)$. 

The purpose of the present paper is to answer this question. We give a full description of the asymptotic behavior of the delta-invariant, the Arakelov metric, and the Arakelov-Green's function in cases of arbitrary semistable degenerations over the disc $\Delta$. Our results yield those of Jorgenson and Wentworth as a special case, and furnish an alternative approach to them. In our setup, the leading coefficients of the asymptotics that we find are interpreted in terms of S. Zhang's theory \cite{zh} of admissible Green's functions on polarized weighted graphs, and thus come with a precise combinatorial interpretation. We recall Zhang's theory extensively in the main text below. 

Let $\pi \colon \Xbar \to \Delta$ be a semistable curve of positive genus over the unit disc $\Delta$. Assume that $\pi$ is smooth over $\Delta^*$, and write $X=\pi^{-1}\Delta^*$. To the special fiber $\Xbar_0$ one has canonically associated a polarized weighted graph $G$, by taking the dual complex of $\Xbar_0$ and remembering the arithmetic genera of the irreducible components of $\Xbar_0$, as well as the local multiplicities of the singularities of $\Xbar_0$ on the total space $\Xbar$. For two real-valued continuous functions $\phi_1$, $\phi_2$ on $\Delta^*$ we write $\phi_1 \sim \phi_2$ if the difference $\phi_1 - \phi_2$ extends as a continuous function over $\Delta$. Our result on the asymptotics of the delta-invariant is then as follows.
\begin{thm} \label{delta}  
Let $\delta$ be the volume of the polarized weighted graph $G$, and let $\vareps$ be Zhang's epsilon-invariant (\ref{defineeps}) of $G$.  Let $\Omega(t)$ be the family of normalized period matrices on $\Delta^*$ determined by a symplectic framing of $R^1\pi_*\zz_X$. Then the Faltings delta-invariant has asymptotics
\begin{equation} \label{asymptdelta} \delta_F(X_t) \sim -(\delta+\vareps) \log|t| - 6 \log \det \Im \Omega(t) 
\end{equation}
as $t \to 0$. 
\end{thm}
As said, we also consider the asymptotics of the Arakelov metric  and the related Arakelov-Green's function $g_\Ar$ in the family $\pi \colon \Xbar \to \Delta$. Denote by $\omega$ the relative dualizing sheaf of $\pi$. Then on the smooth locus $\Sm(\pi)\subset \Xbar$ of $\pi$ we can identify $\omega$ with the relative cotangent bundle. Moreover, on $X\subset \Xbar$ the line bundle $\omega$ is equipped with a canonical fiberwise Arakelov metric $\|\cdot\|_\Ar$.  Let again $G$ denote the polarized weighted dual graph of the special fiber $\Xbar_0$. Let $P$ be a section of $\pi$ with image contained in the smooth locus $\Sm(\pi)$ of $\pi$.
\begin{thm} \label{arakintro} 
 Let $g_\mu$ denote the admissible Green's function (\ref{ZhArGr}) of $G$.  Assume that $P$ specializes to the irreducible component $x \in V(G)$ of $\Xbar_0$. \\ 
 (a) Let $\d z$ be a local generating section of the relative dualizing sheaf $\omega$ around the point $P(0) \in \Sm(\pi)$. Then the asymptotics
\[ \log \| \d z(P) \|_{\Ar,t} \sim - g_\mu(x,x) \log |t| \]
hold as $t \to 0$. \\ 
(b) Assume that $\pi$ has a second section $Q$ with image contained in $\Sm(\pi)$, such that $P(0) \neq Q(0)$. Then we have the asymptotics
\[ g_{\Ar,t}(P,Q) \sim  g_\mu(x,y) \log |t| \]
as $t \to 0$. Here $x, y \in V(G)$ denote the irreducible components of $\Xbar_0$ to which $P, Q$ specialize.
\end{thm}
We make a number of remarks about Theorems \ref{delta} and \ref{arakintro}. If the total space $\Xbar$ is smooth, the volume $\delta$ of $G$ is equal to the number of singularities in the special fiber $\Xbar_0$, i.e., to the discriminant of the family $\pi$ at the origin. The epsilon-invariant $\vareps$ of $G$ is slightly more complicated to describe. For its definition and a discussion of its main properties we refer to Section \ref{sec:admissible} below. One always has $\vareps \geq 0$, and actually $\vareps>0$ unless the special fiber is smooth, or the generic fiber has genus equal to one.

Classical estimates on period matrices (e.g., via the Nilpotent Orbit Theorem) imply that there exists $c \in \zz_{\geq 0}$ such that
\[ \det \Im \Omega(t) \sim -c\log|t|  \]
as $t \to 0$. Moreover, we have $c=0$ if and only if the special fiber $\Xbar_0$ has only separating nodes, i.e., when the dual graph $G$ is a tree. In this case the $ \log \det \Im \Omega(t)$ term in (\ref{asymptdelta}) can be subsumed under the $\sim$-sign. In the case where $\Xbar_0$ has at least one non-separating node, one can replace $ \log \det \Im \Omega(t)$ by $\log(-\log|t|)$. Combined with the previous remark we deduce that for all semistable degenerations over the disc with singular special fiber $\delta_F(X_t)$ blows up towards plus infinity as $t \to 0$. This nicely complements a result by R.~Wilms \cite[Corollary~1.2]{wi} that for each genus $h>0$, the delta-invariant $\delta_F$ is bounded from below on $\mm_h(\cc)$, and reinforces the idea that $\delta_F$ should be thought of as the minus logarithm of a canonical distance to the boundary.

The results of Jorgenson \cite{jo} and Wentworth \cite{we} on the delta-invariant can be stated as follows (we note that \cite{we} also gives precise expressions for constant terms, we leave them out here). Assume that the total space $\Xbar$ of the family $\pi \colon \Xbar \to \Delta$  is smooth, and assume that the special fiber $\Xbar_0$ has precisely one node.  Then if the normalization of the special fiber $\Xbar_0$ has two connected components of genera $i$ and $h-i$ (the ``separating'' case), one has the asymptotics
\begin{equation} \label{sep}   \delta_F(X_t) \sim -\frac{4i(h-i)}{h} \log|t|        
\end{equation}
as $t \to 0$. If the normalization of $\Xbar_0$ is connected (the ``non-separating'' case), one has the asymptotics
\begin{equation} \label{nonsep} \delta_F(X_t) \sim -\frac{4h-1 }{3h }\log|t| - 6 \log(-\log|t|)  
\end{equation} 
as $t \to 0$. We claim that both (\ref{sep}) and (\ref{nonsep}) can be deduced in a straightforward manner from Theorem~\ref{delta}.

Indeed, the epsilon-invariant of a polarized weighted graph $G$ with precisely one edge can be computed explicitly, see for example \cite{mosharp}. For $G$ of genus $h$ with one edge of weight one and with one vertex we have $\vareps = (h-1)/3h$, see \cite[Corollary 4.3]{mosharp}. For $G$ of genus $h$ with one edge of weight one and with two vertices of genera $i$ and $h-i$ we have $\vareps =-1+ 4i(h-i)/h$, see \cite[Lemma 4.4]{mosharp}. These two results together with (\ref{asymptdelta}) immediately reproduce (\ref{sep}) and (\ref{nonsep}). Similarly, we re-obtain to leading order the asymptotics for the Arakelov metric and Arakelov-Green's function in the one-node case as found by Jorgenson and Wentworth  by combining Theorem \ref{arakintro} with explicit calculations of the admissible Green's function for polarized weighted graphs with one edge, as in \cite[Proposition~4.2]{mosharp} and \cite[Lemma~4.4]{mosharp}.

Let $\|\cdot\|_{Q,\Ar}$ denote the Quillen metric \cite{qu} on the line bundle $\det \pi_*\omega$ over $\Delta^*$ derived from the Arakelov metric, and let $\|\cdot\|_H$ denote the Hodge-Petersson metric derived from the $L^2$ inner product on $ \pi_*\omega$. From Deligne's functorial Riemann-Roch \cite[Th\'eor\`eme~11.4]{de} one can derive the identity
\begin{equation} \label{analytic_torsion} 12 \log \| \cdot \|_{Q,\Ar} = 12 \log \|\cdot \|_H + \delta_F  
\end{equation}
up to an additive constant depending only on $h$. This says that  $\frac{1}{12}\delta_F$ is up to a constant depending only on $h$ the Ray-Singer analytic torsion of the Arakelov metric. Let $\sigma$ denote a local generating section of $\det \pi_*\omega$ over $\Delta$. It can be seen that
\begin{equation} \label{hodge_asympt} \log \|\sigma\|_H \sim \frac{1}{2} \log \det \Im \Omega(t) 
\end{equation}
as $t \to 0$ and combining (\ref{analytic_torsion}) and (\ref{hodge_asympt}) with Theorem \ref{delta} we deduce that 
\begin{equation} \label{asympt_quillen} 12\,\log \| \sigma \|_{Q,\Ar} \sim -(\delta+\vareps) \log|t|  
\end{equation}
as $t \to 0$. 

For semistable families of curves the asymptotic behavior of the Quillen metric has so far mainly been studied in the context of the hyperbolic metric, mostly using Selberg zeta function techniques, see for instance \cite{fm} and the references therein. General asymptotic formulae for the Quillen metric associated to a given relative K\"ahler metric are obtained in the references \cite{bi} \cite{bibo} \cite{er} \cite{yo07} \cite{yo98}. These references usually assume the relative K\"ahler metric to extend in a $C^\infty$ manner over the total space of the degeneration. For example, \cite[Corollaire~2.3]{bibo} implies that in our situation, assuming $\Xbar$ smooth, for any relative K\"ahler metric on $X$ that extends in a $C^\infty$ manner over $\Xbar$, the associated Quillen metric $\|\cdot\|_Q$ on $\det \pi_* \omega$ satisfies the asymptotic
\[ 12\,\log \| \sigma \|_Q \sim -\delta \log|t|   \]
as $t \to 0$. Comparing with (\ref{asympt_quillen}) we see that  $\vareps$ measures the failure for the Arakelov metric to extend as a $C^\infty$ metric over $\Xbar$. 

Faltings has obtained in \cite{fadeg} an extension of Theorems~\ref{delta} and~\ref{arakintro} above to the case where the base $\Delta$ is replaced by an arbitrary polydisk $\Delta^n$. This allows for a more complete description of the boundary behavior of the functions $\delta_F$ resp.\ $g_\mathrm{Ar}$ and the metric $\|\cdot\|_\Ar$ near the boundary of moduli space.  One of the main points of \cite{fadeg} is that, contrary to what one might naively expect, if $h \geq 2$ no non-zero multiple of the function $\delta_F$ can be obtained as the minus logarithm of the norm of a rational section of a smooth hermitian line bundle with a metric that is good in the sense of Mumford \cite{mu} on $\overline{\mm}_h(\cc)$. 

From (\ref{asymptdelta}) we can also see this ``non-good'' behavior by observing that for all $h \geq 2$ there exist polarized graphs $G$ of genus $h$ such that the epsilon-invariant $\vareps(G)$ is not a linear form in the weights of its edges. Similarly, for $h \geq 2$ the Arakelov metric $\|\cdot \|_\Ar$ is not a good hermitian metric on the relative dualizing sheaf on the universal curve over $\overline{\mm}_h(\cc)$. 

In order to prove Theorems~\ref{delta} and~\ref{arakintro} we will use  a number of tools. One of them is the concept of a \emph{Lear extension} of a hermitian line bundle. The terminology is due to R.\ Hain and the use of Lear extensions as an effective tool to study and formulate asymptotic properties in Hodge theory and Arakelov theory is demonstrated already in the papers \cite{hrar} and \cite{hain_normal}. Our approach and set-up is very much inspired by the latter. 

Let  $\Ybar \supset Y$ be complex manifolds, such that  the complement $D=\Ybar \setminus Y$ is a normal crossings divisor on $\Ybar$, and let $(L,\|\cdot\|)$ be a continuous hermitian line bundle on $Y$. Then, roughly speaking, a Lear extension of $(L,\|\cdot\|)$ over $\Ybar$ is a line bundle $\Lbar$ over $\Ybar$ coinciding with $L$ over $Y$ such that the metric $\|\cdot\|$ on $L$ extends continuously over $\Lbar$, away from the singular locus $D^\sing$ of $D$. Note that if a Lear extension exists, it is unique up to isomorphism, as $D^\sing$ has codimension at least two in $\Ybar$. Also note that if $Y$ is a curve, then a Lear extension over $\Ybar$ is the same as a continuous extension. 

Another tool that we will use is the \emph{Deligne pairing} of two line bundles on a family of nodal curves, as in \cite{de}. Let $\Ybar \supset Y$ be smooth complex algebraic varieties. Let $\pi \colon \Xbar \to \Ybar$ be a nodal curve, assumed to be smooth over $Y$, and write $X=\pi^{-1}Y$. Let $L, M$ be line bundles on $\Xbar$. The Deligne pairing associates to $L, M$ in a functorial and bi-multiplicative way a line bundle $\pair{L,M}$ on $\Ybar$. Assuming that $L, M$ both carry given $C^\infty$ hermitian metrics over $X$, then the Deligne pairing $\pair{L,M}$ over $Y$ is endowed with a canonical $C^\infty$ hermitian metric.  The idea to use the Deligne pairing as a device to study asymptotics of Arakelov invariants is certainly not new, and goes back to \cite{de}, see also \cite{bibo} \cite{er} \cite{fm}.

Let $\omega$ denote the relative dualizing sheaf of $\pi$. Following traditional notation, we set $\kappa_1=\pair{\omega,\omega}$, which is then a line bundle on $\Ybar$. If $P \colon \Ybar \to \Xbar$ is a section of $\pi$ we put $\psi = P^*\omega$, again following traditional notation. Then as we will see, the line bundle $\psi$ can be canonically identified with the Deligne pairing $\pair{\oo(P),\omega}$. We denote by $\opar$ the line bundle $\oo(P)$ on $X$ endowed with its $C^\infty$ hermitian metric derived from the Arakelov-Green's function $g_\Ar$, and we denote by $\omar$ the restriction of $\omega$ to $X$, endowed fiberwise with the Arakelov metric $\|\cdot \|_\Ar$.

We will derive Theorems \ref{delta} and \ref{arakintro} from the following main results.
\begin{thm} \label{main_first}  Let $\Ybar \supset Y$ be smooth complex algebraic varieties, such that $D=\Ybar \setminus Y$ is a normal crossings divisor on $\Ybar$. Let $\pi \colon \Xbar \to \Ybar$ be a nodal curve of positive genus, which we assume to be smooth over $Y$. Then  the $C^\infty$ hermitian line bundle $\pair{\omar,\omar}$ on $Y$ has a Lear extension over $\Ybar$. Assume that $\pi$ has a section $P$ resp.\ two sections $P, Q$. Then also the $C^\infty$ hermitian line bundles $\pair{\opar,\omar}$ resp.\ $\pair{\opar,\oqar}$ on $Y$ have a Lear extension over $\Ybar$. 
\end{thm}
\begin{thm} \label{main_second} Assume that $\pi \colon \Xbar \to \Delta$ is a nodal curve of positive genus over the unit disc, smooth over $\Delta^*$. Let $G$ denote the polarized weighted dual graph of the special fiber $\Xbar_0$, and let $\vareps$ be Zhang's epsilon-invariant (\ref{defineeps}) of $G$. Let $g_\mu$ be Zhang's admissible Green's function (\ref{ZhArGr}) of $G$.\\
(a) The Lear extension of $\pair{\omar,\omar}$ over $\Delta$ satisfies the equality
\begin{equation} \label{L1} \qquad \overline{\pair{\omar,\omar}} = \kappa_1 - \vareps[0] \, . 
\end{equation}
(b) Assume that $\pi \colon \Xbar \to \Delta$ is equipped with a section $P$ with image contained in the smooth locus $\Sm(\pi)$, such that $P$ specializes to $x \in \V(G)$.  Then the Lear extension of $\pair{ \opar, \omega_\Ar }$ over $\Delta$ satisfies the equality
\begin{equation} \label{L2} \qquad \overline{\pair{ \opar, \omega_\Ar }} = \psi - g_\mu(x,x)[0] \, . 
\end{equation}
(c) Assume that $\pi \colon \Xbar \to \Delta$ is equipped with two sections $P, Q$ with image contained in $\Sm(\pi)$. Assume that $P, Q$ specialize to $x \in \V(G)$ resp.\ $y \in V(G)$. Then the Lear extension of $\pair{\opar,\oqar}$ over $\Delta$ satisfies the equality
\begin{equation} \label{L3} \qquad \overline{\pair{\opar,\oqar}} = \pair{\oo(P),\oo(Q)} + g_\mu(x,y)[0] \, . 
\end{equation}
\end{thm}

We briefly describe the structure of our proof of Theorems \ref{main_first} and \ref{main_second}.  Let $j \colon J \to Y$ be the jacobian fibration associated to the smooth proper curve $X \to Y$. Assume from the outset that $\pi$ has a section $P \colon \Ybar \to \Xbar$. Let $\delta_P \colon X \to J$ be the Abel-Jacobi map associated to the section $P$, i.e., the $Y$-morphism $X \to J$ that sends a point $z \in X_y$ to the isomorphism class of $P_y-z$ in $J_y=\Jac(X_y)$. Let $h>0$ denote the genus of the fibers of $\pi$ and let $\kappa \colon Y \to J$ denote the morphism that sends $y \in Y$ to the divisor class of the degree zero line bundle $(2h-2)P_y - \omega_{X_y}$ in~$J_y$.  

As jacobians are canonically self-dual, the product $J \times_Y J$ carries a canonical Poincar\'e bundle $\mathcal{P}$, together with a rigidification along the zero section. By general results \cite{bhdj} \cite{hain_normal} \cite{hdj}, the Poincar\'e bundle $\mathcal{P}$ carries a canonical $C^\infty$ hermitian metric compatible with the rigidification. We denote by $\bb$ the restriction of $\mathcal{P}$ to the diagonal in $J \times_Y J$. From these data, we obtain a canonical $C^\infty$ hermitian line bundle $\delta_P^*\bb$ on $X$ and a canonical $C^\infty$ hermitian line bundle $\kappa^*\bb$ on $Y$.

The key to our argument is the observation (see Propositions \ref{deltasq} and \ref{omega_in_terms_of_biext}) that there exist canonical isometries
\begin{equation} \label{firstisom} \pair{\delta_P^*\bb, \delta_P^*\bb} \isom \pair{\opar,\omar}^{\otimes 4h}\otimes \pair{\omar,\omar} 
\end{equation}
and
\begin{equation} \label{secondisom} \pair{\opar,\omar}^{\otimes 4h^2} \isom \pair{\delta_P^*\bb, \delta_P^*\bb} \otimes \kappa^* \bb 
\end{equation}
of $C^\infty$ hermitian line bundles over $Y$. The isometries (\ref{firstisom}) and (\ref{secondisom}) imply that $\pair{\omar,\omar}$ and $\pair{\opar,\omar}$ have a Lear extension over $\Ybar$ if and only if both $\pair{\delta_P^*\bb, \delta_P^*\bb}$ and $\kappa^*\bb$ have a Lear extension over $\Ybar$. It follows from general results \cite{lear} \cite{hain_normal} \cite{pearldiff} that $\delta_P^* \bb$ and $\kappa^*\bb$ have a Lear extension over $\Xbar$ resp.\ $\Ybar$ (in the case of $\delta_P^*\bb$, we need to assume that $\Xbar$ is smooth). Following \cite{hain_normal} they can be computed explicitly. 

The main technical issue of this paper is then to prove that the Deligne pairing $\pair{\delta_P^*\bb, \delta_P^*\bb}$ has a Lear extension over $\Ybar$. We will see that the naive expectation that $\overline{\pair{\delta_P^*\bb, \delta_P^*\bb}  } $ should equal $\pair{\overline{\delta_P^*\bb}, \overline{\delta_P^*\bb}}$ fails in general, due to a ``height jump'' of the metric on $\delta_P^*\bb$ around the singular points of the fibers of $\Xbar \to \Ybar$. The terminology goes back to  \cite[Section~14]{hain_normal}. We are able to control the height jump, and an explicit calculation shows that $\overline{\pair{\delta_P^*\bb, \delta_P^*\bb}  }$ exists, and that the defect between $\overline{\pair{\delta_P^*\bb, \delta_P^*\bb}  }$ and $\pair{\overline{\delta_P^*\bb}, \overline{\delta_P^*\bb}}$ can be directly expressed in terms of the combinatorics of the dual graph of the special fiber. 

Theorem \ref{arakintro} follows in a straightforward manner from Theorem \ref{main_second}(b) and (c). In order to prove Theorem \ref{delta} we combine Theorem \ref{main_second}(a) with Mumford's functorial Riemann-Roch isomorphism (\ref{mumfordnorm}). By doing so we avoid the use of the ``spin-bosonization formula'' \cite[p.~402]{fa} for $\delta_F$ which is crucial in \cite{jo} and \cite{we}. 

Our asymptotic analysis of the metric on $\delta_P^*\bb$ near points of codimension two builds on work by G.~Pearlstein \cite{pearldiff} \cite{pearlhiggs}, by J.\ I.\ Burgos Gil, D.~Holmes and the author \cite{bhdj}, and by D.\ Holmes and the author \cite{hdj}.  Very similar results were obtained independently by O.\ Amini, S.\ Bloch, J.\ I.\ Burgos Gil and J.\ Fres\'an \cite{abbf}. The correct general framework to carry out this type of analysis is that of the several variables $SL_2$-orbit theorem of Pearlstein \cite{pearldiff} and Kato-Nakayama-Usui \cite{knu}, as explained in  the work of T.\ Hayama and Pearlstein in \cite{hp}.

The paper is organized as follows. In Sections \ref{jacobians}--\ref{lear} we discuss the main tools that we will use. These preliminary sections mostly contain known results and may well be skipped on first reading. In Section \ref{jacobians} we recall various useful ways of describing the jacobian of a compact and connected Riemann surface. Then in Section \ref{sec:arakmetric} we recall from \cite{ar} the Arakelov-Green's function $g_\Ar$ and the associated Arakelov metric $\|\cdot\|_\Ar$. Sections \ref{prelimgraphs}--\ref{sec:admissible} treat weighted graphs, metrized graphs, harmonic analysis on metrized graphs, the admissible Green's function and the epsilon-invariant. In Section \ref{prelimsemistable} we set the necessary conventions concerning nodal, semistable and stable curves. The Deligne pairing together with its canonical metric and the Faltings delta-invariant are recalled in Sections \ref{sec:delignepairing} and \ref{sec:metrization}. In Section \ref{sec:var_mhs} we study graded-polarized variations of mixed Hodge structures and state the Nilpotent Orbit Theorem for such variations due to Pearlstein. In Section  \ref{asympresults} we discuss, using the machinery of Section \ref{sec:var_mhs}, the asymptotics of the period and Abel-Jacobi map for families of curves with nodal degeneration over polydiscs. In Section \ref{lear} we recall from \cite{hrar} \cite{hain_normal} the notion of a Lear extension.

The first new results, including the isometries (\ref{firstisom}) and (\ref{secondisom}), are then proven in Section \ref{sec:keyisom}. In Sections \ref{learI} and \ref{delignelearII} we show the existence of the Lear extensions of some relevant Deligne pairings, and compute them explicitly in the case of a nodal curve over the unit disc. In Section~\ref{mainaux} we compute the Lear extension of $\pair{\delta_P^*\bb,\delta_P^*\bb}$, which is a crucial step in our argument as discussed above. In Section~\ref{proofmainresult} we derive Theorems~\ref{main_first} and~\ref{main_second} from the results of this computation. In Section \ref{sec:admpairing} we rephrase the results of Theorem \ref{main_second}  in terms of Zhang's admissible Deligne pairing from \cite{zh}, that we adapt here to the context of a nodal curve over a disc. In Section~ \ref{sec:corollaries} we finally derive Theorems~\ref{delta} and~\ref{arakintro} from Theorem~\ref{main_second}.

\subsection*{Acknowledgements} I thank Jos\'e Burgos Gil and David Holmes for many valuable discussions. 

\section{Jacobians} \label{jacobians} 

The purpose of this preliminary section is to discuss various ways of defining the jacobian of a pure Hodge structure of type $\{(-1,0), (0,-1)\}$, as well as their interrelations. At the end of this section we discuss the Abel-Jacobi element associated to two points $P, Q$ on a compact and connected Riemann surface $C$ in each of the various disguises of $\Jac(C)$. The material in this section is entirely classical.

Let $(V,F^\bullet)$ be a pure Hodge structure of type $\{(-1,0), (0,-1)\}$. We call the set $\mathrm{Ext}^1_{\mathrm{MHS}}(\zz(0),V)$ of extensions of $\zz(0)$ by $V$ in the category of mixed Hodge structures the (intermediate) jacobian of $(V,F^\bullet)$. By \cite[Section 3.5]{ps} we have a canonical bijection $\mathrm{Ext}^1_{\mathrm{MHS}}(\zz(0),V) \isom V_\cc/(V + F^0V_\cc)$, and this gives the jacobian of $(V,F^\bullet)$ a natural structure of complex torus. 

Assume that $(V,F^\bullet,Q)$ is a polarized Hodge structure of type $\{ (-1,0), (0,-1) \}$ and rank $2h$. Given a  symplectic basis $(e_1,\ldots,e_h,f_1,\ldots,f_h)$ of $(V,Q)$, there exists a unique associated normalized basis $(w_1,\ldots,w_h)$ of $F^0V_\cc$, determined by demanding that $w_i = -\sum_{j=1}^h \Omega_{ij}e_j + f_i$ for some matrix $\Omega$. We call $\Omega$ the period matrix of $(V,F^\bullet,Q)$ with respect to $(e_1,\ldots,e_h,f_1,\ldots,f_h)$. The Riemann bilinear relations imply that $\Omega^t=\Omega$ and $\Im \Omega >0$, so that $\Omega$ lies in Siegel's upper half space $U_h$ of degree $h$. 

Assume an extension
\begin{equation} \label{anyextension} 0 \to V \to V' \to \zz(0) \to 0 
\end{equation}
of mixed Hodge structures is given, i.e. an element of $\mathrm{Ext}^1_{\mathrm{MHS}}(\zz(0),V)$. Then $V'$ has weight filtration
\[ W_\bullet \colon \quad 0 \subset W_{-1} = V \subset W_0 = V' \, . \]
We denote by $F'^\bullet$ the Hodge filtration of $V'$. Taking $F'^0(-)_\cc$ in (\ref{anyextension}) then yields the extension
\[ 0 \to F^0V_\cc \to F'^0V'_\cc \to \cc \to 0 \]
of $\cc$-vector spaces. As can be readily checked, for each $e_0 \in V'$ that lifts the canonical generator of $\zz(0)$ in (\ref{anyextension}) there exists a unique $w_0 \in F'^0V'_\cc$ such that $w_0 \in e_0 + \cc$-$\spann(e_1,\ldots,e_h)$. Given such a lift $e_0$, we let $\delta = (\delta_1,\ldots,\delta_h) \in \cc^h$ be the coordinate vector determined by the identity $w_0=e_0+\sum_{j=1}^h \delta_j e_j$. We call $\delta$ the period vector of the mixed Hodge structure $(V',F'^\bullet,W_\bullet)$ on the basis $(e_0,e_1,\ldots,e_h,f_1,\ldots,f_h)$ of the $\zz$-module $V'$. Replacing $e_0$ by an element from $e_0+V$ changes $\delta$ by an element of $\zz^h+\Omega \zz^h$. The resulting map
$\mathrm{Ext}^1_{\mathrm{MHS}}(\zz(0),V) \to \cc^h/(\zz^h+\Omega \zz^h)$ is a bijection, compatible with the canonical structure of complex torus on $\mathrm{Ext}^1_{\mathrm{MHS}}(\zz(0),V)$. 

Let $C$ be a compact and connected Riemann surface of genus $h$. Then its first homology group $V=\rmH_1(C)$ carries a canonical pure polarized Hodge structure of type $\{(-1,0), (0,-1)\}$. The polarization is given by the intersection form $Q$. The Hodge filtration on $V_\cc$ can be described as follows. Let $\hh$ denote the $\cc$-vector space of harmonic $1$-forms on $C$. The integration map $\hh \otimes V_\cc \to \cc$ yields a natural identification $V_\cc \isom \hh^*$. Let $\omega(C)\subset \hh$ denote the subspace of holomorphic $1$-forms on $C$, and $\overline{\omega(C)} \subset \hh$ the subspace of anti-holomorphic $1$-forms. We have a natural decomposition $\hh=\omega(C)\oplus\overline{\omega(C)}$ and similarly $\hh^*=\omega(C)^* \oplus \overline{\omega(C)}^*$. Then $F^0V_\cc$ is the subspace of $V_\cc$ corresponding to $\omega(C)^*$ under the isomorphism $V_\cc \isom \hh^*$. 

We write $\Jac(C)$ as a shorthand for the jacobian of $(V,F^\bullet)$. The standard pairing $\rmH_1(C) \otimes \rmH^1(C) \to \zz$ induces an isomorphism $V_\cc/F^0V_\cc \isom F^0\rmH^1(C)_\cc^*=\omega(C)^*$. Then the inclusion $V \to V_\cc/F^0V_\cc$ corresponds to the integration map $V \to \omega(C)^*$ given by $c \mapsto \int_c$. We obtain a third useful description of $\Jac(C) = V_\cc/(V+F^0V_\cc)$ as the complex torus $\omega(C)^*/V$.  

For $P, Q \in C$ we write $V(P,Q)$ as a shorthand for the relative homology group $\rmH_1(C,\{P,Q\})$, endowed with its canonical mixed Hodge structure. We obtain an extension of mixed Hodge structures
\begin{equation} \label{extensionMHS} 0 \to V \to V(P,Q) \to \zz(0) \to 0 
\end{equation}
canonically associated to the divisor $P-Q$. Here the $\zz(0)$ on the right hand side is to be identified with the reduced homology group $\tilde{\rmH}_0(\{P,Q\})$, with the class of $P-Q$ corresponding to the positive generator of $\zz$. We call the resulting element of $\Jac(C)= \mathrm{Ext}^1_{\mathrm{MHS}}(\zz(0),V)$ the Abel-Jacobi element associated to the divisor $P-Q$, denoted by $\int_Q^P$. Under the identification $\Jac(C) = \omega(C)^*/V$ the Abel-Jacobi element corresponds to the element that is usually denoted by $\int_Q^P$. If $\Omega$ is a period matrix of $C$ on a chosen symplectic basis of $V$, with normalized basis $(w_1,\ldots,w_h)$ of $F^0V_\cc=\omega(C)^*$, the Abel-Jacobi element in $\cc^h/(\zz^h+\Omega \zz^h)$ associated to the divisor $P-Q$ is given by the vector $\int_Q^P(\omega_1,\ldots,\omega_h)$, where $(\omega_1,\ldots,\omega_h)$ is the dual basis of $(w_1,\ldots,w_h)$.

\section{Arakelov-Green's function and Arakelov metric} 
\label{sec:arakmetric}

In this section we introduce the Arakelov-Green's function $g_\Ar$ of a compact and connected Riemann surface of positive genus, see (\ref{deldelbarg_Ar}) and (\ref{normalization}). Also we introduce the Arakelov metric $\|\cdot\|_\Ar$ on the line bundle of holomorphic differentials. The main references for this section are \cite{ar} and \cite{fa}.

Let $C$ be a compact and connected Riemann surface of genus $h >0$. Denote by $\omega$ the line bundle of holomorphic differentials on $C$. On $\omega(C)$ we have a natural  hermitian inner product via the prescription
\begin{equation} \label{defineinnerproduct} \pair{\eta,\eta'} = \frac{i}{2} \int_C \eta \wedge \bar{\eta}' \, . 
\end{equation}
Let $(\eta_1,\ldots,\eta_h)$ be an orthonormal basis of $\omega(C)$. The Arakelov $(1,1)$-form of $C$ is defined to be the element
\begin{equation} \label{defArakvolume} \mu_\Ar = \frac{i}{2h} \sum_{j=1}^h \eta_j \wedge \overline{\eta}_j 
\end{equation}
of $A^2(C)$. It follows from the Riemann-Roch theorem that $\mu_\Ar$ is actually a volume form. Moreover, the Arakelov $(1,1)$-form is clearly normalized such that $\int_C \mu_\Ar = 1$. 

The Arakelov-Green's function is the generalized function on $C \times C$ determined by the conditions
\begin{equation} \label{deldelbarg_Ar}
\partial \bar{\partial}_z \, g_\Ar(P,z) =\pi i \, (\mu_\Ar(z) - \delta_P(z))
\end{equation}
and
\begin{equation} \label{normalization}
\int_C g_\Ar(P,z) \,\mu_\Ar(z) = 0
\end{equation}
for all $P  \in C$. An application of Stokes's theorem shows that one has a symmetry property
\begin{equation} \label{symmetry}
g_\Ar(P,Q) = g_\Ar(Q,P)
\end{equation}
for all distinct $P$, $Q$ in $C$. We have a local expansion
\begin{equation}
g_\Ar(P,Q) = \log |t(P) - t(Q)| + O(|t(P)-t(Q)|)
\end{equation}
for all distinct $P$, $Q$ in a coordinate chart $t \colon U \isom \Delta$ of $C$. In this expansion, the $O$-term is a $C^\infty$ function. It follows that the Arakelov-Green's function $g_\Ar(P,\cdot)$ develops a logarithmic singularity of order one at $P$.

The Arakelov-Green's function $g_\Ar$ induces a natural $C^\infty$ hermitian metric $\|\cdot \|$ on the holomorphic line bundle $L=\mathcal{O}_{C \times C}(\Delta)$ on $C \times C$, where $\Delta$ is the diagonal, by putting $\log \|1\|(P,Q)=g_\Ar(P,Q)$ for distinct $P, Q$ in $C$. Here $1$ denotes the canonical generating section of $\mathcal{O}_{C \times C}(\Delta)$. By restriction to vertical or horizontal slices of $C \times C$ we obtain natural $C^\infty$ hermitian metrics on the line bundles $\oo_C(P)$ for each $P \in C$.

There exists a canonical adjunction isomorphism
\begin{equation} \label{adjunction} \oo_{C \times C}(-\Delta)|_\Delta \isom \omega \, . 
\end{equation}
By this isomorphism one obtains a canonical $C^\infty$ residual metric $\| \cdot \|_\Ar$ on $\omega$. We call $\|\cdot \|_\Ar$ the Arakelov metric, and we usually write $\omega_\Ar$ to denote the line bundle $\omega$ equipped with the Arakelov metric. A $C^\infty$ hermitian line bundle $(L,\|\cdot\|)$ on $C$ is called admissible if its first Chern form $c_1(L,\|\cdot\|)$ is a multiple of $\mu_\Ar$.  Each $\oo_C(P)$ with its metric derived from $g_\Ar$ is admissible, and so is $\omega_\Ar$, by  \cite[Section~4]{ar}.  

Let $\Omega$ be a period matrix of $C$ on a  normalised basis of differentials $(\omega_1,\ldots,\omega_h)$ (see Section \ref{jacobians}). 
Then we have
\begin{equation} \label{hodgenorm} (\Im \Omega)_{jk} =\frac{i}{2} \int_C \omega_j \wedge \overline{\omega}_k = \pair{\omega_j,\omega_k }  \, . 
\end{equation}
This gives the useful formulas
\begin{equation} \label{formulaArmeasure}
\mu_\Ar = \frac{i}{2h} \sum_{1 \leq j,k \leq h} (\Im \Omega)^{-1}_{jk} \omega_j \wedge \overline{\omega}_k  
\end{equation}
and
\begin{equation} \label{normhodge} \| \omega_1 \wedge \cdots \wedge \omega_h \|_{H}^2 = \det \Im \Omega  \, ,  
\end{equation}
see for instance \cite[Section~2]{we}. Here $\|\cdot\|_H$ is the norm on $\det \omega(C)$ induced by the inner product (\ref{defineinnerproduct}).

Let $J = \cc^h/(\zz^h + \Omega \zz^h)$ be the jacobian of $C$. The intersection form on $H_1(C)$ gives rise to a canonical isomorphism between $J$ and its dual torus $\check{J}$. It follows that $J \times J$ carries a canonical Poincar\'e bundle $\pp$ together with a rigidification at the origin. We recall that $\pp$ carries a canonical hermitian metric, compatible with the rigidification. We denote then by $\bb$ the $C^\infty$ hermitian line bundle on $J$ obtained by restricting $\pp$ to the diagonal. By a slight abuse of language, we will refer to $\bb$ as the Poincar\'e line bundle on $J$.
\begin{prop} \label{normbiext} Let $p \colon \cc^h \to J$ be the projection, and let $U \subset J$ be an analytic open subset. Let $s$ be a  holomorphic generating section of $\bb$ over $U$, and write $f = \left(p|_{p^{-1}U}\right)^*s \in \oo_{\cc^h}(p^{-1}U)$. Then for all $z \in p^{-1}U$ we have
\[ \log\|s\|(p(z)) = \log|f|(z)-2\pi\, (\Im z)^t (\Im \Omega)^{-1} (\Im z)  \, . \]
For the first Chern form $c_1(\bb)$ of $\bb$ the identity
\[ \ceeone(\bb) = \frac{-1}{\pi i} \deldelbar \, 2\pi \,(\Im z)^t (\Im \Omega)^{-1} (\Im z)   \]
of $(1,1)$-forms holds on $J$.
\end{prop}
\begin{proof} For the first formula we refer to \cite[Proposition 3.3]{hdj}. The second formula follows from the first since locally on $U$ we can write
\[ c_1(\bb)|_U = \frac{1}{\pi i} \deldelbar \log\|s\| \, , \]
and the $\deldelbar$ of $\log|f|$ vanishes.
\end{proof}
For $P \in C$ we denote by $\delta_P \colon C \to J$ the Abel-Jacobi map that sends
\begin{equation} \label{abeljacobi} Q \mapsto \int_Q^P (\omega_1,\ldots,\omega_h) \, .  
\end{equation}
\begin{cor} \label{cor:ceeone}
For all $P \in C$ the equalities 
\[  c_1(\delta_P^*\bb) =  \frac{-1}{\pi i} \deldelbar \, 2\pi(\Im \delta_P)^t (\Im \Omega)^{-1} (\Im \delta_P)  = 2h \,\mu_\Ar  
 \]
of $(1,1)$-forms hold. In particular $\delta_P^*\bb$ is admissible. 
\end{cor}
\begin{proof} The first equality is clear from Proposition \ref{normbiext}. From (\ref{abeljacobi}) we find 
\[ 2i\, \partial (\Im \delta_P)_j = \omega_j \, , \quad
-2i \, \bar{\partial} (\Im \delta_P)_k = \overline{\omega}_k \]
and hence
\[ 4\,\partial (\Im \delta_P)_j \wedge \overline{\partial} (\Im \delta_P)_k = \omega_j \wedge \overline{\omega}_k \, . \] 
This then gives
\[ \begin{split}
\frac{-1}{\pi i} \deldelbar \, 2\pi(\Im \delta_P)^t (\Im \Omega)^{-1} (\Im \delta_P) & = 4i \sum_{1 \leq j, k \leq h} (\Im \Omega)^{-1}_{jk} 
\partial (\Im \delta_P)_j \wedge \overline{\partial} (\Im \delta_P)_k
\\
& =i \sum_{1 \leq j,k \leq h} (\Im \Omega)^{-1}_{jk} \omega_j \wedge \overline{\omega}_k \, . 
\end{split} \] 
The latter form is equal to $ 2h\, \mu_\Ar$ by equation (\ref{formulaArmeasure}).
\end{proof}

\section{Graphs, weighted graphs, metrized graphs} \label{prelimgraphs}

We base our terminology on \cite[Section~2]{bf}. A finite graph $G$ consists of a finite set $E(G)$ and a non-empty finite set $V(G)$ together with a map of sets $ E(G) \to (V(G) \times V(G))/\mathfrak{S}_2$, the vertex assignment map. We call $E(G)$ the set of edges of $G$, and $V(G)$ the set of vertices of $G$. A divisor on $G$ is an element of $\rr^{V(G)}$. 
An orientation of $G$ is a lift $E(G) \to V(G) \times V(G)$ of the vertex assignment map, usually denoted by $e \mapsto (e^-,e^+)$. 

If $G$ is a finite connected graph with set of edges $\E(G)$ and set of vertices $\V(G)$, and $M$ is a set, then an $M$-labelling of $G$ is any map $\ell \colon \E(G) \to M$. When $M$ is a subset of $\rr_{>0}$, we call $(G,\ell)$ a weighted graph. Let $(G,\ell)$ be a weighted graph. The discrete Laplacian $L \colon \rr^{V(G)} \to \rr^{V(G)}$ is given as follows. Fix an orientation on $G$. For $f \in \rr^{V(G)}$ we define $d_*f \in \rr^{E(G)}$ by
\[ (d_*f)(e) = \frac{f(e^+)-f(e^-)}{\ell(e)} \, , \] 
and for $\alpha \in \rr^{E(G)}$ we define $d^*\alpha \in \rr^{V(G)}$ by
\[ (d^*\alpha)(x) = \sum_{e \in E \colon x=e^+} \alpha(e) - \sum_{e \in E \colon x=e^-} \alpha(e) \, . \]
Then for $f \in \rr^{V(G)}$ we put $L(f)=d^*d_*f \in \rr^{V(G)}$. It can be checked that $L(f)$ is independent of the choice of orientation. It is often convenient to think of $L$ as a matrix with rows and columns indexed by $V(G)$ and with entries $L(x,y)=L(\delta_x)(y)$ for $x, y \in V(G)$. Let $L^+$ be the pseudo-inverse of $L$. Then for $x, y \in \V(G)$ we define $\bar{g}(x,y) = L^+(x,y)$. By extending bilinearly we obtain a map $\bar{g} \colon \rr^{V(G)} \times \rr^{V(G)} \to \rr$.

A metrized graph $\Gamma$ is a compact connected metric space such that $\Gamma$ is a point or for each $x \in \Gamma$ there exist $n \in \zz_{>0}$ and $\eps \in \rr_{>0}$ such that $x$ has a neighborhood isometric to the set 
\[ S(n,\eps)= \{ z \in \cc \, : \, z = t e^{2 \pi i k/n} \,\, \textrm{for some} \,\, 0 \leq t < \eps \,\, \textrm{and some}\, \, k \in \zz \} \, , \]
endowed with its path metric. If $\Gamma$ is a metrized graph, not a point, then for each $x \in \Gamma$ the integer $n$ is uniquely determined, and is called the valence of $x$. An orientation of $\Gamma$ is an oriented simplicial decomposition of $\Gamma$. Let $V_0 \subset \Gamma$ be the set of points $x \in \Gamma$ with valency $\neq 2$. This is a finite set. We call any finite non-empty set $V \subset \Gamma$ containing $V_0$ a vertex set of $\Gamma$.

If $V$ is a vertex set of $\Gamma$ then $\Gamma \setminus V$ is a finite union of open intervals. The closure of a connected component of $\Gamma \setminus V$ is called an edge associated to $V$. Let $E$ be the set of edges associated to $V$. For $e \in E$ we call $e\setminus e^o$ the set of endpoints of $e$. This is a finite set consisting of either one (when $e$ is a loop) or two (when $e$ is a closed interval) elements. If $\Gamma$ is oriented, each set of endpoints is an ordered subset $\{e^-,e^+\}$ of $V$. We denote by $\ell(e)$ the length of an edge $e$, and by $\delta(\Gamma)=\sum_{e \in E} \ell(e)$ the volume of $\Gamma$.

A metrized graph $\Gamma$ with vertex set $V$ can be viewed as an electrical network by identifying each element of $V$ with a node, and each edge $e \in E$ with a wire of resistance $\ell(e)$. In particular, when $x, y$ are points of $\Gamma$, we have the effective resistance $r(x,y)$ between $x$ and $y$.  A divisor on $\Gamma$ is an element of $\rr^V$.

Let $(G,\ell \colon E(G) \to  \rr_{>0})$ be a finite connected  weighted graph. Then to $(G,\ell)$ is naturally associated a metrized graph $\Gamma$ by glueing the closed intervals $[0,\ell(e)]$, where $e$ runs through $E(G)$, according to the vertex assignment map. The metrized graph $\Gamma$ is equipped with a distinguished vertex set $V(G) \subset \Gamma$. The valence of a vertex $x \in V(G)$ is equal to its valence as a point on $\Gamma$.  It is well known, see for instance  \cite[Section~7]{hdj}, that for all $x, y \in V(G)$ the identity
\begin{equation} \label{greenandresistance} \bar{g}(x-y,x-y)=r(x,y) 
\end{equation}
holds, where $r$ denotes effective resistance on the associated metrized graph $\Gamma$. 

\section{Harmonic analysis on metrized graphs} \label{harmonic}

The following is based on \cite[Section~2]{bf}, \cite[Section~2]{cr} and \cite[Appendix]{zh}. 

Let $\Gamma$ be a metrized graph, not a point, and assume a vertex set $V$ is given, with associated edge set $E$. For $e \in E$, let $\d \,y_e$ be the usual Lebesgue measure on $e$ and $\ell(e)$ be the volume of $e$. For $x \in V$ with valence $n$, let $S(n,\eps) \isom U_x \subset \Gamma$ be a star-shaped open neighborhood of $x$ in $\Gamma$. A connected component of $U_x \setminus \{x\}$ is called an emanating direction from $x$. We denote the set of emanating directions from $x$ by $E(x)$. Each emanating direction from $x$ is naturally identified with a unit vector $v_k$ connecting the origin in $\cc$ with the root of unity $\exp(2\pi i k/n)$. For $f$ an $\rr$-valued function on $\Gamma$ which is smooth outside $V$ with respect to the measure $\d y$, we define the one-sided derivative of $f$ at $x$ in the direction $v \in E(x)$ to be 
\[ df_v(x) = \lim_{y \to 0+} \frac{ f(x+yv)-f(x) }{y} \, , \]
if the limit exists. We denote by $C(\Gamma)$ the set of 
$\rr$-valued continuous functions on $\Gamma$ that are smooth outside $V$ and for which $df_v(x)$ exists for all $x \in V$ and all emanating directions $v \in E(x)$. We denote by $C(\Gamma)^*$ the set of linear functionals on $C(\Gamma)$. We call an element of $C(\Gamma)^*$ a current on $\Gamma$. For example, each $x \in V$ gives rise to a Dirac measure $\delta_x \in C(\Gamma)^*$ given by sending $f \mapsto f(x)$. Integration of currents over $\Gamma$ gives a natural linear map $C(\Gamma)^* \to \rr$.

The Laplacian $\Delta \colon C(\Gamma) \to C(\Gamma)^*$ is the linear operator given by
\[ \Delta(f) = \Delta^{\mathrm{dis}}(f)+ \Delta^{\mathrm{cts}}(f) \, , \quad \Delta^{\mathrm{dis}}(f)=-\delta(f) \, , \quad \Delta^{\mathrm{cts}}(f) = -f''(y) \, \d \, y \, , \]
where the map $\delta \colon C(\Gamma) \to C(\Gamma)^*$ is defined via
\[ \delta(f) = \sum_{x \in V} \sum_{v \in E(x) } df_v(x) \, \delta_x    \]
for all $f \in C(\Gamma)$.

We have the following connection with the usual discrete Laplacian $L \colon \rr^V \to \rr^V$ associated to the vertex set $V$. We obtain a linear map $L_V \colon C(\Gamma) \to \rr^V$ by putting $L_V(f)=L(f|_V)$ for $f \in C(\Gamma)$. Let $\psi \colon \rr^V \to C(\Gamma)^*$ be the canonical map sending a divisor $c = \sum_{x\in V} c(x) x$ to the corresponding Dirac measure $\psi(c)=\sum_{x \in V} c(x) \delta_x$. In particular we have $\psi(c)(g)=\sum_{x \in V}c(x) g(x)$ for all $g \in C(\Gamma)$ and $c \in \rr^V$. 
Assume that $f \in C(\Gamma)$ is quadratic on each edge $e \in E$, and suppose that $f''=a(e)/\ell(e)$ on $e \in E$. 
Let $\nu_f$ denote the discrete measure with support in $V$ that assigns mass $\nu_f(x)=-\frac{1}{2}\sum_{e \in E(x)} a(e)$ to $x \in V$. Then we have the equality of currents
\begin{equation} \label{explicitdeltaf} 
-\delta(f) = \psi (L_V(f)) - \nu_f 
\end{equation}
in $C(\Gamma)^*$. Indeed, if $x$ is endpoint $e^+$ of $e$ and $e$ lies in the emanating direction $v$ from $x$ then a calculation shows that
\[ df_v(x) = \lim_{y \to 0+} \frac{ f(x+yv)-f(x) }{y}  = \frac{f(e^-)-f(e^+)}{\ell(e)} - \frac{1}{2}a(e) \, . 
\]
 
Let $K_\can$ be the canonical divisor on $\Gamma$ given by $K_\can(x)= v(x)-2$ for each $x \in V$. Let $r(e)$ denote the effective resistance between the endpoints of $e$ in $\Gamma \setminus e^o$. This is set to be $\infty$ if $\Gamma \setminus e^o$ is disconnected. Then we define the currents 
\[ \mu_{\can}^{\mathrm{dis}} = -\frac{1}{2}\delta_{K_\can} \, , \quad
\mu_{\can}^{\mathrm{cts}} = \sum_{e \in E} \frac{\d \, y_e}{\ell(e) + r(e)} \, , \quad
\mu_{\can} = \mu_{\can}^{\mathrm{dis}} + \mu_{\can}^{\mathrm{cts}}  \]
in $C(\Gamma)^*$.
We write 
\begin{equation} \label{defFoster}
F(e)=\int_e \mu_{\can}^{\mathrm{cts}}=\frac{\ell(e)}{\ell(e)+r(e)}  
\end{equation}
for each $e \in E$. We have borrowed this notation from \cite{bf}, where $F(e)$ is called the ``Foster coefficient''.
From \cite[Theorem 2.11]{cr} we deduce that $\int_\Gamma \mu_\can = 1$. 

For $x \in \Gamma$ we have that $r(x,\textit{-})$ is piecewise quadratic, hence $r(x,\textit{-}) \in C(\Gamma)$.  It follows from the proof of \cite[Theorem 2.11]{cr} (see in particular Lemma 2.14 of loc. cit.) that
\begin{equation} \label{laplaceresistance}
\frac{1}{2} \Delta_y^{\mathrm{dis}} r(x,y) = \mu_{\can}^{\mathrm{dis}} - \delta_x(y) \quad \textrm{and} \quad \frac{1}{2} \Delta_y^{\mathrm{cts}} r(x,y) = \mu_{\can}^{\mathrm{cts}} \, , 
\end{equation}
so that
\begin{equation} \label{chinburgrumely} \frac{1}{2} \Delta_y r(x,y) = \mu_{\can}(y) - \delta_x(y)  \end{equation}
for all $x \in \Gamma$. 

The following identities will turn out to be useful later. Let $t \colon e \isom [0,\ell(e)]$ denote a parametrization of $e$. Then we have the explicit formula
\begin{equation} \label{explicitreff} r(x,y)=r(x,e^-)\frac{t(y)}{\ell(e)} + r(x,e^+)\frac{\ell(e) - t(y)}{\ell(e)} + F(e)\frac{t(y)(\ell(e)-t(y))}{\ell(e)}   
\end{equation}
for $y \in e$. In particular we find that 
\begin{equation} \label{useful}
dr_v(x,e^+)= \frac{r(x,e^-)-r(x,e^+)}{\ell(e)} + F(e) \, , 
\end{equation}
where $v$ is the emanating direction from $e^+$ corresponding to $e$. 

Finally, let $\nu$ be the measure supported on $V$ with mass $\nu(v)=\sum_{e \in E(v)} F(e)$ at the vertex $v$. Combining (\ref{explicitdeltaf}), (\ref{laplaceresistance}) and (\ref{useful}) we find the identity
\begin{equation} \label{discretelaplaceresistance}
\psi (L_V(r(x,y))) =  2\, (\mu_\can^{\mathrm{dis}}(y)-\delta_x(y))+\nu(y) 
\end{equation}
in $C(\Gamma)^*$, for all $x \in \Gamma$.

\section{Admissible Green's function and epsilon-invariant} \label{sec:admissible}

In this section we recall from \cite{zh} the definition of admissible measure on a polarized metrized graph, and its associated admissible Green's function. Also we introduce the epsilon-invariant, and prove a couple of identities for polarized metrized graphs that we will need later. 

Let $\Gamma$ be a metrized graph, with a given vertex set $V$ and associated edge set $E$. Let $h$ be a positive integer and let $K \in \zz^V$ be an integer-valued divisor of degree $2h-2$ on $\Gamma$. We define Zhang's admissible measure with respect to $K$ to be the element
\begin{equation} \label{measure} \mu = \frac{1}{2h} \left( \delta_K + 2 \, \mu_{\can} \right) 
\end{equation}
of $C(\Gamma)^*$. It is clear that $\int_\Gamma \mu = 1$. We call $h$ the genus of $(\Gamma,K)$.

Let $x \in \Gamma$. Zhang's admissible Green's function relative to $K$ is the element of $C(\Gamma)$  determined by the conditions 
\begin{equation} \label{ZhArGr} \Delta_y \,g_\mu(x,y) = \delta_x(y)-\mu(y) \, , \quad \int_\Gamma g_\mu(x,y) \, \mu(y) = 0  \, . 
\end{equation}
Note the similarity with the definition of the usual Arakelov-Green's function in (\ref{deldelbarg_Ar}) and (\ref{normalization}).
We refer to \cite[Section 3 and Appendix]{zh} for a proof that $g_\mu(x,\textit{-})$ exists in $C(\Gamma)$ for all $x \in \Gamma$, and a discussion of some of its main properties. For example, let $r(x,y)$ be the effective resistance between $x, y \in \Gamma$, then we have the identity 
\begin{equation} \label{gZhandRes}
 r(x,y) = g_\mu(x,x)-2g_\mu(x,y) + g_\mu(y,y) 
\end{equation}
for all $x, y \in \Gamma$. See \cite[Equation (3.5.1)]{zh}.

\begin{prop} \label{defc} There exists a real number $c(\Gamma,K)$ such that
\[  c(\Gamma,K) + g_\mu(x,x) + g_\mu(K,x) = 0 \]
for all $x \in \Gamma$.
\end{prop}
\begin{proof} This is the content of \cite[Theorem~3.2]{zh}.
\end{proof}

We let the epsilon-invariant of $(\Gamma,K)$ be the real number
\begin{equation} \label{defineeps} \varepsilon(\Gamma,K) = \int_\Gamma g_\mu(y,y)((2h-2)\mu(y)+\delta_K(y)) \, . 
\end{equation}
By  \cite[Theorem~4.4]{zh} one always has $\vareps(\Gamma,K) \geq 0$, and actually $\vareps(\Gamma,K)>0$ unless $\Gamma$ is a point, or the genus of $(\Gamma,K)$ is equal to one.
\begin{prop} \label{epsalt}
We have 
\[ 
 \varepsilon(\Gamma,K) = 4(h-1)(g_\mu(x,x)+g_\mu(K,x)) - g_\mu(K,K)  
\]
for all $x \in \Gamma$.
\end{prop}
\begin{proof} From Proposition \ref{defc} we find
\[ \begin{split} \varepsilon(\Gamma,K) & = 
-\int_\Gamma (c(\Gamma,K)+g(K,y))((2h-2)\mu(y) + \delta_K(y)) \\
& = -4(h-1)c(\Gamma,K) - g_\mu(K,K) \, . \end{split} \]
Using Proposition \ref{defc} again to eliminate $c(\Gamma,K)$ we find the required identity.
\end{proof}

Let $x \in \Gamma$. We put
\[ \tau(\Gamma) = \frac{1}{2} \int_\Gamma r(x,y)\, \mu_{\can} (y) \, . \]
By \cite[Lemma 2.16]{cr} this is independent of the choice of $x \in \Gamma$. 

\begin{prop} \label{taualternative}
For each $x \in \Gamma$ we have
\[ 2 \,r(x,K)  +4\,\tau(\Gamma) = 4h \,g_\mu(x,x)+\vareps(\Gamma,K)  \, . \]
\end{prop}
\begin{proof} By \cite[Lemma 4.1]{mosharp} we have
\begin{equation} \label{mori} \varepsilon(\Gamma,K) = 2h\, g_\mu(x,K) + r(x,K) 
\end{equation}
for all $x \in \Gamma$. 
Combining Proposition \ref{defc} and equation (\ref{mori}) we find
\[ \begin{split} 2\,r(x,K)  & = 
2 \,\vareps(\Gamma,K)-4h \, g_\mu(x,K) \\
& = 4h \,g_\mu(x,x)+2\,\vareps(\Gamma,K) + 4h\,c(\Gamma,K)  \end{split} \]
for all $x \in \Gamma$.
We are done once we prove that $4h\,c(\Gamma,K)=-\vareps(\Gamma,K)-4\,\tau(\Gamma)$.
Proposition \ref{defc} implies
\[ c(\Gamma,K) = - \int_\Gamma g_\mu(x,x)\,\mu(x) \, . \]
From (\ref{gZhandRes}) we then readily obtain
\[ c(\Gamma,K) = -\frac{1}{2} \int_{\Gamma \times \Gamma} r(x,y)\,\mu(x)\,\mu(y) \, . \]
Equation (\ref{mori}) gives
\[ \varepsilon(\Gamma,K) = \int_{\Gamma \times \Gamma} r(x,y) \,\delta_K(y)\, \mu(x) \, . \]
We thus find
\[ \begin{split} 
\varepsilon(\Gamma,K) + 4h \,c(\Gamma,K) & = \int_{\Gamma \times \Gamma} r(x,y)\, \mu(x) \, (\delta_K(y) -2h\,\mu(y)) \\
& = -2 \int_{\Gamma \times \Gamma} r(x,y) \, \mu(x) \, \mu_{\can}(y) = -4\,\tau(\Gamma) \, . 
\end{split}  \]
This proves the proposition.
\end{proof}
For each $e \in E$ and $x \in \Gamma$ we put
\[ \tau^{\mathrm{cts}}(\Gamma,x,e) = \frac{1}{2}  \int_e r(x,y) \,\mu_{\can}^{\mathrm{cts}}(y)  \, . \]
Next we put for $x \in \Gamma$
\[ \tau^{\mathrm{cts}}(\Gamma,x) = \sum_{e \in E} \tau^{\mathrm{cts}}(\Gamma,x,e)=\frac{1}{2} \int_\Gamma r(x,y)\, \mu_{\can}^{\mathrm{cts}}(y) \, . \]
Furthermore we let
\begin{equation} \label{eta} \eta(\Gamma,e) = \frac{1}{3}F(e)^2\ell(e) \, , \quad \eta(\Gamma)= \sum_{e \in E} \eta(\Gamma,e) = \frac{1}{3} \sum_{e \in E} F(e)^2\ell(e) \, ,
\end{equation}
with $F(e)$ the Foster coefficient of $e$ as in Section \ref{harmonic}.
Recall from Section \ref{harmonic} that we denote by $\nu$ the discrete measure supported on $V$ with mass $\nu(v)=\sum_{e \in E(v)} F(e)$ at the vertex $v$.
\begin{prop} \label{tauandeta} Let $x \in \Gamma$. Then the identity
\[ 4 \,\tau^{\mathrm{cts}}(\Gamma,x) = \eta(\Gamma) + \int_\Gamma r(x,y) \,\nu (y)   \]
holds.
\end{prop}
\begin{proof} Let $e \in E$. Recall that $\mu_\can^{\mathrm{cts}}(y)=F(e) \,\d \,y_e/\ell(e)$ on $e$. Using (\ref{explicitreff}), putting $u(y)=t(y)/\ell(e)$  we compute 
\begin{equation} \label{onlye} \begin{split}
4\,\tau^{\mathrm{cts}}(\Gamma,x,e) & = 2 \int_e r(x,y) \,\mu_\can^{\mathrm{cts}}(y) \\  & = 2\,F(e)\int_0^1\left( r(x,e^-)u + r(x,e^+)(1 - u) + F(e)\ell(e)u(1-u) \right) \, \d u \\
& =    F(e) \cdot \left(r(x,e^-)+r(x,e^+) + \frac{F(e)\ell(e)}{3} \right) \, .  
\end{split}
\end{equation}
It follows that
\[ 4\,\tau^{\mathrm{cts}}(\Gamma,x) = \eta(\Gamma) + \sum_{e \in E} F(e) \cdot (r(x,e^-)+r(x,e^+)) \, .  \]
Observing that
\[ \sum_{e \in E} F(e) \cdot (r(x,e^-)+r(x,e^+)) = 
\sum_{y \in V} r(x,y) \sum_{e \in E(y)} F(e) = \int_\Gamma r(x,y)\, \nu (y) \]
we obtain the proposition.
\end{proof}

\section{Nodal curves, semistable curves, stable curves} \label{prelimsemistable} 

Our basic reference here is \cite[Chapter~X]{acg}. In this paper, a complex algebraic curve $C$ is said to be nodal if $C$ is connected, reduced, projective, and each singular point of $C$ is an ordinary double point.  A nodal curve $C$ is called semistable (stable) if any $\PP^1$ contained in $C$ meets the rest of $C$ in at least two (three) points. For $C$ a nodal curve, the dual graph $G$ associated to $C$ is the finite connected graph whose vertex set $V(G)$ consists of the irreducible components of $C$, and whose edge set $E(G)$ consists of the singular points of $C$, with vertex assignment map given by sending a singular point $e \in E(G)$ to the set of irreducible components of $C$ that contain $e$.

A map of complex algebraic varieties $\pi \colon \Xbar \to \Ybar$ is called a nodal (semistable, stable) curve if $\pi$ is proper, flat, and all fibers of $\pi$ are nodal (semistable, stable) curves. For a nodal curve $\pi \colon \Xbar \to \Ybar$ we write $\mathrm{Sm}(\pi)$ for the locus where $\pi$ is smooth, and $\mathrm{Sing}(\pi)$ for the locus where $\pi$ is not smooth. The singular locus $\mathrm{Sing}(\pi)$ of $\pi$ is a closed subvariety of $\Xbar$ which is finite unramified over $\Ybar$. We will often work with nodal curves $\Xbar \to \Delta^n$ over the polydisc $\Delta^n$ of dimension $n$. By definition this is to be the base change, in the category of analytic spaces, along a map $f \colon \Delta^n \to \Ybar$, of a nodal curve over a complex algebraic variety $\Ybar$.

Let $\pi \colon \Xbar \to \Ybar$ be a nodal curve, and let $y \in \Ybar$ be a point. Then following \cite[Section~2]{ho} we associate to $y$ a canonical labelled graph $(G_y,\ell_y)$ as follows. The underlying graph $G_y$ is the dual graph of the nodal curve $\Xbar_y$. Let $\oo_{\Ybar,y}$ be the local ring for the holomorphic structure sheaf of $\Ybar$ at $y$. The value set of $\ell_y$ is then $\oo_{\Ybar,y} /\oo_{\Ybar,y}^*$, and the labelling $\ell_y \colon \E(G_y) \to \oo_{\Ybar,y}/\oo_{\Ybar,y}^*$ is given by the following procedure. 

Let $e \in \E(G_y)$ be a singular point of $\Xbar_y$. The local structure of nodal curves in the analytic category (cf. \cite[Proposition X.2.1]{acg}) yields that $e$ on $\Xbar$ has a neighborhood which is isomorphic, as an analytic space over $\Ybar$, to a neighborhood of $(0,y)$ in the analytic subspace of $\cc^2 \times \Ybar$ with equation $xy=a$, for some $a\in \oo_{\Ybar,y}$. The function $a$ is well-defined up to units in $\oo_{\Ybar,y}$. We set $\ell_y(e)$ to be the class of $a$ in $\oo_{\Ybar,y}/\oo_{\Ybar,y}^*$. If $\oo_{\Ybar,y}$ is a discrete valuation ring, and $a \neq 0$, then we usually think of $\ell_y(e) \in \zz_{\geq 0}$ as the normalized discrete valuation of $a$. In particular, the graph $G_y$ is then canonically weighted.

As is almost clear from the definitions, the canonical labelled graph is functorial with respect to base change.
In particular we can speak about the canonical labelled graphs associated to the fibers of a nodal curve $\Xbar \to \Delta^n$ over a polydisc. When $\pi \colon \Xbar \to \Delta$ is a nodal curve over the unit disc, the canonical weighted graph behaves well with respect to minimal desingularization of $\Xbar$, as is expressed by the next proposition.
\begin{prop} \label{minimaldesing} Let $\Xbar \to \Delta$ be a nodal curve, assumed to be smooth over $\Delta^*$. Let $\tilde{\Xbar} \to \Xbar$ be the minimal desingularization of $\Xbar$. Let $G$, $\tilde{G}$ be the canonically weighted dual graphs of the central fibers $\Xbar_0$ resp.\ $\tilde{\Xbar}_0$. Then the associated metrized graphs of $G$ and $\tilde{G}$ are canonically isometric.
\end{prop}
\begin{proof} Let $t$ be the coordinate on the unit disc. Let $e \in E(G)$ be a singular point of $\Xbar_0$. A local equation of $\Xbar$ around $e$ is then $xy-t^n$ for some $n \in \zz_{>0}$, and for instance the proof of \cite[Lemma~1.12]{dm} shows that the local minimal desingularization of $\Xbar$ at $e$ replaces $e$ by a chain of $n-1$ projective lines. That is, upon passing from $G$ to $\tilde{G}$, the edge $e$ of length $n$ in the metrized graph associated to $G$ is subdivided into $n$ segments of unit length.
\end{proof}

\section{Deligne pairing} \label{sec:delignepairing}

Let $\Xbar$ and $\Ybar$ be complex algebraic varieties and let $\pi \colon \Xbar \to \Ybar$ be a nodal curve. Then following \cite[Section XIII.5]{acg} and \cite{de} we have a canonical bi-multiplicative pairing for line bundles $L, M$ on $\Xbar$, resulting in a line bundle $\pair{L,M}$ on $\Ybar$. 
Locally on an open set $U \subset \Ybar$ the line bundle $\pair{L,M}$ is generated by symbols $\pair{\ell,m}$ with $\ell$ a nonzero rational section of $L$ on $\pi^{-1}U$ and $m$ a nonzero rational section of $M$ on $\pi^{-1}U$, such that the divisors of $\ell, m$ on $\Xbar$ have disjoint support. We have the relations
\begin{equation} \label{relations} \pair{\ell, fm } = f(\divisor \ell)\pair{\ell,m} \, , \quad
\pair{f\ell,m} = f(\divisor m) \pair{\ell, m}  
\end{equation}
for rational functions $f$ on $\Xbar$. Here, the function $f(\divisor \ell)$ should be interpreted as coming from a norm: when $D$ is an effective relative Cartier divisor on $\Xbar$, then we put $f(D)=\mathrm{Nm}_{D/\Ybar}(f)$. The Weil reciprocity law $f(\divisor g)=g(\divisor f)$ shows that this construction by generators and relations indeed gives a line bundle on $\Ybar$. 

The formation of the Deligne pairing of two line bundles $L, M$ on $\Xbar$ is compatible with base change. Assume that $P$ is a section of $\pi$ with image contained in the smooth locus $\Sm(\pi)$ of $\pi$, and let $N$ be a line bundle on $\Ybar$. Then we have canonical isomorphisms
\begin{equation} \label{canisoms} \pair{M,\pi^*N} \isom N^{\otimes \deg M} \, , \quad \pair{L,M} \isom \pair{M,L} \, , \quad \pair{\oo(P),M} \isom P^* M 
\end{equation}
of line bundles on $\Ybar$. Let $\omega$ be the relative dualizing sheaf of $\pi$. Then $\omega$  is a line bundle on $\Xbar$ by \cite[Section~1]{dm}, and we have a canonical adjunction isomorphism
\begin{equation} \label{canonicaladj} \pair{\oo(P),\oo(P)} \isom \pair{\oo(P),\omega}^{\otimes -1} \, . 
\end{equation}

Assume now that the family $\pi \colon \Xbar \to \Ybar$ is a semistable curve of genus $h>0$. Then $\Ybar$ allows a canonical classiyfing map $J \colon \Ybar \to \overline{\mm}_h$ to the moduli stack of stable curves of genus~$h$. (If $h=1$, we assume that a zero-section of the family is given so that we have a classifying map to $\overline{\mm}_{1,1}$). The discriminant line bundle $\delta$ on $\Ybar$  is defined to be the pullback, along $J$, of the line bundle associated to the boundary divisor of $\mm_h$ in $\overline{\mm}_h$. Mumford's functorial Riemann-Roch \cite[Theorem 5.10]{mu} \cite[Th\'eor\`eme~2.1]{mb} gives an isomorphism
\begin{equation} \label{mumfordnorm} \mu \colon (\det \pi_* \omega)^{\otimes 12} \isom \pair{\omega,\omega} \otimes \delta 
\end{equation}
of line bundles on $\Ybar$, canonical up to a sign. We refer to $\mu$ as the Mumford isomorphism. 

\section{Metrization of the Deligne pairing, and the delta-invariant} \label{sec:metrization}

Let $Y$ be a smooth complex algebraic variety  and consider a smooth proper curve  $\pi \colon X \to Y$ over $Y$ of positive genus. 
Let $L, M$ be line bundles on $X$, both equipped with $C^\infty$ hermitian metrics. Then by \cite[Section 6]{de} the Deligne pairing $\pair{L,M}$ has a canonical structure of $C^\infty$ hermitian line bundle, which can be given explicitly as follows. Let $\ell$, $m$ be non-zero rational sections of $L$, $M$ resp.\ with disjoint support. Then we have, as functions on $Y$
\begin{equation} \label{defmetricpairing}
\log \| \pair{\ell,m} \|  = (\log\|m\|) [\divisor \ell] + \int_\pi \log \|\ell\| \ceeone(M)
\, . 
\end{equation}
We have a symmetry relation $\|\pair{\ell,m}\|=\|\pair{m,\ell}\|$, cf. \cite[Section~6.3]{de}. It is clear from (\ref{defmetricpairing}) that for rational functions $f$ on $\Xbar$ we have
\[ \log\|\pair{\ell,fm}\| = \log \|\pair{\ell,m}\| + (\log|f|)[\divisor \ell] \, ,  \]
showing that the Deligne norm $\|\cdot \|$ is compatible with the relations (\ref{relations}). Finally, the canonical isomorphisms
\begin{equation} \label{delignenorm} \pair{M,\pi^*N} \isom N^{\otimes \deg M} \, , \quad \pair{L,M} \isom \pair{M,L}  
\end{equation}
from (\ref{canisoms}) are isometries, when $L, M$ and $N$ are equipped with $C^\infty$ hermitian metrics.

Let $P \colon Y \to X$ be a section of $\pi$, and let $\omega$ denote the relative dualizing sheaf of~$\pi$. We denote by $\opar$ resp.\ $\omar$ the $C^\infty$ hermitian line bundles on $X$ that we obtain by putting fiberwise the Arakelov metric on $\oo(P)$ resp.\ $\omega$. We recall that both $\opar$ and $\omar$ are fiberwise admissible. Each of the Deligne pairings $\pair{\oo(P),\omega}$, $\pair{\omega,\omega}$ and $\pair{\oo(P),\oo(P)}$ now carries an induced structure of $C^\infty$ hermitian line bundle on $Y$. If there is a second section $Q$ of $X \to Y$, then so does the Deligne pairing $\pair{\oo(P),\oo(Q)}$. We denote these $C^\infty$ hermitian line bundles on $Y$ respectively by
\begin{equation} \label{metrics} \pair{\opar,\omar} \, , \,\, \pair{\omar,\omar} \, , \,\, \pair{\opar,\opar} \,\, \textrm{and} \,\,  \pair{\opar,\oo(Q)_\Ar} \, . 
\end{equation} 
Now let $M$ be a fiberwise admissible $C^\infty$ hermitian line bundle on $X$. Denote by $1_P$ the canonical rational section of $\oo(P)$. Then for any nonzero rational section $m$ of $M$ with support away from $P$ we have
\[ \log \|P^*m\| = (\log \|m\|)[\divisor 1_P] \, . \]
From (\ref{defmetricpairing})  we obtain on the other hand that
\[ \log \| \pair{1_P,m} \| = (\log \|m\|)[\divisor 1_P] + \int_\pi \log \|1_P\| \ceeone(M) = (\log \|m\|)[\divisor 1_P]  \]
by the normalization condition (\ref{normalization}).
It follows that the canonical isomorphism $\pair{\opar,M} \isom P^* M $ from (\ref{canisoms}) is an isometry. For example, since $\omar$ is fiberwise admissible we find that the isomorphism $\pair{\opar,\omega_\Ar} \isom P^* \omar$ is an isometry. Further, by definition of the Arakelov metric, the adjunction isomorphism (\ref{canonicaladj}) is an isometry. We conclude that we have a canonical isometry
\begin{equation} \label{adj}
\pair{\opar,\opar} \isom \pair{\opar,\omar}^{\otimes -1}  
\end{equation}
of $C^\infty$ hermitian line bundles on $Y$.

Mumford's functorial Riemann-Roch (\ref{mumfordnorm}) gives a canonical isomorphism, up to sign
\[ \mu_Y \colon (\det \pi_*\omega)^{\otimes 12} \isom \pair{\omega,\omega} \]
of line bundles on $Y$. The Faltings delta-invariant is defined to be the minus logarithm of the norm of $\mu_Y$, 
\begin{equation} \label{deffaltingsdelta} \delta_F = -\log\|\mu_Y\| \, ,
\end{equation}
where $\pair{\omega,\omega}$ is endowed with the Deligne metric derived from $\omar$, and $\det \pi_* \omega$ is endowed with the Hodge metric $\|\cdot\|_H$ induced from (\ref{defineinnerproduct}).

Let $j \colon J \to Y$ be the jacobian fibration associated to $X \to Y$. 
Then $J \times_Y J$ carries a canonical Poincar\'e bundle $\pp$, together with a rigidification at the origin, and $\pp$ carries a canonical hermitian metric, compatible with the rigidification. We denote by $\bb$ the $C^\infty$ hermitian line bundle on $J$ obtained by restricting $\pp$ to the diagonal.

Suppose that $E$, $F$ are divisors on $X$ of relative degree zero over $Y$. Then we obtain a section $\eta$ of $J \times_Y J \to Y$ given by sending $y \in Y$ to the classes of $E_y$, $F_y$ in $\Jac(X_y)$. We have a canonical isometry (see for example \cite[Corollaire 4.14.1]{mp})
\[ \eta^*\pp \isom \pair{\oo(E)_\Ar,\oo(F)_\Ar}^{\otimes -1} \, .   \]
In particular, for the section $\nu \colon Y \to J$ associated to the divisor $E$ we have a canonical isometry
\begin{equation} \label{delignepoinc} \nu^*\bb \isom \pair{\oo(E)_\Ar,\oo(E)_\Ar}^{\otimes -1}   \, .
\end{equation}
Assume that $X \to Y$ extends into a nodal curve $\Xbar \to \Ybar$, where $\Ybar$ is a smooth complex algebraic variety and where $\Ybar \setminus Y$ is a normal crossings divisor on $Y$. Then the section $\nu$ is a normal function section (cf. \cite[Definition 5.2]{hain_normal}) of the family of jacobians $J \to Y$.

\section{Variations of mixed Hodge structures}  \label{sec:var_mhs}

The purpose of this section is to recall a few notions and results related to graded-polarized variations of mixed Hodge structures, in particular the Nilpotent Orbit Theorem from \cite{pearlhiggs}.  In the next section we then make this Nilpotent Orbit Theorem explicit for families of curves with nodal degeneration, in order to obtain precise asymptotic formulae for the period and Abel-Jacobi map in such families. For the notion of a graded-polarized variation of mixed Hodge structures over a complex manifold we refer to \cite{pearlhiggs} and \cite[Chapter~14.4]{ps}. 

Let $V$ be a $\qq$-vector space, and $N$ a nilpotent endomorphism of $V$. There is then a unique increasing filtration $L_\bullet$ of $V$ such that
\begin{enumerate}
\item for each $i \in \zz$ one has $NL_i \subseteq L_{i-2}$,
\item for each $i \in \nn$ one has that $N^i$ induces an isomorphism $ N^i \colon \Gr_i^L V \isom \Gr_{-i}^L V $ of vector spaces.
\end{enumerate} 
Assume that $V$ is moreover equipped with a finite increasing filtration $W_\bullet$, such that $N$ preserves $W_\bullet$. An increasing filtration  $M_\bullet$ of $V$ is then called a weight filtration for $N$ relative to $W_\bullet$ if:
\begin{enumerate}
\item for each $i \in \zz$ one has $NM_i \subseteq M_{i-2}$,
\item for each $k \in \zz$ and each $i \in \mathbb{N}$ one has that $N^i$ induces an isomorphism
\[ N^i \colon \Gr_{k+i}^M \Gr_k^W V \isom \Gr_{k-i}^M \Gr_k^W V \]
of vector spaces.
\end{enumerate}
If a weight filtration for $N$ relative to $W_\bullet$ exists on $V$, it is unique. 

Let $(V,W_\bullet,F^\bullet,Q_\bullet)$ be a graded-polarized mixed Hodge structure on $V$, with weight filtration $W_\bullet$, graded-polarization $Q_\bullet$ and Hodge filtration $F^\bullet$. Associated to $(V,W_\bullet,F^\bullet,Q_\bullet)$ one has a natural classifying space (period domain) $\mm$ of graded-polarized mixed Hodge structures on  $(V,W_\bullet, Q_\bullet)$, by varying the Hodge filtration.  Let $G=G(V,W_\bullet,Q_\bullet)$ be the associated $\qq$-algebraic group consisting of elements $ g \in GL(V)^W $ such that $\Gr^W_\bullet(g) \in \mathrm{Aut}(Q_\bullet)$. Then $G(\rr)$ acts transitively on $\mm$, and provides $\mm$ with an embedding into a ``compact dual'' $\check{\mm} \supset \mm$, which is the orbit, inside a flag variety parametrizing decreasing filtrations of $V_\cc$ compatible with $W_\bullet$, of any point in $\mm$ under the action of $G(\cc)$. The inclusion $\mm \subset \check{\mm}$ gives $\mm$ a natural structure of complex manifold. 

Let $n \geq m \geq 0$ be integers, let $Y=(\Delta^*)^m \times \Delta^{n-m}$, and let $(\vv,\ww_\bullet,\ff^\bullet,\boldsymbol{Q}_\bullet)$ be a graded-polarized variation of mixed Hodge structures over $Y$.  Let $(V,W_\bullet,F^\bullet,Q_\bullet)$ be a reference fiber at a point $t_0\in Y$ near the origin, write $G=G(V,W_\bullet,Q_\bullet)$ as above, and let $\mm$ be the associated period domain. Then we have a monodromy representation $\rho \colon \pi_1(Y,t_0) \to G(\qq)$, and denoting by $\Gamma \subset G(\qq)$ its image we have a natural associated holomorphic period map $\phi \colon Y \to \mm/\Gamma$.

Assume from now on that the variation $(\vv,\ww_\bullet,\ff^\bullet,\boldsymbol{Q}_\bullet)$ on $Y$ is admissible. For a precise definition we refer to \cite[Section~14.4.1]{ps}. The admissibility condition implies in particular that all local monodromy operators around the branches of the boundary divisor $D=\Delta^n \setminus Y$are unipotent. The vector bundle $\VV=\vv \otimes_\cc \oo_Y$ has a canonical Deligne extension $\tilde{\VV}$ over $\Delta^n$, and the central fiber $\tilde{\VV}(0)$ comes equipped with a natural Hodge filtration $\tilde{\ff}^\bullet(0)$. 

Moreover, let $(V,W_\bullet,F^\bullet,Q_\bullet)$ be a reference fiber of $(\vv,\ww_\bullet,\ff^\bullet,\QQ_\bullet)$ near $D$, and denote the local monodromy logarithms in $\mathfrak{g}=\Lie G(\rr)$ by $N_1,\ldots,N_m$. Let $\mathcal{C}$ be the $\rr_{>0}$-span of the $N_i$ inside $\mathfrak{g}$. Then each element of $\mathcal{C}$ is nilpotent, and compatible with $W_\bullet$, the relative weight filtration of $(V,W_\bullet)$ exists for each element in $\mathcal{C}$, and is in fact constant on $\mathcal{C}$. Using parallel transport along the Gauss-Manin connection, the central fiber $\tilde{\VV}(0)$ of the Deligne extension is then  equipped with a canonical increasing filtration $M_\bullet$. The triple $(\tilde{\VV}(0),\tilde{\ff}^\bullet(0),M_\bullet)$ is a mixed Hodge structure, the limit mixed Hodge structure of the variation.

Let $U$ be Siegel's upper half plane. We denote by $e \colon U^m \to (\Delta^*)^m$ the uniformization map given by $(z_1,\ldots,z_m) \mapsto (\exp(2\pi i z_1),\ldots,\exp(2\pi i z_m))$. Via $e$ the period map $\phi$ of our variation lifts to a map $\tilde{\phi} \colon U^m \times \Delta^{n-m} \to \mm$. 
Then as $N_i \in \Lie G(\rr)$ we have $\exp(\sum_{i=1}^m z_iN_i) \in G(\cc)$ for all $z_1,\ldots,z_m \in U$. Let then $\tilde{\psi} \colon  
U^m \times \Delta^{n-m} \to \check{\mm}$ be the map 
\[ \tilde{\psi}(z_1,\ldots,z_m,t_{m+1},\ldots,t_n) = \exp(-\sum_{i=1}^m z_iN_i) \tilde{\phi}(z_1,\ldots,z_m,t_{m+1},\ldots,t_n) \, . \] 
Then $\tilde{\psi}$ descends to an ``untwisted'' period map $\psi \colon (\Delta^*)^m \times \Delta^{n-m} \to \check{\mm}$.

In this set-up, one has the following Nilpotent Orbit Theorem \cite[Section~6]{pearlhiggs}. 
\begin{thm} \label{nilpotentorbit} Let $(\vv,\ww_\bullet,\ff^\bullet,\boldsymbol{Q}_\bullet)$ be an admissible graded-polarized variation of mixed Hodge structures over $Y=(\Delta^*)^m \times \Delta^{n-m}$. (a) The untwisted period map $\psi$ extends to a holomorphic map $\psi \colon \Delta^n \to \check{\mm}$. (b) The value $\psi(0)$ at zero coincides with the Hodge filtration of the limit mixed Hodge structure $(\tilde{\VV}(0),\tilde{\ff}^\bullet(0),M_\bullet)$. 
\end{thm}

\section{Asymptotics of the period and Abel-Jacobi map} \label{asympresults}

The purpose of this section is make the Nilpotent Orbit Theorem sufficiently explicit for the period and Abel-Jacobi map associated to a family of curves with nodal degeneration. We note that the relevant graded-polarized variations of mixed Hodge structure will indeed be admissible.  We mention \cite[Section~1]{ca}, \cite[Section~2]{fr} and \cite{hain_periods} as references for the material below. However, we warn the reader that these references deal with the Hodge structure on cohomology (pure weight $+1$) of curves, whereas here we deal with homology (pure weight $-1$). The results we find are certainly well known, but we have not been able to localize them as such in the literature. 

If $(V,F^\bullet,W_\bullet)$ is a mixed Hodge structure, we decide to call $\mathrm{Ext}^1_{\mathrm{MHS}}(\zz(0),V)$ the generalized jacobian of $(V,F^\bullet,W_\bullet)$. For example, let $C$ be a nodal complex algebraic curve of arithmetic genus $h$, and write $M=\rmH_1(C)$. Then we let $\Jac(C)=\mathrm{Ext}^1_{\mathrm{MHS}}(\zz(0),M)$ be the generalized jacobian of (the homology) of $C$. It turns out that $\Jac(C)$ can be given a natural structure of complex algebraic group, namely as the group of isomorphism classes of invertible sheaves on $C$ that are of degree zero on every component of $C$. Let $\nu \colon \tilde{C}\to C$ denote the normalization of $C$, and write $U=\rmH_1(\tilde{C})=\oplus \rmH_1(\tilde{C}_i)$, with $\tilde{C}_i$ the normalizations of the irreducible components of $C$. Let $E =\Jac(U)=\oplus \Jac(\tilde{C}_i)$. Then $E$ is the jacobian of a pure polarized Hodge structure of type $\{(-1,0), (0,-1)\}$. Write $F= \Jac(C)$. 

Let $G$ be the dual graph of the nodal curve $C$. Pulling back invertible sheaves along $\nu$ gives a canonical extension of algebraic groups
\begin{equation} \label{extension}  
 1 \to \rmH^1(G) \otimes \mathbb{G}_\mathrm{m} \to F \to  E \to 0 \, . 
\end{equation}
Taking the homotopy long exact sequence of (\ref{extension}) we obtain an extension of mixed Hodge structures
\[ 0 \to \rmH^1(G) \otimes \zz(1) \to \rmH_1(C) \to \rmH_1(\tilde{C}) \to 0 \, . \]
Vice versa, we may obtain (\ref{extension}) by applying $\mathrm{Ext}_{\mathrm{MHS}}^1(\zz(0),-)$ to the above extension. 

Let $P, Q \in C$ be points. Then we have an extension
\[ 0 \to M=\rmH_1(C) \to \rmH_1(C,\{P,Q\}) \to \zz(0) \to 0 \]
of mixed Hodge structures. Generalizing the notation from Section \ref{jacobians} we denote by $\int_Q^P$ the corresponding element of $F=\mathrm{Ext}^1_{\mathrm{MHS}}(\zz(0),M)$. By a slight abuse of notation, we also denote by $\int_Q^P$ the resulting image in $E=\mathrm{Ext}^1_{\mathrm{MHS}}(\zz(0),U)$ under the projection $F \to E$. 

An alternative way of describing $\int_Q^P$ is as follows. Let $\omega_C$, $\omega_{\tilde{C}}$ denote the dualizing sheaves of $C$ and $\tilde{C}$, respectively. The canonical map $\nu_*\omega_{\tilde{C}} \to \omega_{C}$ yields an inclusion $\omega_{\tilde{C}}(\tilde{C}) \to \omega_C(C)$ and hence a surjection $ \omega_C(C)^\lor \to \omega_{\tilde{C}}(\tilde{C})^\lor $. This surjection coincides with the induced map $\Lie(F) \to \Lie(E)$ from (\ref{extension}). We can alternatively write $E = \omega_{\tilde{C}}(\tilde{C})^\lor/\rmH_1(\tilde{C})$, and we find the following lemma.
\begin{lem} \label{AJdescription}
The Abel-Jacobi element $\int_Q^P$ of $E=\omega_{\tilde{C}}(\tilde{C})^\lor/\rmH_1(\tilde{C})$ is the functional given by integrating elements of $\omega_{\tilde{C}}(\tilde{C})\subset \omega_C(C)$ along paths from $P$ to $Q$ over $C$. 
\end{lem}

Now let $\Ybar = \Delta^n$, and let $\pi \colon \Xbar \to \Ybar$ be a nodal curve, assumed to be smooth over $Y=(\Delta^*)^m \times \Delta^{n-m}$. Let $X=\pi^{-1}Y$ and let $\vv=\R^1 \pi_* \zz_X(1)$. Then $\vv$ underlies a canonical admissible variation of pure polarized Hodge structure $(\vv,\ff^\bullet,\boldsymbol{Q})$ of weight $-1$ over $Y$. From now on, we will usually suppress the polarizations from our notation, as they will be clear from the context.
Let $\Xbar_0$ be the fiber of $\Xbar \to \Ybar$ over the origin, let $G$ be its dual graph, and let $\nu \colon \tilde{\Xbar}_0 \to \Xbar_0$ be its normalization.
Our discussion above yields an extension of mixed Hodge structures
\[ 0 \to \rmH^1(G) \to \rmH_1(\Xbar_0) \to \rmH_1(\tilde{\Xbar}_0) \to 0    \]
with $\rmH^1(G) $ pure of type $(-1,-1)$.
Let $(V,F^\bullet)$ be a reference fiber of $\vv$ near the origin. By the very construction of the limit mixed Hodge structure on the central fiber of the canonical extension of $\VV=\vv \otimes_\cc \oo_Y$, the first part $\rmH^1(G) \to \rmH_1(\Xbar_0) $ of the extension can be realized as a piece of the monodromy (relative) weight filtration on $V$. More precisely, let $N$ be any element of the monodromy cone associated to $V_\qq$, then we have $N^2=0$ and the associated filtration on $V_\qq$ reads
\[ M_\bullet \colon \quad 0 \subset M_{-2} \subset M_{-1} \subset M_0 = V_\qq \, , \]
with (given our choice of coordinates) canonical identifications $M_{-2}=\Im N \cong \rmH^1(G)_\qq$ and $M_{-1}=\Ker N \cong \rmH_1(\Xbar_0)_\qq$. 

Let $r=\rank \rmH^1(G)$. One can choose a symplectic basis $(e_1,\ldots,e_h,f_1,\ldots,f_h)$ of $(V,Q)$ such that:
\begin{enumerate}
\item $\rmH^1(G)=\spann{(e_1,\ldots,e_r)}$, 
\item $\rmH_1(\Xbar_0)=\spann{(e_1,\ldots,e_h,f_{r+1},\ldots,f_h)}$, 
\end{enumerate}
In particular, $(\bar{e}_{r+1},\ldots,\bar{e}_h,\bar{f}_{r+1},\ldots,\bar{f}_h)$ is a symplectic basis of the pure polarized Hodge structure 
$\Gr_{-1}^M V \cong \rmH_1(\tilde{\Xbar}_0)$ of type $\{(-1,0), (0,-1)\}$. 
With respect to our chosen basis, each local monodromy operator $N_j$ has the form
\[ N_j = \begin{pmatrix}[c|c] 0 & A_j \\
\hline 
0 & 0 \\
\end{pmatrix} \, , \]
where $A_j$ is an integral symmetric positive semi-definite $h$-by-$h$ matrix, and where each non-zero element of $A_j$ is in the left upper $r$-by-$r$ block of $A_j$. 

As is well known, the period domain $\mm$ associated to $\vv$ can be realized as Siegel's upper half space $U_h$ of degree $h$. We have $G(V)=\Sp(2h)_\qq$, and the action of $G(V)(\rr)$ on $U_h$ is given by
\[ \begin{pmatrix}[c|c] A & B \\
\hline 
C & D \\
\end{pmatrix} \cdot M = (AM+B)(CM+D)^{-1} \, , \,     
\begin{pmatrix}[c|c] A & B \\
\hline 
C & D \\
\end{pmatrix} \in \Sp(2h,\rr) \, , \, M \in U_h \, . \]
The period map $\Omega \colon Y \to  U_h/\Gamma$ is given by associating to each $t \in Y$ the period matrix of $X_t$ on the chosen symplectic basis of $V$ (cf. Section \ref{jacobians}). Here $\Gamma$ is the image of the monodromy representation into $G(V)(\qq)=\Sp(2h,\qq)$. For a commutative ring $R$, denote by $S(r \times r,R)$ the set of symmetric $r$-by-$r$ matrices with values in $R$.
\begin{prop} \label{asymptperiod} 
(a) There exists a holomorphic map $\psi \colon \Delta^n \to S(h \times h, \cc)$ such that for $(z,t) \in U^m \times \Delta^{n-m}$ with $e(z)$ sufficiently close to zero the equality
\[ \Omega(e(z),t) = \sum_{j=1}^m A_j z_j + \psi(e(z),t) 
\]
holds in $S(h \times h, \cc)$.
(b) Writing
\[ \psi =  \begin{pmatrix}[c|c] \psi_{11} & \psi_{12} \\
\hline 
\psi_{21} & \psi_{22} \\
\end{pmatrix}   \]
with $\psi_{11}$ an $r$-by-$r$ matrix we have that $\psi_{22}(0)$ is equal to the period matrix of $(\Gr_{-1}^M V,\tilde{\ff}^\bullet(0))$ on the symplectic basis $(\bar{e}_{r+1},\ldots,\bar{e}_h,\bar{f}_{r+1},\ldots,\bar{f}_h)$. 
\end{prop}
\begin{proof} First of all, we have $\exp(N_jz_j)M=A_jz_j+M$ for each $M \in U_h$, $z_j \in U$, and $j=1,\ldots,m$. Part (a) then follows from Theorem \ref{nilpotentorbit}(a). As for part (b), we know by Theorem \ref{nilpotentorbit}(b) that $\psi(0)$ is the period matrix of the  limit mixed Hodge structure $(V,\tilde{\ff}^\bullet(0),M_\bullet)$ of $(\boldsymbol{V},\ff^\bullet)$ on the symplectic basis $(e_1,\ldots,e_h,f_1,\ldots,f_h)$ of $V$. That is, there exists a basis $(v_1,\ldots,v_h)$ of $\tilde{\ff}^0(0) V_\cc$ such that the identities $v_i = -\sum_{j=1}^h \psi(0)_{ij}e_j + f_i$ hold in $V_\cc$ for $i=1,\ldots,h$. Note that $v_{r+1},\ldots,v_h$ are all in $M_{-1,\cc}=\cc$-$\spann{(e_1,\ldots,e_h,f_{r+1},\ldots,f_h)}$. Let 
\[ p \colon \tilde{\ff}^0(0)V_\cc \cap M_{-1,\cc} \to \tilde{\ff}^0(0)(M_{-1,\cc}/M_{-2,\cc}) \] 
be the projection. Then we have the identities $p(v_i)=-\sum_{j=r+1}^h \psi(0)_{ij}\bar{e}_j + \bar{f}_i$ in $M_{-1,\cc}/M_{-2,\cc}=\Gr_{-1}^M V_\cc $ for $i=r+1,\ldots,h$. As the $p(v_i)$ are clearly linearly independent and as $\dim \tilde{\ff}^0(0)\Gr_{-1}^M V_\cc = h-r$ we conclude that $\psi_{22}(0)$ is the period matrix of the pure polarized Hodge structure $\Gr_{-1}^M V$ on its symplectic basis $(\bar{e}_{r+1},\ldots,\bar{e}_h,\bar{f}_{r+1},\ldots,\bar{f}_h)$.
\end{proof}

Now assume two distinct sections $P$, $Q$ of $\pi \colon \Xbar \to \Ybar$ are given. Over each $t \in Y$ we obtain an extension
\[  0 \to \rmH_1(X_t) \to \rmH_1(X_t,\{P_t,Q_t\}) \to \zz(0) \to 0 \]
in the category of mixed Hodge structures. Varying $t \in Y$ we obtain a canonical extension
\[ 0 \to \vv \to \vv(P,Q) \to \zz(0) \to 0 \]
of variations of mixed Hodge structure over $Y$. The variation $\vv(P,Q)$ is graded-polarized and moreover admissible. We are interested in the asymptotics of its period map near $D=\Ybar \setminus Y$. 

The weight filtration of the variation is
\[ W_\bullet \colon \quad 0 \subset \boldsymbol{W}_{-1}=\boldsymbol{V}_\qq \subset \boldsymbol{W}_0 = \boldsymbol{V}(P,Q)_\qq \, , \]
so that $\Gr_{-1}^{\boldsymbol{W}} \boldsymbol{V}(P,Q)_\qq=\boldsymbol{V}_\qq$, $\Gr_{0}^{\boldsymbol{W}} \boldsymbol{V}(P,Q)_\qq=\qq(0)$. We denote the Hodge filtration of $\vv(P,Q)_\qq$ by $\ff'^\bullet$. We start by taking a reference fiber $V(P,Q)$ of $\boldsymbol{V}(P,Q)$ and augmenting our chosen symplectic basis of $V$ by an $e_0 \in V(P,Q)$ lifting the canonical generator of $\zz(0)$ as in Section \ref{jacobians}.
Letting $M_\bullet$ as before be the monodromy (relative) weight filtration on $V_\qq$, we recall that, given our choice of coordinates, there is a natural isomorphism $M_{-1} \cong \rmH_1(\Xbar_0)_\qq$. The topological construction of this isomorphism via Picard-Lefschetz theory (cf. \cite[Section X.9]{acg}) makes it clear that, likewise, the relative homology $\rmH_1(\Xbar_0,\{P(0),Q(0)\})_\qq$ can be realized inside $V(P,Q)_\qq$, as the submodule $M_{-1}+\qq e_0$.

The admissibility of our variation implies that the relative weight filtration $M'_\bullet$ on the reference fiber $V(P,Q)_\qq$ exists. Let $N$ be any element from the monodromy cone acting on $V(P,Q)_\qq$. Our first aim is to determine the matrix shape of $N$ on our chosen basis $(e_0,e_1,\ldots,e_h,f_1,\ldots,f_h)$ of $V(P,Q)$. As $N^2=0$, the filtration $L_\bullet$ associated to $N$ (see the beginning of Section \ref{sec:var_mhs}) on $V(P,Q)_\qq$ is 
\[ L_\bullet \colon \quad 0 \subset L_{-1} \subset L_0 \subset L_1=V(P,Q)_\qq \, , \]
with $L_{-1}=\Im(N)$, $L_0=\Ker(N)$. On the other hand, the weight filtration is
\[ W_\bullet \colon \quad 0 \subset W_{-1}=V_\qq \subset W_0=V(P,Q)_\qq \, , \] 
which is indeed compatible with $N$. As the monodromy action on $\Gr_0^W =\qq(0)$ is trivial, we have that $\Im(N) \subset V_\qq$, so that $N^{-1}V_\qq=V(P,Q)_\qq$. We note that $W_\bullet$ has length two. It then follows from \cite[Proposition 2.16]{sz} and 
\cite[Proposition 2.11]{sz} that $N$ is a ``strict'' endomorphism, and that the weight filtration of $N$ relative to $W_\bullet$ is equal to the ``convolution'' of $L_\bullet$ and $W_\bullet$. The first statement means that $N^{-1}W_k = W_k + \Ker N$ for all $k \in \zz$, and the second statement means that $V(P,Q)_\qq=N^{-1}V_\qq$ and that
\[ M'_\bullet \colon \quad 0\subset M'_{-2} \subset M'_{-1} \subset M'_0 \subset M'_{1}= V(P,Q)_\qq  \]
with $M'_{-2}=\Im(N)$, $M'_{-1}=\Im(N)+\Ker(N|_{V_\qq})=\Ker(N|_{V_\qq})=M_{-1}$ and $M'_0=V_\qq+\Ker(N)=V(P,Q)_\qq$ is the weight filtration of $N$ relative to $W_\bullet$. The equalities $V(P,Q)_\qq=N^{-1}V_\qq=V_\qq+\Ker(N)$  imply that $\Ker(N) \supsetneqq \Ker(N|_{V_\qq})$ and hence that $\Im(N)=\Im(N|_{V_\qq})$, that is, $M'_{-2}=M_{-2}$.

The period domain associated to $(V(P,Q),W_\bullet)$ is $\cc^h \times U_h$, and the associated algebraic group has $\qq$-points 
\[ G(V(P,Q),W_\bullet)(\qq) = \left\{ \begin{pmatrix}[c|c|c]
1 & 0 & 0 \\ \hline
m & A & B \\ \hline
n & C & D
\end{pmatrix}  \, :  \, m, n \in \qq^h \, , \, \begin{pmatrix}[c|c]
A & B \\ \hline C & D \end{pmatrix} \in \Sp(2h,\qq)
\right\} \, .  \]
The action of $G(V(P,Q),W_\bullet)(\rr)$ on $\cc^h \times U_h$  is given by
\[ \begin{pmatrix}[c|c|c]
1 & 0 & 0 \\ \hline
m & A & B \\ \hline
n & C & D
\end{pmatrix} (v,M) = (v+m+Mn,(AM+B)(CM+D)^{-1}) \, , \, v \in \cc^h \, , \, M \in U_h \, .
\]

Varying $t \in Y$ and then taking $F^0$ we obtain a $(\cc^h \times U_h)/\Gamma$-valued period map
\[ (\delta,\Omega) \colon Y \to (\cc^h \times U_h)/ \Gamma \]
associated to the variation $\vv(P,Q)$. For each $t \in Y$, the vector $\delta(t)$ is a lift of the Abel-Jacobi element in $\cc^h/(\zz^h + \Omega(t)\zz^h)$ associated to the divisor $P(t)-Q(t)$ of $X_t$. 

\begin{prop} \label{asymptAJ} 
(a) There exist  a holomorphic single-valued map $\alpha \colon \Delta^n \to \cc^h$ and vectors $b_1,\ldots, b_m \in \qq^h$ with $A_jb_j \in \zz^h$  for $j=1,\ldots,m$ such that for $(z,t) \in U^m \times \Delta^{n-m}$ with $e(z)$ sufficiently close to zero the equality 
\[  \delta(e(z),t) = \sum_{j=1}^m A_jb_j z_j + \alpha(e(z),t)
\]
holds in $\cc^h$.
(b) Writing $ \alpha = \genfrac{(}{)}{0pt}{}{\alpha_1}{\alpha_2} $ with $\alpha_1 \colon \Delta^n \to \cc^r$ we have that the vector $\alpha_2(0)$ lifts the Abel-Jacobi element $\int_{Q(0)}^{P(0)}$ of $E=\Jac(\tilde{\Xbar}_0)$ determined by integrating elements of $\omega_{\tilde{\Xbar}_0 }(\tilde{\Xbar}_0)$ along paths from $P(0)$ to $Q(0)$ on $\Xbar_0$.
\end{prop}
\begin{proof} Let $N_j$ denote the local monodromy operator of $V(P,Q)$ around the branch of $D$ determined by $t_j=0$. The equality $\Im(N_j)=\Im(N_j|_{V_\qq})$ that we established above shows that $N_j$ has a matrix
\[ N_j = \begin{pmatrix}[c|c|c]
0 & 0 & 0 \\ \hline
A_jb_j & 0 & A_j \\ \hline
0 & 0 & 0
\end{pmatrix} \]
on the basis $(e_0,e_1,\ldots,e_h,f_1,\ldots,f_h)$, for some $b_j \in \qq^h$ such that $A_jb_j \in \zz^h$ for each $j=1,\ldots,m$. Then for all $(v,M) \in \cc^h \times U_h$ and $z_j \in U$ we have $\exp(N_jz_j)(v,M)=(v+A_jb_jz_j,M+A_jz_j)$, and we find part (a) of the proposition by applying Theorem \ref{nilpotentorbit}(a). 
As for part (b), we know by Theorem \ref{nilpotentorbit}(b) that $\alpha(0)$ is the period vector of the limit mixed Hodge structure $(V(P,Q),\tilde{\ff}'^\bullet(0),M'_\bullet)$ of the variation $\boldsymbol{V}(P,Q)$. That is, the normalized basis $(v_1,\ldots,v_h)$ of $\tilde{\ff}^0(0)V_\cc$ on the symplectic basis $(e_1,\ldots,e_h,f_1,\ldots,f_h)$ that we obtained in the proof of Theorem \ref{asymptperiod}(b) can be extended to a basis $(v_0,v_1,\ldots,v_h)$ of $\tilde{\ff}'^0(0)V(P,Q)_\cc$ in such a way that the identity $v_0 = e_0 + \sum_{i=1}^h \alpha(0)_ie_i$ holds in $V(P,Q)_\cc$. As we discussed above, the sub-mixed Hodge structure $H_1(\Xbar_0)+\zz e_0$ of $(V(P,Q),\tilde{\ff}'^\bullet(0),M'_\bullet)$ can be realized as the extension $\rmH_1(\Xbar_0,\{P(0),Q(0)\})$ of $\rmH_1(\Xbar_0)$. Let $U=H_1(\Xbar_0)/H^1(G)$ and $U(P,Q)=(H_1(\Xbar_0)+\zz e_0)/H^1(G)$. We then obtain an extension
\begin{equation} \label{Uextension} 0 \to U  \to U(P,Q) \to  \zz(0) \to 0 
\end{equation}
of mixed Hodge structures. From Lemma \ref{AJdescription} we obtain that the element of $E=\mathrm{Ext}^1_{\mathrm{MHS}}(\zz(0),U)$ corresponding to (\ref{Uextension}) is the Abel-Jacobi element determined by integrating elements of $\omega_{\tilde{\Xbar}_0 }(\tilde{\Xbar}_0)$ along paths from $P(0)$ to $Q(0)$ on $\Xbar_0$. Our task is thus to show that 
$\alpha_2(0)=(\alpha(0)_{r+1},\ldots,\alpha(0)_h)$ is a period vector of the mixed Hodge structure $U(P,Q)$. 
Let 
\[ p \colon \tilde{\ff}^0(0)V_\cc \cap M_{-1,\cc} \to \tilde{\ff}^0(0)(M_{-1,\cc}/M_{-2,\cc})=\tilde{\ff}^0(0)U_\cc \]
be the projection. We recall from the proof of Theorem \ref{asymptperiod}(b) that $v_{r+1},\ldots,v_h \in M_{-1,\cc}$ and that $(p(v_{r+1}),\ldots,p(v_h))$ is a basis of $\tilde{\ff}^0(0)U_\cc$.  Next we note that $v_0 \in M'_{-1,\cc}+\cc e_0$. Let 
\[ q \colon \tilde{\ff}'^0(0)V(P,Q)_\cc \cap (M'_{-1,\cc}+\cc e_0) \to \tilde{\ff}'^0(0)(M'_{-1,\cc}+\cc e_0)/M'_{-2,\cc}= \tilde{\ff}'^0(0)U(P,Q)_\cc \] 
be the projection. Then $(q(v_0),q(v_{r+1}),\ldots,q(v_h))$ is a basis of $\tilde{\ff}'^0(0) U(P,Q)_\cc$ extending the basis $(q(v_{r+1}),\ldots,q(v_h))=(p(v_{r+1}),\ldots,p(v_h))$ of $\tilde{\ff}^0(0)U_\cc$.
Moreover we have the identity $ q(v_0) = \bar{e}_0 + \sum_{i=r+1}^h \alpha(0)_i\bar{e}_i$ in $U(P,Q)_\cc$. This shows that $(\alpha(0)_{r+1},\ldots,\alpha(0)_h)$ is a period vector of $U(P,Q)$.
\end{proof}

\section{Lear extensions} \label{lear}

The notion of Lear extension is introduced by Hain in \cite{hrar} \cite{hain_normal}. The notion has turned out to be  a very useful tool in formulating and analyzing asymptotic properties of Hodge theoretic and Arakelov theoretic invariants. The terminology is justified by Theorem \ref{generalexistence} below, which was first obtained by D.~Lear in his PhD thesis \cite{lear}.

Let $Y \subset \Ybar$ be complex manifolds, and assume that $D = \bar{Y} \setminus Y$ is a normal crossings divisor on $\Ybar$. Let $L$ be a holomorphic line bundle on $Y$. An extension of $L$ over $\Ybar$ consists of a pair $(\Lbar, \alpha \colon \Lbar|_Y \isom L)$, where $\Lbar$ is a holomorphic line bundle on $\Ybar$ and $\alpha$ is an isomorphism of holomorphic line bundles. There is a natural notion of isomorphism of extensions $(\Lbar, \alpha)$. We denote the set of isomorphism classes of extensions of $L$ by $\Pic_L(\Ybar)\subset \Pic(\Ybar)$. 

Assume that $L$ is equipped with a continuous hermitian metric $\|\cdot\|$.
A Lear extension of $(L,\|\cdot\|)$ over $\bar{Y}$ is an extension $(\Lbar, \alpha)$ of $L$ over $\Ybar$ with the property that the continuous metric $\alpha^*\|\cdot\|$ on $\Lbar|_Y$ extends as a continuous metric over $\overline{L}|_{\bar{Y} \setminus D^{\mathrm{sing}}}$. 
Note that if $(L,\|\cdot\|)$ has a Lear extension over $\Ybar$, the underlying line bundle $\Lbar$ is unique up to isomorphism, since $D^\mathrm{sing}$ lies in codimension $\geq 2$ on $\Ybar$. 

We conclude that all Lear extensions yield the same class in $\Pic_L(\Ybar)$, and usually we think of ``the'' Lear extension of $(L,\|\cdot\|)$ as this class in $\Pic_L(\Ybar)$. By slight abuse of language, if a tensor power $L^{\otimes N}$ has a Lear extension $(\overline{L^{\otimes N}},\alpha)$, we also say that $L$ has a Lear extension, and we call the class of the $\qq$-line bundle $\frac{1}{N} \overline{L^{\otimes N}}$ in $\Pic_L(\Ybar) \otimes \qq$ its Lear extension. Note that this class is independent of the choice of $N$. 

Let $p \in D \setminus D^\sing$. We say a coordinate chart $(t_1,\ldots,t_n) \colon U \isom \Delta^n$ of $\bar{Y}$ with center $p$ is adapted to $D$ if $D \cap U$ is given by the equation $t_1=0$. Clearly one can verify Lear extendability over $D$ locally on coordinate charts adapted to $D$. 
Write $D_1$ for the divisor of $\Delta^n$ given by the equation $t_1=0$. The following gives a criterion for continuous extendability of a continuous hermitian line bundle on $\Delta^n \setminus D_1=\Delta^* \times \Delta^{n-1}$ over $\Delta^n$, and hence for Lear extendability in general.  
\begin{prop} \label{learchar} Let $(L,\|\cdot\|)$ be a continuous hermitian line bundle over $\Delta^n \setminus D_1$, and let $M$ be a holomorphic line bundle on $\Delta^n$ coinciding with $L$ on $\Delta^n \setminus D_1$. Let $\mu \in \qq$. Let $s$ be a generating section of $M$ over $\Delta^n$. The following assertions are equivalent: (i) the function $\log\|s\| - \mu \log |t_1|$ on $\Delta^n \setminus D_1$ extends continuously over $\Delta^n$, (ii) the continuous hermitian line bundle $(L,\|\cdot\|)$ has a Lear extension $\Lbar$ over $\Delta^n$ and we have $\Lbar = M+\mu D_1$.
\end{prop}
\begin{proof} Write $M(\mu)$ as a shorthand for $M+\mu D_1$. Note that $M(\mu)$ coincides with $L$ over $\Delta^n \setminus D_1$, and that the rational section $s'=t_1^{-\mu}s$ is a generating section of $M(\mu)$ over $\Delta^n$. We have that $M(\mu)$ can be equipped with the structure of continuous hermitian $\qq$-line bundle extending $(L,\|\cdot\|)$ over $\Delta^n$ if and only if $\log \|s'\|$ extends continuously over $\Delta^n$. Since $\log \|s'\| = \log \|s\| - \mu \log |t_1|$ we obtain the equivalence of assertions (i) and (ii).
\end{proof}
Now assume that $\Ybar$ is a smooth complex algebraic variety, with $D=\Ybar \setminus Y$ a normal crossings divisor on $\Ybar$. Let $\vv$ be a polarized variation of Hodge structures  of weight $-1$ over $Y$. Following \cite[Section~5]{hain_normal} we have canonically associated to $\vv$ a family of intermediate jacobians $J(\vv)\to Y$. The total space $J(\vv)$ carries a Poincar\'e line bundle $\bb$, rigidified along the origin and equipped with a canonical $C^\infty$ hermitian metric. The main result of Lear's thesis \cite{lear}, reproduced in \cite[Corollary~6.4]{hain_normal}, is the following extension result. 
\begin{thm} \label{generalexistence} Let $\nu \colon Y \to J(\vv)$ be a normal function section (cf. \cite[Definition 5.2]{hain_normal}) of the family of intermediate jacobians $J(\vv) \to Y$. Then the $C^\infty$-hermitian line bundle $L=\nu^* \bb$ on $Y$ has a Lear extension over $\bar{Y}$.
\end{thm}
The case of special interest to us is given by sections of jacobians obtained from relative degree zero divisors on families of curves with nodal degeneration over $D$, cf. the end of Section~\ref{sec:metrization}.

\section{Key isometries} \label{sec:keyisom}

We continue with the notation set in Section~\ref{sec:metrization}.
We draw here a few consequences of the canonical isometry (\ref{delignepoinc}). Let $h>0$ denote the genus of the fibers of the family $\pi \colon X \to Y$, and assume that a section $P \colon Y \to X$ of $\pi$ is given. 
\begin{prop} \label{kappa} Let $\kappa \colon Y \to J$ be the section of the jacobian fibration $j \colon J \to Y$ associated to the relative degree zero divisor $E=(2h-2)P-\omega$ on $X$ over $Y$. Then one has a canonical isometry
\[ \kappa^* \bb \isom \pair{\opar,\omar}^{\otimes 4h(h-1)} \otimes \pair{\omar,\omar}^{\otimes -1}  \]
of $C^\infty$ hermitian line bundles on $Y$.
\end{prop}
\begin{proof} By equation (\ref{delignepoinc}) we have a canonical isometry 
\[ \kappa^* \bb \isom \pair{(2h-2)\opar-\omega_\Ar,(2h-2)\opar-\omega_\Ar}^{\otimes -1} \, . \] 
We obtain the result upon expanding the right hand side using
the canonical adjunction isometry $\pair{\opar,\opar} \isom \pair{\opar,\omar}^{\otimes -1}$ from (\ref{adj}).
\end{proof}

If $X \to Y$ has two sections $P, Q \colon Y \to X$, we let $\delta \colon Y \to J$ be the section of $j \colon J \to Y$ associated to the relative degree zero divisor $E=P-Q$ of $X$ over $Y$. 

\begin{prop} \label{diagonal_metric} Assume $X \to Y$ has two sections $P, Q \colon Y \to X$. Then we have a canonical isometry
\[  \delta^*\bb  \isom \pair{\opar,\oqar}^{\otimes 2} \otimes \pair{\opar,\omar} \otimes \pair{\oqar,\omar}  \]
of $C^\infty$ hermitian line bundles on $Y$.
\end{prop}
\begin{proof} By (\ref{delignepoinc}) we have a canonical isometry $\delta^*\bb  \isom \pair{\oo(P-Q)_\Ar,\oo(P-Q)_\Ar}^{\otimes -1}$. The stated result follows upon expanding the right hand side, using the canonical adjunction isometry $\pair{\opar,\opar} \isom \pair{\opar,\omar}^{\otimes -1} $ from (\ref{adj}).
\end{proof}
Assume $X \to Y$ has a section $P \colon Y \to X$. Let $\delta_P \colon X \to J$ be the $Y$-morphism given by sending $z \in X_y$ to the class of $P(y)-z$ in $\Jac(X_y)$. 
\begin{prop} \label{deltaPupperstar}  We have a canonical isometry
\[ \delta_P^* \bb \isom \opar^{\otimes 2} \otimes \pi^*\pair{ \opar, \omar} \otimes \omar  \]
of $C^\infty$ hermitian line bundles on $X$.
\end{prop}
\begin{proof} Consider the first projection $\pi_1 \colon X \times_Y X \to X$. It comes equipped with two tautological sections, namely a section denoted $\tilde{P}$ induced from $P$, and the diagonal section $Q$. Let $\tilde{\bb}$ on $X \times_Y J$ denote the pullback of $\bb$  along the second projection $X \times_Y J \to J$. Let $\delta \colon X \to X \times_Y J$ be the section of the projection $X \times_Y J \to X$ associated to the two sections $\tilde{P}$ and $Q$ of $X \times_Y X \to X$.
Then we have a canonical isometry $\delta_P^*\bb \isom \delta^*\tilde{\bb}$. Let $\tilde{\omega}_\Ar = \pi_2^*\omega_\Ar$ denote the relative dualizing sheaf of $\pi_1$ endowed with the Arakelov metric. Then by Proposition \ref{diagonal_metric} we have a canonical isometry
\[ \delta^*\tilde{\bb} \isom \pair{ \oo(\tilde{P})_\Ar, \oqar }^{\otimes 2} \otimes \pair{\oo(\tilde{P})_\Ar, \tilde{\omega}_\Ar } \otimes
\pair{\oqar, \tilde{\omega}_\Ar }    \]
of $C^\infty$ hermitian line bundles on $X$. The proof follows upon observing the canonical isometries
\[ \pair{ \oo(\tilde{P})_\Ar, \oqar } \isom Q^* \oo(\tilde{P})_\Ar \isom \opar \, ,  \,
\pair{\oo(\tilde{P})_\Ar, \tilde{\omega}_\Ar } \isom \pi^*\pair{\opar,\omega_\Ar} \, , \]
and 
\[ \pair{\oqar,\tilde{\omega}_\Ar} \isom Q^* \tilde{\omega}_\Ar \isom \omar 
 \]
of $C^\infty$ hermitian line bundles on $X$.
\end{proof}

\begin{prop} \label{deltasq} Assume $X \to Y$ has a section $P \colon Y \to X$. Then we have a canonical isometry
\[ \pair{\delta_P^*\bb, \delta_P^*\bb} \isom \pair{\opar,\omar}^{\otimes 4h} \otimes \pair{\omar,\omar}  \]
of $C^\infty$ hermitian line bundles on $Y$.
\end{prop}
\begin{proof} Expand the Deligne pairing of the right hand side in Proposition \ref{deltaPupperstar} with itself, using the adjunction isometry $\pair{\opar,\opar} \isom \pair{\opar,\omar}^{\otimes -1} $ and the canonical isometries 
\[ \pair{\opar,\pi^*\pair{\opar,\omar}} \isom \pair{\opar,\omar} \] and 
\[ \pair{\omar,\pi^*\pair{\opar,\omar}} \isom \pair{\opar,\omar}^{\otimes 2h-2} \]
that we obtain from (\ref{delignenorm}).
\end{proof}
\begin{prop} \label{omega_in_terms_of_biext} Assume $X \to Y$ has a section $P \colon Y \to X$. Then we have a canonical isometry
\[ \pair{\opar,\omar}^{\otimes 4h^2} \isom \pair{\delta_P^*\bb, \delta_P^*\bb} \otimes \kappa^* \bb \]
of $C^\infty$ hermitian line bundles on $Y$.
\end{prop}
\begin{proof} This follows immediately upon combining Propositions \ref{kappa} and \ref{deltasq}.
\end{proof}

\section{Calculation of Lear extensions} \label{learI}

Our aim in this section is to explicitly calculate the Lear extensions associated to the various $C^\infty$ hermitian line bundles in Section \ref{sec:keyisom}. 

Let $\pi \colon \Xbar \to \Delta$ be a nodal curve over the unit disc, smooth over $\Delta^*$. Let $G$ be the weighted dual graph of the special fiber $\Xbar_0$ (see Section~\ref{prelimsemistable}). Let $P(\Xbar)$ be the (additively written) group of line bundles on $\Xbar$. Then we have a canonical specialization map $R \colon P(\Xbar) \to \rr^{V(G)}$ given by $R(L)(x)=(L\cdot x)_0=\deg (L|_x)$ for all $L \in P(\Xbar)$ and $x \in V(G)$. We put $K=R(\omega)$ in $\rr^{V(G)}$ where $\omega$ is the relative dualizing sheaf of $\pi$.  Let $\bar{g}$ denote the Green's function on $G$ determined by the discrete Laplacian as in Section \ref{prelimgraphs}.

Let $E$ be a Cartier divisor on $\Xbar$ with support in $\Sm(\pi)$, and with relative degree zero over $\Delta$. Equip $\oo(E)$ over $X=\pi^{-1}\Delta^*$ with the metric derived from the Arakelov-Green's function $g_\Ar$, notation $\oo(E)_\Ar$.  
\begin{prop} \label{learfromhdj} Write $\ee=R(\oo(E))$. Then the Lear extension $\overline{ \pair{ \oo(E), \oo(E) } }$ of the restriction of $\pair{\oo(E),\oo(E)}$ to $\Delta^*$ exists. We have an equality
\[ \overline{ \pair{ \oo(E), \oo(E) } } = \pair{\oo(E),\oo(E)} + \bar{g}(\ee,\ee)[0] \]
of $\qq$-line bundles over $\Delta$.
\end{prop}
\begin{proof} The existence of $\overline{ \pair{ \oo(E), \oo(E) } }$ follows from Theorem \ref{generalexistence} above in combination with (\ref{delignepoinc}). As to the formula for $\overline{ \pair{ \oo(E), \oo(E) } }$, let $\tilde{\Xbar} \to \Xbar$ denote the minimal desingularization of $\Xbar$. Let $\tilde{\Gamma}$ resp.\ $\Gamma$ denote the metrized graphs associated to the dual graphs of the special fibers of $\tilde{\Xbar}$ resp.\ $\Xbar$. By Proposition \ref{minimaldesing}, the minimal desingularization $\tilde{\Xbar} \to \Xbar$ of $\Xbar$ induces a canonical isometry $\tilde{\Gamma} \isom \Gamma$. We see that, since $E$ has support in the smooth locus of $\pi$, upon passing to the minimal desingularization $\tilde{\Xbar}$ neither the left hand side nor the right hand side of the equality to be proven changes. Hence we may assume that $\Xbar$ is smooth. Then the formula follows upon combining \cite[Theorem~2.2]{hdj} and \cite[Corollary~7.5]{hdj}.  
\end{proof} 
Now let $\Ybar$ be a smooth complex algebraic variety, and let $D$ be a normal crossings divisor on $\Ybar$. Let $\pi \colon \Xbar \to \Ybar$ be a nodal curve of genus $h>0$, assumed to be smooth over $Y = \Ybar \setminus D$, and put $X=\pi^{-1}Y$. Using Proposition \ref{learfromhdj} we will calculate several Lear extensions associated to sections of the family of jacobians $j \colon J \to Y$ associated to $X \to Y$ explicitly. Let $\bb$ denote the Poincar\'e bundle on $J$. The Lear extensions $\overline{\kappa^* \bb} $ and $\overline{\delta^*\bb}$ considered in Propositions \ref{kappalear} and \ref{Leardelta} below were calculated on the moduli stack of pointed stable curves by Hain in \cite[Theorems 10.2 and 11.5]{hain_normal}. Our results reproduce Hain's by a test curve argument. 

Assume that $\pi$ has a section $P$. Let $\omega$ denote the relative dualizing sheaf of $\pi$. As before we put $ \kappa_1=\pair{\omega,\omega}$ and $ \psi = \pair{\oo(P),\omega}=P^*\omega$. From Proposition~\ref{kappa} we recall the map $\kappa \colon Y \to J$ associated to the relative divisor $(2h-2)P-\omega$ on $X$ over~$Y$. 
\begin{prop} \label{kappalear} 
The Lear extension of $\kappa^*\bb$ over $\Ybar$ exists. If $\pi \colon \Xbar \to \Delta$ is a nodal curve over the unit disc, let $G$ be the weighted dual graph of the special fiber $\Xbar_0$. Assume that the section $P \colon \Delta \to \Xbar$ passes through $\Sm(\pi)$. Then one has an equality
\[ \overline{\kappa^*\bb} = 4h(h-1)\psi - \kappa_1 - \bar{g}((2h-2)x-K,(2h-2)x-K)[0] \]
of $\qq$-line bundles over $\Delta$. Here $K$ is the divisor on $G$ induced by $\omega$, and $x \in \V(G)$ is the irreducible component of $G$ where $P$ specializes.
\end{prop}
\begin{proof} The existence of the Lear extension of $\kappa^*\bb$ over $\Ybar$ follows from Theorem \ref{generalexistence} since $\kappa$ is a normal function section of $J \to Y$. Let $E=(2h-2)P-\omega$.
 By equation (\ref{delignepoinc}) we have a canonical isometry
\[ \kappa^*\bb \isom \pair{\oo(E)_\Ar,\oo(E)_\Ar}^{\otimes -1}  \]
of $C^\infty$ hermitian line bundles over $Y$. The formula is then obtained by applying Proposition \ref{learfromhdj}, noting that the equality 
\[ -\pair{\oo(E),\oo(E)}=-\pair{(2h-2)\oo(P)-\omega,(2h-2)\oo(P)-\omega} = 4h(h-1)\psi - \kappa_1  \]
holds over $\Delta$ by the adjunction isomorphism $ \pair{\oo(P),\oo(P)} \isom \pair{\oo(P),\omega}^{\otimes -1}$.
\end{proof}
Next assume that $\pi \colon \Xbar \to \Ybar$ has two sections $P$, $Q$. We recall from Proposition \ref{diagonal_metric} the section $\delta$ of $J \to Y$ associated to the divisor $P-Q$. 
\begin{prop} \label{Leardelta} 
The Lear extension of $\delta^*\bb$ over $\Ybar$ exists. If $\pi \colon \Xbar \to \Delta$ is a nodal curve over the unit disc, let $G$ be the weighted dual graph of the special fiber $\Xbar_0$, and assume that the sections $P, Q \colon \Delta \to \Xbar$ both pass through $\Sm(\pi)$. Assume that $P$ specializes onto $x \in V(G)$, and $Q$ specializes onto $y \in V(G)$. Let $r$ denote the effective resistance on the weighted dual graph $G$ of $\Xbar_0$. Then one has the equalities
\[ \begin{split} \overline{\delta^*\bb} & = 
-\pair{\oo(P-Q),\oo(P-Q)}- \bar{g}(x-y,x-y)[0]
\\
& =2 \, \pair{\oo(P),\oo(Q)} + \pair{\oo(P),\omega} + \pair{\oo(Q),\omega} - \bar{g}(x-y,x-y)[0] \\ 
& = 2 \, \pair{\oo(P),\oo(Q)} + \pair{\oo(P),\omega} + \pair{\oo(Q),\omega} - r(x,y)[0]  \end{split} \]
of $\qq$-line bundles on $\Delta$.
\end{prop}
\begin{proof} The existence of the Lear extension $\overline{\delta^*\bb}$ follows from Theorem \ref{generalexistence} since $\delta$ is a  normal function section of $J \to Y$. To obtain the first equality, we apply Proposition \ref{learfromhdj} to $\pi \colon \Xbar \to \Delta$ with $E = P-Q$. We obtain the second equality by expanding $\pair{\oo(P-Q),\oo(P-Q)}$ using the adjunction isomorphism $\pair{\oo(P),\oo(P)} \isom \pair{\oo(P),\omega}^{\otimes -1}$. The third equality follows from (\ref{greenandresistance}).
\end{proof}
Recall the $Y$-morphism $\delta_P \colon X \to J$ given by sending the point $z \in X_y$ to the Abel-Jacobi element $\int_z^P$ in $J_y$ for all $y \in Y$. Our next aim is to study the Lear extension of $\delta_P^*\bb$. From now on we assume therefore that $\Xbar$ is smooth. It is then automatic that the image of any section $P \colon \Ybar \to \Xbar$ is contained in the smooth locus $\Sm(\pi)$ of $\pi \colon \Xbar \to \Ybar$. Also we note that $\Xbar \setminus X$ is a normal crossings divisor on $\Xbar$.
\begin{prop} \label{deltalear} Assume that $\Xbar$ is smooth, and let $P \colon \Ybar \to \Xbar$ be a section of $\pi$. Then the Lear extension of $\delta_P^*\bb$ over $\Xbar$ exists. Let $\pi \colon \Xbar \to \Delta$ be a nodal curve over the unit disc, with $\Xbar$ smooth, and with $\pi$ smooth over $\Delta^*$. Let $G$ be the weighted dual graph of $\Xbar_0$ and let $r$ denote effective resistance on $G$. Then one has an equality
\begin{equation} \label{LearonXbar} \overline{\delta_P^*\bb}  = 2\,\oo(P) + \pi^*\pair{\oo(P),\omega} + \omega - \sum_{y \in \V(G)} r(x,y) \, y  
\end{equation}
of $\qq$-line bundles on $\Xbar$. Here $\omega$ is the relative dualizing sheaf of $\pi$, and $x \in \V(G)$ is the irreducible component of $\Xbar_0$ where $P$ specializes.
\end{prop}
\begin{proof}   Let $\Zbar = \Xbar \times_{\Ybar} \Xbar$. The first projection $\pi_1 \colon \Zbar \to \Xbar$ comes equipped with two natural sections, one called $\tilde{P}$ induced by $P$, and the diagonal section $Q$.  Let $\tilde{\bb}$ on $X \times_Y J$ denote the pullback of $\bb$  along the second projection $X \times_Y J \to J$. Let $\delta \colon X \to X \times_Y J$ be the section of the projection $X \times_Y J \to X$ obtained by taking the difference of the two sections $\tilde{P}$ and $Q$. Then we have a canonical isometry $\delta_P^*\bb \isom \delta^*\tilde{\bb}$.
The existence of the Lear extension $\overline{\delta_P^*\bb}$ follows then from Proposition \ref{Leardelta}. Consider the case that $\Ybar=\Delta$. Note that $\tilde{P}$ and the restriction of $Q$ to $\Sm(\pi)$ pass through the smooth locus of $\pi_1$.  Formula (\ref{LearonXbar}) then follows from the formula in Proposition \ref{Leardelta} using a test curve argument over $\Sm(\pi) \subset \Xbar$. There are two things to note. First, let $\pi_2 \colon \Zbar \to \Xbar$ be the second projection, and let $\tilde{\omega}=\pi_2^*\omega$ denote the relative dualizing sheaf of $\pi_1$. Then over $\Xbar$ we have canonical isomorphisms of line bundles
\[ \pair{ \oo(\tilde{P}),\oo(Q) } \isom \oo(P) \, , \quad
\pair{\oo(\tilde{P}),\tilde{\omega}} \isom \pi^*\pair{ \oo(P), \omega }      \, , \quad 
\pair{\oo(Q),\tilde{\omega}} \isom \omega \, , \]
where the Deligne pairings on the left hand sides are taken along the nodal curve $\pi_1 \colon \Zbar \to \Xbar$. Second, note that labelled dual graphs are functorial under the base change from $\Xbar \to \Ybar$ to $\Zbar \to \Xbar$, and next under the base change along any holomorphic map $f \colon \Delta \to \Xbar$. So if we let $f \colon \Delta \to \Xbar$ be a holomorphic map with $f(0)$ a smooth point of an irreducible component of $\Xbar_0$ then the labelled dual graph at the origin of the nodal curve $f^*\Zbar \to \Delta$ is canonically isomorphic to $G$.
\end{proof}

\section{Lear extension of a Deligne pairing} \label{delignelearII}

In this section we prove a technical result (Theorem \ref{thm:learextensiondeligne}) that allows one, under quite general conditions, to prove the existence of, and  to calculate precisely, the Lear extension of a Deligne pairing of two $C^\infty$ hermitian line bundles. We will see that, in general, the statement ``the Lear extension of the Deligne pairing is equal to the Deligne pairing of the Lear extensions'' does not hold true. This non-functoriality is closely related to the ``height jumping'' phenomenon near points of codimension two as discussed in \cite[Section~14]{hain_normal}. Our result gives a  quantitative approach to the height jump. 

Let $D$ be the divisor of $\Delta^n$ given by the equation $t_1=0$.
In this section we consider nodal curves $\pi \colon \Xbar \to \Delta^n$, such that $\Xbar$ itself is smooth, and such that $\pi$ is smooth over $Y=\Delta^n \setminus D$. We write $X = \pi^{-1}Y$ as usual. Let $\Sigma$ be the singular locus of $\pi$.  
We have that $\Xbar \setminus X =\pi^{-1}D$ is a normal crossings divisor of $\Xbar$, and $(\pi^{-1}D)^{\mathrm{sing}}=\Sigma$. The map $\Sigma \to D$ induced by $\pi$ is finite and unramified. Put $\Sigma_0=\Sigma \cap \Xbar_0$; this is a finite set.

Let $e \in \Xbar_0$ be a point. If $e \in \Sigma_0$, then locally around $e$ we can choose coordinates $(u,v,t_2,\ldots,t_n)$ such that the projection $\pi$ is given by $(u,v,t_2,\ldots,t_n)\mapsto (uv,t_2,\ldots,t_n)$. On the other hand if $e \notin \Sigma_0$, then locally around $e$ we can choose coordinates $(z,t_1,\ldots,t_n)$ such that the projection $\pi$ is given by $(z,t_1,\ldots,t_n) \mapsto (t_1,\ldots,t_n)$.
For each $\eps \in \rr_{>0}$ and for each $e \in \Xbar_0$ we let $U_{e,\eps}$ denote the following open neighborhood of $e$ in $\Xbar$. If $e \in \Sigma_0$ we let $U_{e,\eps}$ be given by the conditions $|u|, |v| < \eps^{1/2}$, $|t_i|<\eps$ for $i=2,\ldots,n$. If $e \notin \Sigma_0$ we let $U_{e,\eps}$ be given by the conditions $|z|<\eps$, $|t_i|<\eps$ for $i=1,\ldots,n$. 

For $\eps \in \rr_{>0}$ small enough there exists a finite set $\mathcal{U}$ of mutually disjoint open neighborhoods $U_{e,\eps}$ with centers $e \in \Xbar_0$ such that (i) $\mathcal{U}$ contains all $U_{e,\eps}$ with $e \in \Sigma_0$, and (ii) the elements of $\mathcal{U}$ together cover $\pi^{-1}\Delta_\eps^n$ up to a subset of Lebesgue measure zero. We call such a set $\mathcal{U}$ a distinguished collection of open neighborhoods associated to $\eps$. 

Let $L$ be a $C^\infty$ hermitian line bundle on $X$, and let $\ell$ be a non-zero rational section of $L$. Assume that $L$ has a Lear extension $\Lbar$ over $\Xbar$. Let $V$ denote the set of irreducible components of $\pi^{-1}D$. There are well-defined rational numbers $a(x)$ for each $x \in V$ such that for $\ell$ viewed as a rational section of $\Lbar$, the equality 
\[ \divisor_{\Xbar} \ell = \overline{\divisor_X \ell} + \sum_{x \in V} a(x)  x \]
holds in $\Pic(\Xbar) \otimes \qq$. Following Proposition \ref{learchar}, this equality is to be interpreted as saying that on a coordinate chart $U$ of $\Xbar$ with center $p$ in the smooth locus of $x \in V$, adapted to $\pi^{-1}D$ and disjoint from $\overline{\divisor_X \ell}$ the function $\log \|\ell\| - a(x) \log |t_1|$ on $U \setminus \pi^{-1}D$ extends continuously over $U$. 

Let $\eps \in \rr_{>0}$. Assume a distinguished collection $\mathcal{U}$ of open neighborhoods associated to $\eps$ is given. Let $U \subset \Xbar$ denote the union of all elements of $\mathcal{U}$. Then we define a logarithmic current $\chi$ on $U \setminus \pi^{-1}D$ associated to $\ell$ as follows. On an $U_{e,\eps}$ with $e \notin \Sigma_0$ we put $\chi=\log \|\ell\| - a(x) \log |t_1|$, if $x$ is the unique irreducible component of $\pi^{-1}D$ such that $e \in x$. On an $U_{e,\eps}$ with $e \in \Sigma_0$
we put $\chi=\log \|\ell\| - a(x)\log|u|-a(y)\log|v|$, where $x$ is the irreducible component of $\pi^{-1}D$ corresponding to the branch through $e \in U_{e,\eps}$ given by the equation $u=0$, and where $y$ 
is the irreducible component of $\pi^{-1}D$ corresponding to the branch through $e \in U_{e,\eps}$ given by the equation $v=0$. We note that $\chi$ extends as a logarithmic current over $U \setminus (\pi^{-1}D)^\sing= U \setminus \Sigma$, with locus of indeterminacy given by 
$\overline{\divisor_X \ell}$.

We say that a function $\gamma \colon \Delta^n_\eps \setminus D \to \rr$ has a log singularity along $D$ if there exists $a \in \qq$ such that the function $\gamma - a \log|t_1|$ extends continuously over $\Delta_\eps^n$. In this case we call $a$ the order of the log singularity of $\gamma$ along $D$.                

\begin{thm} \label{thm:learextensiondeligne} Let $\pi \colon \Xbar \to \Delta^n$ be a nodal curve with $\Xbar$ smooth, such that $\pi$ is smooth over $Y=\Delta^n \setminus D$, where $D$ is the divisor of $\Delta^n$ given by the equation $t_1=0$. Let $X = \pi^{-1}Y$. Let $L, M$ be $C^\infty$ hermitian line bundles on $X$ and let $\ell, m$ be non-zero rational sections of $L, M$ such that $\overline{\divisor_X \ell}$ and $\overline{\divisor_X m}$ are disjoint from the singular locus $\Sigma$ of $\pi$. Let $\eps \in \rr_{>0}$ and let $\mathcal{U}=\{U_{e,\eps} \}$ be a distinguished collection of open neighborhoods on $\pi^{-1}\Delta_\eps^n$ with centers $e \in \Xbar_0$ associated to $\eps$. Let $\chi$ denote the current associated to $\ell$ as above.
Assume the following conditions hold: 
\begin{enumerate}
\item[(a)] $L, M$ have Lear extensions $\Lbar, \Mbar$ over $\Xbar$; 
\item[(b)] for each fiber $F$ of $\pi^{-1}D$ over $D$ the following holds. 
The first Chern current $\ceeone(\Mbar)$ of the continuous hermitian line bundle $\Mbar$ on $\mathrm{Sm}(\pi)$ restricts as a smooth $(1,1)$-form on $F \setminus F^{\mathrm{sing}} \subset \mathrm{Sm}(\pi)$. Moreover, the smooth $(1,1)$-form $\ceeone(\Mbar)|_{F \setminus F^{\mathrm{sing}}}$ extends as a 
smooth $(1,1)$-form over the normalization of $F$;
\item[(c)] the function
\[ \gamma \colon \Delta_\eps^n \setminus D \to \rr \, , \quad t \mapsto \sum_{e \in \Sigma_0} \int_{X_t \cap U_{e,\eps}} \chi \, \ceeone(M)  \] 
has a log singularity along $D$. 
\end{enumerate}
Then $\pair{L,M}$ has a Lear extension $\overline{\pair{L,M}}$ over $\Delta_\eps^n$. Assume that $n=1$, and let $c \in \qq$ be the order of $\gamma$ along $[0]$. Then the equality 
\[  \overline{\pair{L,M}} = \pair{\Lbar,\Mbar} + c  [0]  \]
of $\qq$-line bundles holds on $\Delta_\eps$. 
\end{thm}
\begin{proof} We note that up to adding a bounded continuous function on $\Delta_\eps^n$, the integral in (c) remains unchanged upon replacing $\ell$ by $f \ell$ or $m$ by $gm$, where $f$, $g$ are nonzero rational functions on $\Xbar$. Using this freedom to change $\ell$, $m$, and shrinking $\Delta_\eps$ if necessary  we may assume in addition that $\overline{\divisor_X \ell}$ and $\overline{\divisor_X m}$ have disjoint support.  Then $\ell, m$ give rise to a generating section $\pair{\ell,m}$ of $\pair{L,M}$, with log norm
\[ \log \| \pair{\ell,m} \| = (\log \| m \|)[\divisor \ell] + \int_{X_t} \log \|\ell\| c_1(M)  \]
by equation (\ref{defmetricpairing}).
By Proposition \ref{learchar} our task is to show that the restriction of $\log\|\pair{\ell,m}\|$ to $\Delta_\eps^n \setminus D$ has a log singularity along $D$.  The assumptions that $L, M$ have Lear extensions over $\Xbar$ and that $\overline{\divisor_X \ell}$ and $\overline{\divisor_X m}$ are disjoint and disjoint from $\Sigma$ imply that $(\log \| m \|)[\divisor \ell] $ has a log singularity along $D$. We are thus reduced to showing that the function 
\[ \Delta_\eps^n \setminus D \to \rr \, , \quad t \mapsto \int_{X_t}  \log \|\ell\| c_1(M)  \]
has a log singularity along $D$. Let $R$ denote the set of $e \in \Xbar_0$ such that $e$ occurs as a center of one of the given distinguished open neighborhoods associated to $\eps$. Recall that $\Sigma_0 \subset R$. As the elements of $\mathcal{U}$ together cover $\pi^{-1}\Delta_\eps^n$ up to a subset of Lebesgue measure zero we can write
\begin{equation} \label{bigsplit}  \int_{X_t}  \log \|\ell\|  \, \ceeone(M)  = \sum_{e \in R} \int_{X_t \cap U_{e,\eps}}  \chi \, \ceeone(M) + \sum_{e \in R} \int_{X_t \cap U_{e,\eps}} \left( \log \|\ell\| - \chi \right)\, \ceeone(M) 
\end{equation}
for each $t \in \Delta_\eps^n \setminus D$.
In the first sum on the right hand side, when $e \notin \Sigma_0$, the integrand is a bounded continuous current on $U_{e,\eps}$, and the integral extends as a bounded continous function over $\Delta_\eps^n$. Hence for our purposes we may replace the first sum by the restricted sum
\[ \sum_{e \in \Sigma_0} \int_{X_t \cap U_{e,\eps}} \chi \, \ceeone(M) \]
over centers $e \in \Sigma_0$. This sum has a log singularity along $D$ by assumption (c). 

As to the second sum on the right hand side, Picard-Lefschetz theory implies (see for example \cite[Lemma X.9.19]{acg}) that for $\eps$ sufficiently small there exists a deformation retraction $r \colon \Xbar \to \Xbar_0$ such that for each $t \in \Delta_\eps^n \setminus D$ and each $p \in \Sigma_0$ the intersection $r^{-1}(p) \cap X_t$ is a smoothly embedded circle of $X_t$. Let $V$ denote the set of irreducible components of $\pi^{-1}D$. For each $x \in V$ let $U_x=r^{-1}(x)$. Then for each $x \in V$ we have a real analytic submersion $U_x \to \Delta_\eps^n\setminus D$.
We can thus estimate
\[  \sum_{e \in R} \int_{X_t \cap U_{e,\eps}}  \left( \log \|\ell\| - \chi \right) \ceeone(M)  \sim
\sum_{x \in V} \int_{X_t \cap U_{x}} \left( \log \|\ell\| - \chi \right) \ceeone(M)     \]
for each $t \in \Delta_\eps^n \setminus D$.
Viewing $\ell$ as a rational section of $\Lbar$ and writing
\[ \divisor_{\Xbar} \ell = \overline{\divisor_X \ell} + \sum_{x \in V} a(x)  x \]
we recall that for each $x \in V$ the equality $\log \|\ell\| - \chi = a(x) \log |t_1|$ holds over $U_{x}$. This gives the equality
\[ \int_{X_t \cap U_{x}} \left( \log \|\ell\| - \chi\right)   \ceeone(M) = a(x) \log|t_1| \int_{X_t \cap U_{x}} \ceeone(M) \, . \]
By assumption (b) $\ceeone(M)$ extends as a smooth $(1,1)$-form over the normalization of each fiber of $\pi^{-1}D$ over $D$. This gives, putting $F_x = x \cap \Xbar_0$, the estimate
\[ \int_{X_t \cap U_{x}} \ceeone(M) \sim (F_x \cdot \Mbar)   \]
as $t \to 0$. 
We obtain that
\[  \sum_{e \in R} \int_{X_t \cap U_{e,\eps}} \left( \log \|\ell\| - \chi \right) \ceeone(M)  \sim
   \log|t_1| \sum_{x \in V}  a(x) (F_x \cdot \Mbar) \, ,  
 \]
hence the second sum on the right hand side in (\ref{bigsplit}) has a log singularity of order $ \sum_{x \in V}  a(x) (F_x \cdot \Mbar) $
along $D$. We conclude that 
\begin{equation} \label{generalestimate} \int_{X_t} \log \|\ell\|c_1(M) \sim 
\sum_{e \in \Sigma_0} \int_{X_t \cap U_{e,\eps}} \chi \, \ceeone(M) + \sum_{x \in V} a(x) (F_x \cdot \Mbar) \log |t_1|  
\end{equation} 
has a log singularity along $D$, and hence that the Lear extension of $\pair{L,M}$ exists. 

In order to compute the Lear extension of $\pair{L,M}$ explicitly we make the above argument more precise, under the additional assumption that the base is a small disc $\Delta_\eps$.
We observe first of all that $\pair{\Lbar,\Mbar}$ is a line bundle extending $\pair{L,M}$ over $\Delta$. Following Proposition \ref{learchar} we have that $ \overline{\pair{L,M}} = \pair{\Lbar,\Mbar} + c \, [0]$ if and only if $\log\|s\| \sim c \log |t|$ as $t \to 0$ for a generating section $s$ of $\pair{\Lbar,\Mbar}$. The latter condition is equivalent with
\[ \log \| \pair{\ell,m} \| \sim \left( \ord_0 \pair{\ell,m} + c \right) \, \log|t| \]
as $t \to 0$. We will compute the value of $c$ from this estimate. Let $G$ be the dual graph of $\Xbar_0$. Viewing $\ell, m$ as rational sections of $\Lbar, \Mbar$ we may write
\[ \divisor \ell = \overline{\divisor_X \ell} + \sum_{x \in \V(G)} a(x) \, x \, , \quad
\divisor m = \overline{\divisor_X m} + \sum_{y \in \V(G)} b(y) \, y  \]
for suitable $a(x), b(y) \in \qq$. Still assuming that the supports of $\overline{\divisor_X \ell}$ and $\overline{\divisor_X m}$ do not meet, we find that the local intersection product $\left(\divisor \ell \cdot \divisor m \right)_0$ of $\divisor \ell$ and $\divisor m$ above the origin satisfies
\[ \begin{split} 
\left(\divisor  \ell \cdot  \divisor m \right)_0  & =  \overline{\divisor_X m}  \cdot\sum_{x \in \V(G)}   a(x) x  + \overline{\divisor_X \ell}\cdot \sum_{y \in \V(G)}   b(y) y \\ & \hspace{1cm} + \sum_{x,y \in \V(G)} a(x) b(y) (x\cdot y) \, . 
\end{split} 
\]
We have
\[  ( \overline{\divisor_X m}\cdot \sum_{x \in \V(G)}   a(x) \,x ) \log|t| \sim (\log\|\ell\|)[\divisor m] \, , \]
and similarly
\[   ( \overline{\divisor_X \ell}\cdot \sum_{y \in \V(G)} b(y) \, y ) \log|t| \sim (\log\|m\|)[\divisor \ell] \, . \]
As moreover
\[ \ord_0 \pair{\ell,m} = \left( \divisor \ell \cdot \divisor m \right)_0   \]
we find that
\[ \begin{split} \ord_0 \pair{\ell,m}\log|t|  \sim (\log\|\ell\|) & [\divisor m]  +  (\log\|m\|)[\divisor \ell] \\ & + \sum_{x,y \in \V(G)} a(x) b(y) (x\cdot y) \log|t| \, . \end{split} \]
Recalling that
\[ \log \| \pair{\ell,m} \| = (\log \| m \|)[\divisor \ell] + \int_{X_t} \log \|\ell\|c_1(M)  \]
it follows that if $ \overline{\pair{L,M}} = \pair{\Lbar,\Mbar} + c \, [0]$ then
\[ \begin{split} c \log|t| & \sim \log \| \pair{\ell,m} \| - \ord_0 \pair{\ell,m}\log|t| \\
 & \sim \int_{X_t} \log \|\ell\|c_1(M) - (\log\|\ell\|)[\divisor m] 
- \sum_{x,y \in \V(G)} a(x) b(y) (x\cdot y) \log|t| \\
 & \sim \int_{X_t} \log \|\ell\|c_1(M) - \sum_{x \in V(G)} a(x) (  \overline{\divisor_X m} \cdot x) \log|t| \\ & \hspace{1cm} - \sum_{x,y \in \V(G)} a(x) b(y) (x\cdot y) \log|t| \\
 & \sim \int_{X_t} \log \|\ell\|c_1(M) - \sum_{x \in V(G)} a(x) (\Mbar \cdot x) \log |t| \, . 
\end{split} \]
From (\ref{generalestimate}) we obtain that
\[ \int_{X_t} \log \|\ell\|c_1(M) \sim 
\sum_{e \in \Sigma_0} \int_{X_t \cap U_{e,\eps}} \chi \, \ceeone(M) + \sum_{x \in V(G)} a(x) (\Mbar \cdot x) \log |t| \, . \] 
We conclude that
\[ c \log |t| \sim \sum_{e \in \Sigma_0} \int_{X_t \cap U_{e,\eps}} \chi \, \ceeone(M) \, . \]
This finally proves our formula for $\overline{\pair{L,M}}$.
\end{proof}

\section{Main auxiliary result} \label{mainaux}

Let $\Ybar$ be a smooth complex algebraic variety, or a polydisc, and let $D$ be a normal crossings divisor on $\Ybar$. Write $Y = \Ybar \setminus D$. Let $\pi \colon \Xbar \to \Ybar$ be a nodal curve, assumed to be smooth over $Y \subset \Ybar$. We write $X=\pi^{-1}Y$ as usual. Let $j \colon J \to Y$ be the associated jacobian fibration. Associated to any section $P \colon \Ybar \to \Xbar$, we have an Abel-Jacobi map $\delta_P \colon X \to J$. The family $J \to Y$ carries a canonical Poincar\'e bundle $\bb$ with a canonical $C^\infty$ hermitian metric. Let $\omega$ denote the relative dualizing sheaf of $\pi$.
As before we put $ \kappa_1=\pair{\omega,\omega}$ and $ \psi = P^*\omega = \pair{\oo(P),\omega}$. 
The purpose of this section is to prove the following main auxiliary result.
\begin{thm} \label{leardeltasq} The Lear extension of $\pair{\delta_P^*\bb,\delta_P^*\bb}$ over $\Ybar$ exists. In the case that $\Ybar$ is the unit disc, assume that $P$ passes through the smooth locus of $\pi$, and let  $(\Gamma,K)$ be the polarized metrized graph associated to the dual graph $G$ of $\Xbar_0$. Denote by $g_\mu$ the admissible Green's function of $(\Gamma,K)$, and by $\vareps$ its epsilon-invariant. Then  we have 
\[ \overline{\pair{\delta_P^*\bb,\delta_P^*\bb}} = 4h\, \psi + \kappa_1 - 
\left(4h \, g_\mu(x,x) + \varepsilon \right) [0] \, , \]
where $x \in V(G)$ is the irreducible component of $\Xbar_0$ to which $P$ specializes. 
\end{thm}

We will derive Theorem \ref{leardeltasq} from Propositions \ref{selfdelignepairingLear} and \ref{maintechnical} below. 
\begin{prop} \label{selfdelignepairingLear}
Let $\pi \colon \Xbar \to \Delta$ be a nodal curve over the unit disc, smooth over $\Delta^*$, and assume that the total space $\Xbar$ is smooth. Also assume that a section $P$ of $\pi$ is given, which then necessarily passes through $\Sm(\pi)$. Then the Lear extension of $\delta_P^*\bb$ exists over $\Xbar$. Let $(\Gamma,K)$ be the polarized metrized graph associated to the dual graph $G$ of $\Xbar_0$. Denote by $g_\mu$ the admissible Green's function of $(\Gamma,K)$, by $\vareps$ the epsilon-invariant of $(\Gamma,K)$, and by $\eta$ the eta-invariant of $\Gamma$ as defined in (\ref{eta}).
The Deligne pairing of $\overline{\delta_P^*\bb}$ with itself satisfies
\[ \pair{\overline{\delta_P^*\bb},\overline{\delta_P^*\bb}} = 4h \, \psi + \kappa_1 - (4h \, g_\mu(x,x) + \varepsilon -\eta ) [0]  \]
where $x \in V(G)$ is the irreducible component of $\Xbar_0$ to which $P$ specializes. 
\end{prop}
\begin{prop} \label{maintechnical} Assume that $\pi \colon \Xbar \to \Ybar=\Delta^n$ is a nodal curve, with $\Xbar$ smooth, and with $\pi$ smooth over $Y=\Delta^*\times\Delta^{n-1}$. Assume a section $P$ of $\pi$ is given, which then necessarily passes through the smooth locus of $\pi$. Then the Lear extension of $\pair{\delta_P^*\bb,\delta_P^*\bb}$ over $\Ybar$ exists. Assume that $n=1$ and let $\eta$ denote the eta-invariant (\ref{eta}) of the polarized metrized graph associated to $\Xbar_0$. Then  the equality
\begin{equation}
\label{learextformula} \overline{\pair{\delta_P^*\bb,\delta_P^*\bb}} = \pair{\overline{\delta_P^*\bb},\overline{\delta_P^*\bb}} -\eta \, [0]  
\end{equation}
holds.
\end{prop}
\begin{proof}[Proof of Proposition \ref{selfdelignepairingLear}]  From Proposition \ref{deltalear} we obtain that $\overline{\delta_P^*\bb}$ exists and that
\begin{equation} \label{zero} \overline{\delta_P^*\bb} = 2\,\oo(P) + \pi^* \pair{\oo(P),\omega} + \omega  - \sum_{y \in \V(G)} r(x,y)y \, . 
\end{equation} 
As $\Xbar$ is smooth, all edges in $G$ have weight one, and the (standard matrix of the) discrete Laplacian $L \colon \rr^{\V(G)} \to \rr^{\V(G)}$ is equal to minus the intersection matrix of $\Xbar_0$. We then obtain by direct calculation from (\ref{zero}) that
\begin{equation} \label{selfint} \pair{\overline{\delta_P^*\bb}, \overline{\delta_P^*\bb}   }
= 4h\, \psi + \kappa_1 - (2 \,r(x,K)+ \sum_{y,z \in \V(G)} r(x,y)r(x,z)L(y,z) )[0] \, . 
\end{equation}
Denote by $\psi \colon \rr^{\V(G)} \to C(\Gamma)^*$ the canonical map. Let $\nu$ denote the discrete measure supported on $V$ with mass $\nu(v)=\sum_{e \in E(v)} F(e)$ at the vertex $v$. Then by (\ref{discretelaplaceresistance}) we have
\[ \psi(\sum_{z \in \V(G)} r(x,z)L(y,z)) = 2\,(\mu_{\can}^{\mathrm{dis}}(y) - \delta_x(y)) + \nu(y) \, .  \]
It follows that
\[ \begin{split} \sum_{y,z \in \V(G)} r(x,y)r(x,z)L(y,z) & = \int_\Gamma r(x,y) (2\,( \mu_{\can}^{\mathrm{dis}}(y) - \delta_x(y) )+\nu(y)) \\
& =  \int_\Gamma r(x,y) (2\,\mu_{\can}^{\mathrm{dis}}(y) + \nu(y)) \\
 & = 4\,\tau(\Gamma)-4\,\tau^{\mathrm{cts}}(\Gamma,x) + \int_\Gamma r(x,y)\nu(y)   \, . \end{split} \]
Invoking Proposition \ref{tauandeta} we deduce from this that
\begin{equation} \label{one} \sum_{y,z \in \V(G)} r(x,y)r(x,z)L(y,z) = 4\,\tau(\Gamma)-\eta(\Gamma) \, .   
\end{equation}
By Proposition \ref{taualternative} we have
\begin{equation} \label{two} 2 \,r(x,K)  + 4\,\tau(\Gamma) = 
 4h \, g_\mu(x,x)+\varepsilon(\Gamma,K)  \, . 
\end{equation}
Combining (\ref{one}) and (\ref{two}) we find
\[ 2\,r(x,K) + \sum_{y,z \in \V(G)} r(x,y)r(x,z)L(y,z)
= 4h\,g_\mu(x,x) + \varepsilon(\Gamma,K) -\eta(\Gamma) \, .  \]
We obtain the proposition by inserting this into (\ref{selfint}).
\end{proof}
Our proof of Proposition \ref{maintechnical} consists in an application of Theorem \ref{thm:learextensiondeligne} above, which gives a general criterion for the existence of a Lear extension of a given Deligne pairing, as well as an explicit formula for it once it exists. In particular we will need to control the behavior of the log norm (the archimedean height) of a section of $\delta_P^*\bb$ on coordinate neighborhoods $U_{e,\eps}$ as in the Theorem. 

From the formula in Proposition \ref{normbiext} we see that it suffices to control the maps $\Im \Omega$ and $\Im \delta_P$ on $U_{e,\eps}$. We will use the general expansions of the period and Abel-Jacobi map in local coordinates from Section \ref{asympresults}. Expansions for $\Im \Omega$ and $\Im \delta_P$ follow readily, and we combine them carefully to obtain suitable expansions for the log norm of a section of $\delta_P^*\bb$.

A general approach to local expansions in several variables for the archimedean height is explained in \cite{hp}. It is based on the several variables $SL_2$-orbit Theorem from \cite{knu} \cite{pearldiff}. We also mention \cite{abbf} and \cite{bhdj}, where similarly the asymptotic behavior of the norm on the Poincar\'e bundle is studied, over a several variables parameter space. 
\begin{proof}[Proof of Proposition \ref{maintechnical}]
We write $\vv=\R^1 \pi_* \zz_X(1)$, viewed as a variation of polarized Hodge structures over $Y$ with polarization given by the intersection product. Let $h$ be the genus of the general fiber of $\pi$ and denote by $U_h$ the Siegel upper half-space in degree $h$. Let $\Omega \colon Y \to \Gamma \setminus U_h$ with $\Gamma \cong \zz$ be the associated period map obtained by taking normalized period matrices on a symplectic frame of $\vv$.

Let $\Sigma \subset \Xbar$ be the singular locus of $\pi$ and put $\Sigma_0=\Sigma \cap \Xbar_0$. By \cite[Proposition 9.1.11]{liu} we may choose a non-zero rational section $s$ of $\delta_P^*\bb$ such that $\overline{\divisor_X  s}$ is disjoint from $\Sigma$.  We will apply Theorem \ref{thm:learextensiondeligne} with both $L$, $M$ equal to $\delta_P^*\bb$ and with both $\ell$, $m$ equal to $s$. 
Thus, let $\eps \in \rr_{>0}$ and let $\mathcal{U}=\{U_{e,\eps} \}_{e \in R}$ be a distinguished collection of open coordinate neighborhoods with centers $e \in \Xbar_0$ associated to $\eps$. By our assumption that $\overline{\divisor_X  s}$ is disjoint from $\Sigma$ we can assume that the $U_{e,\eps}$ are small enough so that $s$ is generating over $U_{e,\eps} \setminus \pi^{-1}D$ for all $e \in \Sigma_0$.

Our task is to verify that for sufficiently small $\eps$ each of the three conditions (a)--(c) of Theorem \ref{thm:learextensiondeligne} is satisfied. As to condition (a), the existence of a Lear extension $\overline{\delta_P^*\bb}$ of $\delta_P^*\bb$ follows directly from Proposition \ref{deltalear}. 

As to condition (b), let $R$ be the set of $e \in \Xbar_0$ such that $e$ occurs as a center of one of the given open neighborhoods associated to $\eps$. 
We first work locally on a distinguished open neighborhood $U_{e,\eps}$ with a center $e \in R \setminus \Sigma_0$. Then over $U_{e,\eps}\setminus \pi^{-1}D$ the Poincar\'e-Lelong formula gives us
\begin{equation} \label{lelong} \ceeone(\delta_P^*\bb) = \frac{1}{\pi i} \deldelbar \log \|s\| + \delta_{\overline{\divisor_X s}} \, . 
\end{equation} 
In order to study $\ceeone(\delta_P^*\bb)$ and its behavior near $\pi^{-1}D$ we develop $\log\|s\|$ in the coordinates $(z,t_1,\ldots,t_n)$ of $U_{e,\eps}$. Following Proposition \ref{normbiext}, this requires developing both $(\Im \Omega)^{-1}$ and $\Im \delta_P$ in these coordinates. We will turn to this task first. 

Proposition \ref{asymptperiod}(a) yields an integer $r \in \zz_{\geq 0}$ as well as a matrix  
\begin{equation*} A = \begin{pmatrix}[c|c]  A' & 0 \\ \hline 0 & 0  \end{pmatrix} \in M(h \times h, \zz) \end{equation*}
with $A'\in M(r \times r,\zz)$ symmetric positive definite together with a holomorphic function $\psi \colon \Delta_\eps^n \to S(h \times h,\cc)$ such that
\[ \Omega(t) = \frac{A}{2\pi i} \log t_1 + \psi(t)  \]
for all $t \in \Delta_\eps^n \setminus D$.
This gives for the imaginary part of the period matrix that
\begin{equation} \label{periods} 
-2\pi\,\Im \Omega(t) = A \log|t_1| + B(t)  
\end{equation}
away from $D$, where $B=-2\pi\Im \psi$. The function $B(t)$ is bounded and continuous on $\Delta_\eps^n$.
Writing
\[ B =  \begin{pmatrix}[c|c] B_{11}& B_{12}\\
\hline 
B_{21} & B_{22}\\
\end{pmatrix} \, ,  \]
with $B_{11} \in M(r \times r,\rr)$, the positive definiteness of $\Im \Omega(t)$ for each $t \in \Delta_\eps^n \setminus D$ implies that
the matrices $B_{22}(t)$ are invertible away from $D$. Moreover, we have that the Schur complements 
\[ Q= A'\log|t_1| + B_{11}-B_{12}B_{22}^{-1}B_{21}  \]
of $B_{22}$ in $-2\pi\Im \Omega$ are invertible away from $D$ as well.  
The inverse of $A\log|t_1|+B=-2\pi\,\Im \Omega$ can be written as
\begin{equation} \label{inverse} (A \log|t_1| + B)^{-1} = \begin{pmatrix}[c|c] Q^{-1} & -Q^{-1}B_{12}B_{22}^{-1} \\
\hline 
-B_{22}^{-1}B_{21}Q^{-1} & B_{22}^{-1} + B_{22}^{-1}B_{21}Q^{-1}B_{12}B_{22}^{-1} \\
\end{pmatrix} \, .  
\end{equation}
Making $\eps$ smaller if necessary, we can write, using the boundedness of $B(t)$,
\begin{equation} \label{Qmatrix} Q^{-1} = \frac{1}{\log|t_1|}A'^{-1} + \frac{1}{(\log|t_1|)^2}\delta(t) 
\end{equation} 
with $\delta \colon \Delta_\eps^n \to M(r \times r, \rr)$ a bounded continuous function. Formulae (\ref{inverse}) and (\ref{Qmatrix}) taken together yield an expansion of $(\Im \Omega)^{-1}$ that will be sufficiently precise for our purposes.

We next turn our attention to $\Im \delta_P$. By Proposition \ref{asymptAJ}(a) we have
\[ \delta_P = Ab \frac{\log t_1}{2\pi i} + \alpha \]
for some holomorphic function $\alpha \colon U_{e,\eps} \to \cc^h$ and $b \in \qq^h$. It follows that 
\[ -2\pi \Im \delta_P = Ab \log|t_1| + a  \]
where $a = -2\pi \Im \alpha$.

By Proposition \ref{normbiext} we have
\begin{equation} \label{lognormsection} \begin{split}
\log\|s\| & = \log|f|-2\pi (\Im \delta_P)^t (\Im \Omega)^{-1} (\Im \delta_P) \\
&= \log|f| + (Ab\log|t_1|+a)^t(A \log|t_1|+B)^{-1}(Ab\log|t_1|+a) 
\end{split} 
\end{equation}
for some meromorphic $f$ on $U_{e,\eps}$. In order to expand (\ref{lognormsection}) we write
\[ b=  {\genfrac{(}{)}{0pt}{}{b_1}{b_2}}  \, , \quad
a= {\genfrac{(}{)}{0pt}{}{a_1}{a_2}} \]
with $b_1 \in \qq^r$, $b_2 \in \qq^{h-r}$, $a_1 \colon U_{e,\eps} \to \rr^r$, $a_2 \colon U_{e,\eps} \to \rr^{h-r}$.  
Using equations (\ref{inverse}) and (\ref{Qmatrix}) we find that there exist unique bounded continuous functions $\sigma, \gamma \colon U_{e,\eps} \to \rr$  such that the equality
\[ \log \|s\| = \log|f|+b_1^tA'b_1 \log|t_1| + \sigma + \frac{1}{\log|t_1|} \gamma  \]
holds on $U_{e,\eps}$. An explicit calculation yields that 
\begin{equation} \label{expandsigma} \sigma = a_2^t B_{22}^{-1} a_2 + 2\, a_1^tb_1-2\, a_2^tB_{12}B_{22}^{-1}b_1 + b_1^tA'\delta A' b_1  \, . 
\end{equation}

For $(1,1)$-forms $\omega$, $\omega'$ on $U_{e,\eps} \setminus \pi^{-1}D$ we write $\omega \sim_\pi \omega'$ if, upon writing in local coordinates $(z,t_1,\ldots,t_n)$ 
\[ \omega - \omega' = g\, \d z \, \d \overline{z} + \sum_{i=1}^n h_i \d z \, \d \overline{t}_i + \sum_{i=1}^n k_i \, \d \overline{z} \, \d t_i + \sum_{i,j=1}^n l_{ij} \, \d t_i \, \d \overline{t}_j \, , \]
we have $g=0$. If $\omega \sim_\pi \omega'$ then for all $t \in \Delta_\eps^n\setminus D$ the restrictions of $\omega$ and $\omega'$ to $X_t \cap U_{e,\eps}$ are equal. 

Combining (\ref{lelong}) and (\ref{lognormsection}) and using that
\[ \frac{1}{\pi i}\deldelbar \log|f| + \delta_{\overline{\divisor_X s}} = 0  \quad \textrm{and} \quad \frac{1}{\pi i}\deldelbar \log|t_1|=0 \]
on $U_{e,\eps} \setminus \pi^{-1}D$ we find that on $U_{e,\eps} \setminus \pi^{-1}D$ the equivalence
\[ \ceeone(\delta_P^*\bb)
\sim_\pi \frac{1}{\pi i} \deldelbar \sigma + \frac{1}{\log|t_1|} \frac{1}{\pi i} \deldelbar \gamma  \]
holds. As $\alpha$ is holomorphic, the entries of $\Im \alpha$ and hence of $a_1, a_2$ are annihilated by $\deldelbar$. With this it follows from (\ref{expandsigma}) that  
\[  \deldelbar \, \sigma \sim_\pi \deldelbar \, a_2^t B_{22}^{-1} a_2 \, . \] 
We conclude that the equivalence
\[ \ceeone(\delta_P^*\bb) \sim_\pi \frac{1}{\pi i} \deldelbar  \,
a_2^t B_{22}^{-1} a_2 + \frac{1}{\log|t_1|} \frac{1}{\pi i} \deldelbar \, \gamma \] 
holds away from $\pi^{-1}D$. Now the term $(1/\log|t_1|)\frac{1}{\pi i} \deldelbar \gamma$ converges uniformly to zero as $t_1 \to 0$. Let $p=(0,t_2,\ldots,t_n) \in D_\eps = \Delta_\eps^n \cap D$. 
  We find that for $t \to p$ the family of smooth $(1,1)$-forms $c_1(\delta_P^*\bb)_t$ restricted to $U_{e,\eps}\setminus \pi^{-1}D$ has a well-defined limit on $\Xbar_p \cap U_{e,\eps}$, namely the smooth $(1,1)$-form
\begin{equation} \label{limitform} \begin{split} \frac{1}{\pi i} \deldelbar_z  & \,a_2(z,p)^t B_{22}^{-1}(p) a_2(z,p) \\ & = \frac{-1}{\pi i} \deldelbar_z \, 2\pi (\Im \alpha_2)(z,p) (\Im \psi_{22})^{-1}(p) (\Im \alpha_2)(z,p)
\, . \end{split}
\end{equation}

Let $G$ be the dual graph of the fiber $F=\Xbar_p$ of $\pi$ above $p$. 
Let $x \in V(G)$ be the unique irreducible component of $\Xbar_0$ that intersects $U_{e,\eps}$ non-trivially, and let $q(x)$ be the genus of the normalization of $x$. We claim that (\ref{limitform}) coincides with the restriction to $x \cap U_{e,\eps}$ of $2\,q(x)$ times the Arakelov volume form (\ref{defArakvolume}) of the normalization of $x$. Varying $e$ through $\R \setminus \Sigma_0$ and varying $p$ through $D_\eps$ this claim then shows in particular that condition (b) is verified: let $F$ be any fiber of the projection $\pi^{-1}D \to D$. Then the first Chern current $\ceeone(\overline{\delta_P^*\bb})$ of $\overline{\delta_P^*\bb}$ on $\mathrm{Sm}(\pi)$ restricts as a smooth $(1,1)$-form on $F \setminus F^{\mathrm{sing}} \subset \mathrm{Sm}(\pi)$, and the smooth $(1,1)$-form $\ceeone(\overline{\delta_P^*\bb})|_{F \setminus F^{\mathrm{sing}}}$ extends as a smooth $(1,1)$-form over the normalization of $F$.

In order to prove our claim, we invoke Propositions \ref{asymptperiod}(b) and \ref{asymptAJ}(b). From Proposition \ref{asymptperiod}(b) we obtain that $\psi_{22}(p)$ is the period matrix of the jacobian $\Jac(\tilde{\Xbar}_p)$ of the normalization $\tilde{\Xbar}_p$ of $\Xbar_p$. Proposition \ref{asymptAJ}(b) then yields that $\alpha_2(z,p) \in \cc^{h-r}$ is a lift of the Abel-Jacobi image $\int_P^z$ in $\Jac(\tilde{\Xbar}_p)$. Note that (\ref{limitform}) remains unchanged upon replacing the section $P$ by any other section. Hence we may assume without loss of generality that the section $P$ passes through $x$. In this case our claim follows immediately from Corollary \ref{cor:ceeone}. 

Finally we consider condition (c) of Theorem \ref{thm:learextensiondeligne}. We turn our attention to a coordinate neighborhood $U_{e,\eps}$ with center $e \in \Sigma_0$. Our task is then to show that the function
\[ \Delta_\eps^n \setminus D \to \rr \, , \quad t \mapsto \int_{X_t \cap U_{e,\eps}} \chi\,\ceeone(\delta^*_P\bb) \]
has a log singularity along $D$, where $\chi$ denotes the function $\log\|s\|$ without its linear part determined by the Lear extension of $\delta_P^*\bb$ over $U_{e,\eps}\setminus \pi^{-1}D$. We recall that $\eps$ is chosen small enough so that $U_{e,\eps}$ is disjoint from $\overline{\divisor_X s}$. Also we recall that on $U_{e,\eps}$ we have coordinates $(u,v,t_2,\ldots,t_n)$, with equation $uv=t_1$, and projection $\pi $ given by $(u,v,t_2,\ldots,t_n)\mapsto (uv,t_2,\ldots,t_n)$. We will need to develop $\log \|s\|$ on $U_{e,\eps} \setminus \pi^{-1}D$ in these coordinates. As before, this requires developing $(\Im \Omega)^{-1}$ and $\Im \delta_P$ in these coordinates. For $(\Im \Omega)^{-1}$ there is nothing new, and we can work with the conjunction of formulae (\ref{inverse}) and (\ref{Qmatrix}) above (where now $t_1=uv$).

The situation is different for $\Im \delta_P$. By Proposition \ref{asymptAJ}(a) there exist vectors $b_1, b_2 \in \qq^h$ and a holomorphic map $\alpha \colon U_{e,\eps} \to \cc^h$ such that
\[ \delta_P = Ab_1 \frac{\log u}{2\pi i} + Ab_2 \frac{\log v}{2\pi i} + \alpha \]
on $U_{e,\eps} \setminus \pi^{-1}D$. This gives the identity
\[ -2\pi\, \Im \delta_P = Ab_1 \log|u| + Ab_2 \log|v| + a \]
with $a = -2\pi \Im \alpha$. Combining with Proposition \ref{normbiext} we find 
\[ \begin{split} & \log \|s\| = \log|f| \\ & + \left( Ab_1 \log|u| + Ab_2 \log|v| + a \right)^t
\left( A \log|t_1| + B \right)^{-1} \left( Ab_1 \log|u| + Ab_2 \log|v| + a \right) \end{split} \]
on $U_{e,\eps}\setminus \pi^{-1}D$, where $f$ is a meromorphic function on $U_{e,\eps}$. As by assumption $\overline{\divisor_X s}$ lies away from $U_{e,\eps}$ we can write $\log|f|=c_1\log|u|+c_2\log|v|+\log|\nu|$ with $c_1,c_2 \in \zz$ and with $\nu$ a holomorphic unit on $U_{e,\eps}$. Using (\ref{inverse}) and (\ref{Qmatrix}) and the relation $\log|v|=\log|t_1|-\log|u|$ we find unique bounded continuous functions $\sigma$, $\gamma_1,\gamma_2,\gamma_3$ on $U_{e,\eps}$ and $d_1, d_2, d_3 \in \qq$ such that 
\begin{equation} \label{expandlogs} \begin{split} \log\|s\| =  
 d_1 & \log|u| + d_2 \log|v| + d_3 \frac{ \log|u|\log|v|}{\log|t_1|} \\ & + \log|\nu|+\sigma + \frac{1}{\log|t_1|}\gamma_1+  
\frac{\log|u|}{\log|t_1|} \gamma_2 + \left(\frac{\log|u|}{\log|t_1|}\right)^2 \gamma_3  \end{split} 
\end{equation}
on $U_{e,\eps} \setminus \pi^{-1}D$. 

We put
\begin{equation} \label{definerho}
\rho(u,v) = d_1 \log|u| + d_2 \log|v| + d_3 \frac{ \log|u|\log|v|}{\log|t_1|}  \, , 
\end{equation}
and
\begin{equation} \label{expandtau} \tau(u,v,t_2,\ldots,t_n) = \log|\nu|+\sigma + \frac{1}{\log|t_1|}\gamma_1+  
\frac{\log|u|}{\log|t_1|} \gamma_2 + \left(\frac{\log|u|}{\log|t_1|}\right)^2 \gamma_3  \, . 
\end{equation}
We note that, away from $\pi^{-1}D$,
\begin{equation} \label{deldelbarrho} \begin{split} 
\frac{\deldelbar}{\pi i} \rho(u,v) & =
\frac{\deldelbar}{\pi i} \, d_3 \frac{\log|u|(\log|t_1|-\log|u|)}{\log|t_1|} 
\\ &= -\frac{\deldelbar}{\pi i} \, d_3 \frac{(\log|u|)^2}{\log|t_1|} \\
& \sim_\pi -d_3 \frac{1}{\log|t_1|} \frac{1}{2\pi i} \frac{1}{|u|^2} \d u \, \d  \overline{u}
\, . \end{split}    
\end{equation}
In polar coordinates $u=re^{i\theta}$ this reads 
\begin{equation} \label{polar} \frac{\deldelbar}{\pi i} \rho(u,v) \sim_\pi \frac{2d_3}{\log|t_1|} \frac{1}{2\pi} \d \theta \, \d \log r \, . 
\end{equation} 
We find
\begin{equation} \label{c1explicit} \ceeone(\delta_P^*\bb) = \frac{1}{\pi i} \deldelbar \log \|s\| \sim_\pi \frac{2d_3}{\log|t_1|} \frac{1}{2\pi} \d \theta \, \d \log r + 
\frac{1}{\pi i} \deldelbar \tau 
\end{equation}
on $U_{e,\eps} \setminus \pi^{-1}D$. 

Further we note that
\[ \chi = d_3 \frac{ \log|u|\log|v|}{\log|t_1|} + \tau(u,v,t_2,\ldots,t_n) \, . \]
Our task is thus to show that 
\[ \int_{X_t \cap U_{e,\eps}} \left( d_3 \frac{ \log|u|\log|v|}{\log|t_1|} + \tau  \right)
\left( \frac{2d_3}{\log|t_1|} \frac{1}{2\pi} \d \theta \, \d \log r + 
\frac{1}{\pi i} \deldelbar \tau \right) \]
has a log singularity along $D$. First we note that $\tau$ extends as a bounded continuous function over $U_{e,\eps} \setminus \{(0,0)\}$. 
Moreover, if we write
\[ b_i= {\genfrac{(}{)}{0pt}{}{b_{i1}}{b_{i2}}} \, , \quad a =  
{\genfrac{(}{)}{0pt}{}{a_1}{a_2}} \]
for $i=1, 2$, with $b_{i1} \in \qq^r$, $b_{i2} \in \qq^{h-r}$, $a_1 \colon U_{e,\eps} \to \rr^r$, $a_2 \colon U_{e,\eps} \to \rr^{h-r}$, then explicitly we have 
\[ \sigma = a_2^t B_{22}^{-1} a_2 + 2a_1^t b_{21} -2a_2^t B_{12}B_{22}^{-1} b_{21} + b_{21}^tA'\delta A' b_{21} \, . \]
Similarly to the case $e \notin \Sigma_0$ considered above, we have
\[  \deldelbar \, \sigma \, \sim_\pi  \deldelbar \, a_2^tB_{22}^{-1} a_2 \, . \]
We then obtain from (\ref{expandtau}) that, up to the addition of terms that converge uniformly to zero as $t_1 \to 0$, the form $\frac{1}{\pi i}\deldelbar \tau$ extends as a bounded continuous current over $U_{e,\eps}$. 
A term-by-term analysis using (\ref{expandtau}) reveals that each of the three fiber integrals
\[ \int_{X_t \cap U_{e,\eps}}  d_3 \frac{ \log|u|\log|v|}{\log|t_1|} \frac{1}{\pi i} \deldelbar \tau \, , \quad 
\int_{X_t \cap U_{e,\eps}} \tau \frac{2d_3}{\log|t_1|} \frac{1}{2\pi} \d \theta \, \d \log r  \, , \quad \int_{X_t \cap U_{e,\eps}} \tau \frac{1}{\pi i} \deldelbar \tau \]
extends as a continuous function over $\Delta_\eps^n$.

We are thus left to show that
\[ \int_{X_t \cap U_{e,\eps}}  d_3 \frac{\log|u|\log|v|}{\log|t_1|}   \frac{2d_3}{\log|t_1|} \frac{1}{2\pi} \d \theta \, \d \log r 
 \]
acquires a log singularity along $D$. Note that the fiber integral, viewed as a function on $\Delta_\eps^n \setminus D$, depends only on the $t_1$-coordinate, so that we can proceed directly to the case where $n=1$, where we write $t_1=t$. The whole proposition is proved once we have shown that with $n=1$, the asymptotic estimate
\begin{equation} \label{aimedfor}
\int_{X_t \cap U_{e,\eps}}  d_3 \frac{\log|u|\log|v|}{\log|t_1|}   \frac{2d_3}{\log|t_1|} \frac{1}{2\pi} \d \theta \, \d \log r \sim  -\eta(\Gamma,e)\log|t| 
\end{equation}
holds as $t \to 0$. Indeed, then condition (c) from Theorem \ref{thm:learextensiondeligne} is verified, and moreover we find that the order of the log singularity of the function
\[ \sum_{e \in \Sigma_0}  \int_{X_t \cap U_{e,\eps}} \chi \, \ceeone(\delta_P^*\bb) \]
along the origin of $\Delta_\eps$ is equal to $ - \sum_{e \in \Sigma_0} \eta(\Gamma,e) = -\eta(\Gamma)$.
This gives formula (\ref{learextformula}) by the second part of Theorem \ref{thm:learextensiondeligne}.

Assume from now on therefore that $n=1$. We explicitly compute
\begin{equation} \label{elementary} \begin{split} \int_{X_t \cap U_{e,\eps}} \frac{\log|u|\log|v|}{\log|t|}   & \frac{2}{\log|t|} \frac{1}{2\pi} \d \theta \, \d \log r \\
& = \frac{2}{(\log|t|)^2} \int_{r=|t|\epsilon^{-1/2}}^{r=\epsilon^{1/2}} \log r (\log|t|-\log r) \,\d \log r \\
& = 2 \left( -\frac{1}{12}\frac{ (\log \epsilon)^3}{(\log|t|)^2} + \frac{1}{4}\frac{ (\log \epsilon)^2}{\log|t|} - \frac{1}{6}\log|t| \right)\\
& \sim -\frac{1}{3}\log|t| \, . 
\end{split} 
\end{equation}
Let $G$ be the labelled dual graph of the central fiber $\Xbar_0$ of $\Xbar \to \Delta_\eps$ (cf. Section \ref{prelimsemistable}). Recall that we are assuming that $\Xbar$ is smooth; this implies that each edge $e \in \E(G)$ has label $t$, and hence the corresponding edge of $\Gamma$ is isometric to the unit interval.
In particular we find that $\eta(\Gamma,e)=\frac{1}{3}F(e)^2$. By (\ref{elementary}) we find (\ref{aimedfor}) once we prove that $d_3^2=F(e)^2$. This will be our objective in the remainder of the proof.

 Let $\Zbar=\Xbar \times_{\Delta_\eps} \Xbar$. The first projection $\pi_1 \colon \Zbar \to \Xbar$ has two tautological sections, one called $\tilde{P}$ induced by the given section $P$, and the diagonal section $Q$. The section $\tilde{P}$ induced by $P$ passes through the smooth locus of $\pi_1$. However, for each double point $e \in \Sigma_0$ the image $(e,e)$ under the section $Q$ lies in the singular locus of $\pi_1$. This prevents us from an immediate application of Proposition \ref{Leardelta}, and we need a modification of $\Zbar$ to separate the two sections.

Knudsen's specialization theorem \cite[Theorem~2.4]{kn} gives a modification $\beta \colon \Zbar' \to \Zbar$ which is a contraction, and moreover admits liftings $\tilde{P}'$, $Q'$ of the sections $\tilde{P}$, $Q$ such that $\tilde{P}'$ and $Q'$ are disjoint, and pass through the smooth locus of the composed map $\xi=\pi_1 \beta$. Let $e \in \Sigma_0$. We will need to know the precise structure of the labelled dual graph of the fiber of $\xi$ at $e$.  Let $G_e$ be the labelled dual graph of the fiber of $\pi_1$ above $e$. Note that $G_e$ has the same underlying graph as $G$, but has the label $t$ on each edge of $G$ replaced by the monomial $uv$. We let $w, z$ be the endpoints of the tautological edge $e$ of $G_e$; say that the branch of $\Xbar_0$ through $e$ given by $u=0$ corresponds to the endpoint $w$ of $e$, and the branch given by $v=0$ corresponds to the endpoint $z$ of $e$.  

Let $\tilde{G}_e$ denote the labelled dual graph of $\xi$ over $e$. The proof of Theorem 2.4 in \cite{kn} shows that $\beta$ only modifies $\Zbar$ over points where $Q$ meets the singular locus of $\pi_1$ (called ``Case I'' in \cite{kn}) or over points where $\tilde{P}$ and $Q$ meet (``Case II''). In the fiber of $\Zbar$ over $e$ ``Case I'' applies, the modified point being $(e,e)$. The explicit calculation in ``Case I'' of the proof of Theorem~2.4 in \cite{kn} shows then that the inverse image of $(e,e)$ under $\beta$ is a rational component $y$ of the fiber of $\xi$ above $e$, and that $\tilde{G}_e$ is obtained from $G_e$ by replacing the tautological edge $e$ (carrying label $uv$) of $G_e$ by two edges, one, which we call $e_w$, with label $u$ attached to it, and with endpoints $w$ and $y$; and one, which we call $e_z$, with label $v$ attached to it, and with endpoints $y$ and $z$. 

In graphical terms, in order to obtain $\tilde{G}_e$ from $G_e$ one has to replace the piece
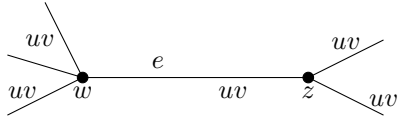
\begin{center}
\begin{tikzpicture}
\draw node[circle, draw, fill=black, 
inner sep=0pt, minimum width=4pt]{} (0,0) -- (3,0) node[circle, draw, fill=black, inner sep=0pt, minimum width=4pt]{};
\draw (-0.5,1) -- (0,0);
\draw (-1,-0.5) -- (0,0);
\draw (-1,0.3) -- (0,0);
\node [below] at (0,0) {$w$};
\node [below] at (2,0) {$uv$};
\node [above] at (1,0) {$e$};
\node [below] at (3,0) {$z$};
\node [left] at (-0.25,0.5) {$uv$};
\node [above] at (-0.8,-0.4) {$uv$};
\node [above] at (3.5,0.25) {$uv$};
\node [above] at (4,-0.5) {$uv$};
\draw (3,0) -- (4,-0.5);
\draw (3,0) -- (4,0.5);
\end{tikzpicture}
\end{center}
of $G_e$ by the piece 
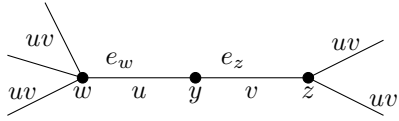
\begin{center}
\begin{tikzpicture}
\draw node[circle, draw, fill=black, 
inner sep=0pt, minimum width=4pt]{} (0,0) -- (1.5,0) node[circle, draw, fill=black, inner sep=0pt, minimum width=4pt]{};
\draw (1.5,0) -- (3,0) node[circle, draw, fill=black, inner sep=0pt, minimum width=4pt]{};
\draw (-0.5,1) -- (0,0);
\draw (-1,-0.5) -- (0,0);
\draw (-1,0.3) -- (0,0);
\node [below] at (0,0) {$w$};
\node [below] at (0.75,0) {$u$};
\node [below] at (2.25,0) {$v$};
\node [below] at (1.5,0) {$y$};
\node [above] at (0.5,0) {$e_w$};
\node [above] at (2,0) {$e_z$};
\node [below] at (3,0) {$z$};
\node [left] at (-0.25,0.5) {$uv$};
\node [above] at (-0.8,-0.4) {$uv$};
\node [above] at (3.5,0.25) {$uv$};
\node [above] at (4,-0.5) {$uv$};
\draw (3,0) -- (4,-0.5);
\draw (3,0) -- (4,0.5);
\end{tikzpicture} 
\end{center}
and leave everything else untouched. Note that the section $Q'$ of $\xi$ specializes onto the new vertex $y$. Let $x$ be the vertex of $\tilde{G}_e$ such that the section $\tilde{P}'$ of $\xi$ specializes onto $x$.

Let $a, b$ be two formal variables and let $\tilde{G}^{a,b}_e$ be the labelled graph obtained from $\tilde{G}_e$ by replacing the label $u$ by $a$, the label $v$ by $b$, and each label $uv$ by $a+b$. 
Let $r(x,y)(a,b) \in \qq(a,b)$ be the effective resistance between the vertices $x$ and $y$ on $\tilde{G}^{a,b}_e$. Then $r(x,y)(a,b)$ is a rational homogeneous function of weight one in $a, b$. Let $r(p,q) \in \qq$  denote the effective resistance between vertices $p, q$ of $G$. Then we have that $r(x,y)(a,0)=r(x,z)a$ and $r(x,y)(0,b)=r(x,w)b$. Imitating the derivation of formula (\ref{explicitreff}) we find the explicit formula
\begin{equation} \label{explicitresistance} r(x,y)(a,b) = r(x,z)a + r(x,w)b + F(e)\frac{ab}{a+b}  
\end{equation}
for $r(x,y)(a,b)$. From the construction of $\beta \colon \Zbar' \to \Zbar$ we deduce that $\delta_P^*\bb$ coincides with the restriction to $X$ of the line bundle $\pair{\oo(\tilde{P}'-Q'),\oo(\tilde{P}'-Q')}^{\otimes -1}$ on $\Xbar$. As we are only interested in $d_3$, from (\ref{expandlogs}) it follows that we may as well replace $s$ by another local section of $\delta_P^*\bb$. Using again \cite[Proposition 9.1.11]{liu} we change  the rational section $s$ in such a way that  $s$, when seen as a rational section of the line bundle $\pair{\oo(\tilde{P}'-Q'),\oo(\tilde{P}'-Q')}^{\otimes -1}$ on $\Xbar$, is generating on $U_{e,\eps}$.  

Recall that on $U_{e,\eps} \setminus \pi^{-1}D$, the function $\tau(u,v,t_2,\ldots,t_n)$ is bounded. 
From (\ref{expandlogs}) we then obtain that $\rho(u,v) = d_1 \log|u| + d_2 \log|v| + d_3 \frac{ \log|u|\log|v|}{\log|t|}  $ can be characterized as the unique rational linear combination of $\log|u|$, $\log|v|$ and $\log|u|\log|v|/\log|t|$ such that on $U_{e,\eps}\setminus \pi^{-1}D$ an equality
\begin{equation} \label{plusbounded} \log \|s\|(u,v) = \rho(u,v) + \, \textrm{bounded function} 
\end{equation}
holds. For each pair of positive integers $m, n$, let $\bar{f}_{m,n} \colon \Delta_\eps \to U_{e,\eps} \subset \Xbar$ be a test curve given by sending $q \mapsto (q^m,q^n)$. 
We see that for each $m, n$ the estimate
\[ \left(\bar{f}_{m,n}^*\log\|s\|\right)(q) = \rho(q^m,q^n) + O(1)
= (d_1m + d_2n + d_3 \frac{mn}{m+n})\log|q| + O(1) \]
holds. On the other hand, combining Propositions  \ref{learchar} and \ref{Leardelta} it is not difficult to see that
\[ \left(\bar{f}_{m,n}^*\log\|s\|\right)(q) = -r(x,y)(m,n) \log|q| + O(1) \]
holds. We deduce that for each pair of positive integers $m, n$ the identity
\[ d_1m + d_2n + d_3 \frac{mn}{m+n} = -r(x,y)(m,n) \]
holds. Combining with (\ref{explicitresistance}) we obtain $d_3=-F(e)$, which is what we wanted. 
\end{proof}
\begin{remark} Let $\pi \colon \Xbar \to \Delta$ with $\Xbar$ smooth be a nodal curve over the unit disc, assumed to be smooth over $\Delta^*$. Write $X =\pi^{-1}\Delta^*$. The above proof shows that the family of Arakelov measures $\mu_{\Ar,t}$ on the fibers $X_t$ has a natural limit measure $\mu_0$ on the special fiber $\Xbar_0$. The measure $\mu_0$ can be described as follows: let $\nu \colon \tilde{\Xbar}_0 \to \Xbar_0$ be the normalization of $\Xbar_0$, and let $M$ be the measure on $\tilde{\Xbar}_0$ which on a connected component $x$ is $q(x)$ times the Arakelov volume form on $x$ if $x$ has positive geometric genus $q(x)$, and the zero measure if $x$ has genus zero. For each node $e$ of $\Xbar_0$ let $F(e)$ be the associated Foster coefficient. From (\ref{c1explicit}) and the identity $d_3=-F(e)$ we deduce that the area of the ``collar'' around the node $e$ with respect to $c_1(\delta_P^*\bb)$ tends to $2F(e)$ as $t \to 0$. We find that 
\[ \mu_0 = \frac{1}{h} \left( \nu_*M + \sum_{e \in \E(G)} F(e) \,\delta_e \right) \, . \]
Let $G$ be the dual graph of $\Xbar_0$ and $(\Gamma,K)$ the associated polarized metrized graph. Then we note that to $\mu_0$ naturally corresponds the measure
\[ \frac{1}{h} \left( \sum_{x \in V(G)} q(x)\,\delta_x + \sum_{e \in \E(G)} F(e)\, \d y(e) \right) \]
on $\Gamma$. Not surprisingly, this is precisely Zhang's admissible measure $\mu$ as defined in (\ref{measure}). Indeed, note that $K(x)=v(x)-2+2\,q(x)$ for all $x \in \V(G)$.
\end{remark}
\begin{proof}[Proof of Theorem \ref{leardeltasq}] 
We first show that $\pair{\delta_P^*\bb,\delta_P^*\bb}$ has a Lear extension over $\Ybar$. It suffices to consider the case that $\Ybar = \Delta^n$ and $D \subset \Ybar$ is the smooth irreducible divisor given by $t_1=0$. The truth value of the statement is unaffected by modifications of the total space $\Xbar$ in the inverse image of $D$. By \cite[Lemma 3.2]{alt} there exists a projective modification $\Xbar' \to \Xbar$ such that the composed morphism $\Xbar' \to \Ybar$ is a nodal curve, smooth over $\Ybar \setminus D$, and with $\Sing(\Xbar')$ either empty or of codimension at least three in $\Xbar'$. The codimension three condition implies that $D$ is reducible (cf. Remark 3.5 of \cite{alt}), so, as  $D$ is irreducible, we conclude that we can reduce to the case that $\Xbar$ is smooth. Then the statement follows from Proposition \ref{maintechnical}. 

We turn next to the proof of the formula
\[ \overline{\pair{\delta_P^*\bb,\delta_P^*\bb}} = 4h \psi + \kappa_1 - 
\left(4hg_\mu(x,x) + \varepsilon \right) [0]  \]
in the case that $\Ybar$ is the unit disc, and $P$ passes through $\Sm(\pi)$. Also in this case we can reduce to the case that $\Xbar$ is smooth. Indeed, upon passing to the minimal desingularization $\tilde{\Xbar}\to \Xbar$ of $\Xbar$, the left hand side of the equality does not change, and neither does the right hand side: the line bundles $\psi$ and $\kappa_1$ remain unchanged, and by Proposition \ref{minimaldesing} the expression $4hg_\mu(x,x) + \varepsilon$ remains unchanged as well. In the case that $\Xbar$ is itself smooth, the claimed formula follows upon combining the formulae in Propositions \ref{selfdelignepairingLear} and \ref{maintechnical}.
\end{proof}

\section{Proof of Theorems \ref{main_first} and \ref{main_second}} \label{proofmainresult}

We can now give the proof of our main results, Theorems \ref{main_first} and \ref{main_second}. Localizing in the analytic topology we may assume from the outset that two sections $P$, $Q$ of $\Xbar \to \Ybar$ are given.
Proposition \ref{kappalear} and Theorem \ref{leardeltasq} imply that the Lear extensions of $\kappa^*\bb$ and of $\pair{\delta_P^*\bb,\delta_P^*\bb}$ exist. From Proposition \ref{omega_in_terms_of_biext}
we then deduce that the Lear extension of $\pair{\opar,\omar}$ exists, and that
\begin{equation} \label{keyformula} 4h^2 \overline{\pair{\opar,\omar}} = \overline{\pair{\delta_P^*\bb,\delta_P^*\bb}} + \overline{\kappa^* \bb} \, .   
\end{equation}
Invoking Proposition \ref{deltasq} we next obtain that the Lear extension of $\omsq$ exists, and that
\begin{equation} \label{learomsq} \overline{\omsq} = \overline{\pair{\delta_P^*\bb,\delta_P^*\bb}} - 4h   \overline{\pair{\opar,\omar} } \, . 
\end{equation}
Proposition \ref{Leardelta} states that the Lear extension of $\delta^*\bb$ exists.  Combining this with Proposition \ref{diagonal_metric} and the existence of $\overline{\pair{\opar,\omar}}$ and $\overline{\pair{\oo(Q)_\Ar,\omar}}$ established above gives that the Lear extension of $ \pair{\opar,\oo(Q)_\Ar}$ exists. More precisely we find the equality
\begin{equation} \label{Leardiagonal} 2 \overline{\pair{\opar,\oo(Q)_\Ar}} = \overline{\delta^*\bb} - \overline{\pair{\opar,\omar} } - \overline{\pair{\oo(Q)_\Ar,\omar} }    
\end{equation}
on $\Ybar$. This settles Theorem \ref{main_first}. 

Specializing to the case when the base is the unit disc $\Delta$, we have the following more precise formulae. From now on we assume that the sections $P, Q$ pass through $\Sm(\pi)$. Let $G$ be the dual graph of $\Xbar_0$, and let $(\Gamma,K)$ be the associated polarized metrized graph. Let $g_\mu$ be the admissible Green's function of $(\Gamma,K)$ and let $\bar{g}$ be the Green's function on $G$ introduced in Section \ref{prelimgraphs} using the pseudo-inverse of the discrete Laplacian on $G$. Then for any two divisors $\ee, \ff$ of degree zero on $G$ we obtain from (\ref{greenandresistance}) and (\ref{gZhandRes}) the identities 
\begin{equation} \label{nice}  \bar{g}(\ee,\ff) = -\frac{1}{2}r(\ee,\ff)  = g_\mu(\ee,\ff) \, . 
\end{equation}
Let $\vareps$ be the epsilon-invariant of $G$. Let $x, y \in V(G)$ be the components of $\Xbar_0$ to which $P, Q$ specialize. From Proposition \ref{kappalear} combined with (\ref{nice}) we obtain that
\begin{equation} \label{recallkappalear} \overline{\kappa^*\bb} = 4h(h-1)\psi-\kappa_1-g_\mu((2h-2)x-K,(2h-2)x-K)[0] \, , 
\end{equation}
and from Theorem \ref{leardeltasq} we recall that
\begin{equation} \label{recalldeltasq} \overline{\pair{\delta_P^*\bb,\delta_P^*\bb}} = 4h \psi + \kappa_1 - 
\left(4h\,g_\mu(x,x) + \varepsilon \right) [0] \, .  
\end{equation}
Together with (\ref{keyformula}) we then find
\begin{equation} \label{fourhsq} \begin{split}
4h^2 \overline{\pair{\opar,\omar}} & = \overline{\pair{\delta_P^*\bb,\delta_P^*\bb}} + \overline{\kappa^* \bb}  \\ & = 4h^2 \,\psi - (4hg_\mu(x,x) + \varepsilon + g_\mu((2h-2)x-K,(2h-2)x-K))[0] \, .  \end{split}  
\end{equation}
From Proposition \ref{epsalt} we recall that
\[ \varepsilon = 4(h-1)(g_\mu(x,x)+g_\mu(K,x)) - g_\mu(K,K) \, . \]
We obtain from this that
\[ 4h\,g_\mu(x,x) + \varepsilon + g_\mu((2h-2)x-K,(2h-2)x-K) = 4h^2 \,g_\mu(x,x) \, . \]
Combining with (\ref{fourhsq}) we find that
\begin{equation} \label{eqn:omsq}
 \overline{\pair{\opar,\omar}} = \psi - g_\mu(x,x)[0] \, , 
\end{equation} 
which settles part (b) of Theorem \ref{main_second}. Combining (\ref{learomsq}), (\ref{recalldeltasq}) and (\ref{eqn:omsq}) we then obtain
\[ \overline{\omsq} = \kappa_1 -  \varepsilon [0]   \]
which settles part (a).
Write $\psi_1=\pair{\oo(P),\omega}$ and $\psi_2=\pair{\oo(Q),\omega}$.
From Proposition \ref{Leardelta} we obtain that
\begin{equation} \label{recallLeardelta}
\overline{\delta^*\bb} = 2 \pair{\oo(P),\oo(Q)}+\psi_1 + \psi_2  - g_\mu(x-y,x-y)[0] \, .
\end{equation}
Combining (\ref{Leardiagonal}), (\ref{eqn:omsq}) and (\ref{recallLeardelta}) gives that
\[ \overline{\pair{\opar,\oo(Q)_\Ar}} = \pair{\oo(P),\oo(Q)} + g_\mu(x,y)[0] \, , \]
which settles part (c). This finishes the proof of Theorem \ref{main_second}.

\section{Zhang's admissible Deligne pairing} \label{sec:admpairing}

The identities in Theorem \ref{main_second} can be conveniently rephrased in terms of Zhang's admissible Deligne pairing from \cite{zh}. The purpose of this section is to explain this.  Let $\pi \colon \Xbar \to \Delta$ be a nodal curve of positive genus, such that $\pi$ is smooth over $\Delta^*$. Let $\omega$ be the relative dualizing sheaf of $\pi$. Let $G$ be the dual graph of the special fiber $\Xbar_0$, and let $\Gamma$ be the associated metrized graph, with vertex set $V=V(G)$. We let $C(\Gamma)$ and $\Delta$ be the function space and Laplacian on $\Gamma$ as in Section \ref{harmonic}. 

A \emph{compactified divisor} on $\Gamma$ is to be any element of $\rr^{V} \oplus C(\Gamma)$. If $D+f$ and $E+g$ are compactified divisors on $\Gamma$, one defines their intersection product to be the real number
\[ (D+f,E+g) = g(D) + f(E) - \int_\Gamma g \, \Delta \, f  \, . \]
The intersection product on $\rr^{V} \oplus C(\Gamma)$ is symmetric and bilinear. The \emph{curvature form} of a compactified divisor $D+f$ is to be the current $\delta_D - \Delta f$ in $C(\Gamma)^*$.

Let $P(\Xbar)$ be the (additively written) group of line bundles on $\Xbar$. We recall from Section \ref{learI} that we have a canonical specialization map $R \colon P(\Xbar) \to \rr^V$. We put $K=R(\omega)$. We note that $\deg K=2h-2$, where $h>0$ is the genus of $\pi$.
An element of the group $P(\Xbar) \oplus C(\Gamma)$ is called a \emph{compactified line bundle} on $\Xbar$. The curvature form of a compactified line bundle $L+f$ is by definition the curvature form of the compactified divisor $R(L)+f$.
 For compactified line bundles $L+f$ and $M+g$ we let $\pair{L+f,M+g}$ be the line bundle on $\Delta$ given by 
\[ \pair{L+f,M+g} = \pair{L,M} + (R(L)+f,R(M)+g)[0] \, , \] 
where $\pair{L,M}$ is the usual Deligne pairing of the ordinary line bundles $L$ and $M$.

Let $g_\mu$ be Zhang's admissible Green's function on $(\Gamma,K)$. 
An \emph{admissible line bundle} on $\Xbar$ is by definition a compactified line bundle on $\Xbar$ of the shape $L + c + g_\mu(R(L),\textit{-})$ for some $c \in \rr$ and $L \in P(\Xbar)$. The terminology ``admissible'' is explained by the fact that among the compactified line bundles, the admissible ones are precisely those whose curvature form is a multiple of Zhang's admissible measure $\mu$ (cf. Section \ref{sec:arakmetric}). 

Let $P \colon \Delta \to \Xbar$ be a section of $\pi$ with image contained in the smooth locus of $\pi$. Then one has the special admissible line bundles 
\[ \oo(P)_a = \oo(P) + g_\mu(R(P),\textit{-}) \quad \textrm{and} \quad
\omega_a = \omega + c(\Gamma,K) + g_\mu(K,\textit{-}) = \omega - g_\mu(\textit{-},\textit{-}) \]
on $\Xbar$, where $c(\Gamma,K) \in \rr$ is as in Proposition \ref{defc}.

It follows easily from the above that for each pair of sections $P, Q$ of $\Xbar \to \Delta$ with image contained in $\Sm(\pi)$ we have
\begin{equation} \label{Iisom}
\pair{\oo(P)_a,\oo(Q)_a} = \pair{\oo(P),\oo(Q)} + g_\mu(x,y)[0] \, , 
\end{equation}
where $x=R(P)$, $y=R(Q)$. Let $\vareps$ denote the epsilon-invariant of $(G,K)$. By \cite[Theorem 4.2]{zh} resp.\ \cite[Theorem 4.4]{zh} there exist canonical isomorphisms
\begin{equation} \label{IIisom}
\pair{ \oo(P)_a,\omega_a(P_a)} \isom \oo \quad \textrm{and} \quad
\pair{ \omega_a, \omega_a } \isom \pair{\omega,\omega} - \vareps[0]
\end{equation}
of line bundles on $\Delta$. 

The above set-up for a generically smooth nodal curve over a disc extends in a straightforward manner to the setting of a generically smooth nodal curve $\Xbar \to \Ybar$ over a smooth complex algebraic curve $\Ybar$. We omit the details. In this global setting, one readily sees, using the local equalities (\ref{Iisom}) and (\ref{IIisom}), that Theorem~\ref{main_second} implies the following result phrased in terms of the admissible Deligne pairing.
\begin{thm}  \label{lear=adm}
Let $\pi \colon \Xbar \to \Ybar$ be a nodal curve of positive genus over a complex algebraic curve $\Ybar$. Assume that $\pi$ is smooth over the dense open subset $Y \subset \Ybar$. Write $X=\pi^{-1}Y$ and denote by $\omar$ the line bundle $\omega$ on $X$ equipped with the Arakelov metric. Let $\omega_a$ denote Zhang's admissible relative dualizing sheaf of $\pi$. Then the equality
\[ \overline{\pair{\omar,\omar}} = \pair{\omega_\a,\omega_\a}  \]
holds. Suppose that $\pi$ has a section $P$, resp.\ two sections $P, Q$, with image contained in the smooth locus $\Sm(\pi)$. Denote by $\opar$ and $\oqar$ the line bundles $\oo(P)$ and $\oo(Q)$ on $X$ equipped with the Arakelov metric. Then we further have the equalities
\[ \overline{\pair{ \opar, \omar }} = \pair{\oo(P)_\a,\omega_\a} \quad \textrm{and} \quad
 \overline{\pair{\opar,\oqar}} = \pair{\oo(P)_\a,\oo(Q)_\a} \, .
\]
Here $\oo(P)_\a$ and $\oo(Q)_\a$ denote the canonical admissible line bundles on $\Xbar$ associated to the sections $P$ resp.\ $Q$.
\end{thm}

\section{Proofs of Theorems \ref{delta} and \ref{arakintro}} \label{sec:corollaries}

In this final section we derive Theorems \ref{delta} and \ref{arakintro} from Theorem~\ref{main_second}. 
\begin{proof}[Proof of Theorem \ref{delta}]
Let $(\omega_1,\ldots,\omega_h)_t$ be the family of normalized bases of the Hodge bundle $\pi_*\omega_{X/\Delta^*}$ determined by some symplectic frame of $R^1\pi_*\zz_X$. Write $\sigma=\omega_1 \wedge \ldots \wedge \omega_h$ and let $\lambda_H$ denote the determinant of $\pi_*\omega_{X/\Delta^*}$ over $\Delta^*$ equipped with the norm derived from the $L^2$-metric (\ref{defineinnerproduct}) on $\pi_*\omega_{X/\Delta^*}$. By equation (\ref{normhodge}) we have 
\[ \log \| \sigma \|_H = \frac{1}{2} \log \det \mathrm{Im} \,\Omega(t) \]
for all $t \in \Delta^*$. This equality together with the growth rate 
\[ \det \Im \Omega(t) \sim -c \log |t| \]
as $t \to 0$ that follows from the Nilpotent Orbit Theorem (cf. Proposition \ref{asymptperiod}) shows that $\sigma$ extends as a generating section of the canonical Mumford extension \cite{hi} of $\lambda_H$ over $\Delta$. Following \cite[p.~225]{fc} the Mumford extension of $\lambda_H$ is equal to $\lambda_1=\det \pi_* \omega$, and so the section $\sigma$ extends as a generating section of $\lambda_1$ over $\Delta$. Now the Mumford isomorphism (\ref{mumfordnorm}) gives a globally trivializing section $\mu$ of the hermitian line bundle $\lambda_{H}^{\otimes 12} \otimes \pair{\omar,\omar}^{\otimes -1}$ over $\Delta^*$. We recall that by definition the Faltings delta-invariant satisfies
\[ \delta_F=-\log\|\mu\|  \, . \]
Let $\tau$ be a local generating section of $\kappa_1=\pair{\omega,\omega}$ around $0$. Then
from part (a) in Theorem \ref{main_second} combined with Proposition \ref{learchar} we obtain the asymptotics 
\[ \log \|\tau\| \sim - \varepsilon \log |t| \]
as $t \to 0$. Let $\delta$ be the volume of the weighted dual graph of $\Xbar_0$. As $\Xbar \to \Delta$ is assumed to be semistable, we obtain from (\ref{mumfordnorm}) that the section $\mu$ extends as a global trivializing section of the line bundle $12\lambda_1 - \kappa_1 - \delta [0]$ over $\Delta$. We derive from this that $t^{-\delta} \mu$ is a local generating section of $12 \lambda_1 - \kappa_1$. Hence $t^{-\delta} \mu$ differs by a holomorphic unit from the local generating section 
$\sigma^{\otimes 12} \otimes \tau^{\otimes -1}$ of $12\lambda_1-\kappa_1$. We see that
\[ \begin{split}
\delta \log|t| + \delta_F & = -\log\|t^{-\delta}\mu\| \sim -12 \log\|\sigma\|_H + \log \|\tau\| \\
 & \sim -6 \log \det \mathrm{Im} \,\Omega(t) - \vareps \log|t|
 \, , \end{split} \]
and Theorem \ref{delta} follows.
\end{proof}
\begin{proof}[Proof of Theorem \ref{arakintro}]
Part (b) of Theorem \ref{main_second} states that
\[ \overline{\pair{ \opar, \omega_\Ar }} = \psi - g_\mu(x,x)[0] \, . \]
By assumption we have that $P^* \d z$ is a local generating section of the line  bundle $\psi=P^*\omega=\pair{\oo(P),\omega}$ around the origin. We recall that we have a canonical isometry $P^*\omar \isom \pair{\opar,\omar}$ over $\Delta^*$. Applying Proposition \ref{learchar} we then find the asymptotics
\[ \log \| P^*\d z \|_{\Ar,t} \sim -g_\mu(x,x)\log|t|  \]
as $t \to 0$. This gives part (a) of Theorem \ref{arakintro}.  Part (c) of Theorem \ref{main_second} states that
\[ \overline{\pair{\opar,\oqar}} = \pair{\oo(P),\oo(Q)} + g_\mu(x,y)[0] \, . \] 
As by assumption $P, Q$ have empty intersection above the origin, we have that $\pair{1_P,1_Q}$ is a local generating section of the Deligne pairing $\pair{\oo(P),\oo(Q)}$ around the origin. Applying again Proposition \ref{learchar} we obtain the asymptotics
\[ g_{\Ar,t}(P,Q) = \log \| \pair{1_P,1_Q}\|_t \sim  g_\mu(x,y) \log |t|
\]
as $t \to 0$, and this yields part (b) of Theorem \ref{arakintro}. 
\end{proof}

\vspace{0.5cm}

\noindent Address of the author:\\ \\
Mathematical Institute  \\
Leiden University \\
PO Box 9512  \\
2300 RA Leiden  \\
The Netherlands  \\ \\
Email: \verb+rdejong@math.leidenuniv.nl+

\end{document}